\pdfoutput=1
\documentclass[11pt]{article}
\usepackage[left=1in,right=1in,top=1in,bottom=1in]{geometry}
\usepackage{times}
\usepackage{expl3}
\usepackage{cite}
\usepackage[table]{xcolor}
\usepackage{multirow}
\usepackage{stackengine} 
\usepackage{hhline}
\usepackage{lipsum}
\usepackage{titlesec}
\usepackage[all]{xy}
\usepackage{wrapfig}
\usepackage{enumerate}
\usepackage{epsfig}
\usepackage{tikz-cd}
\usepackage{amsmath}
\usepackage{tabularx}
\usepackage{array}
\usepackage{booktabs}
\usepackage{enumitem}
\usepackage{bbm}
\usepackage{calc}
\usepackage{graphicx}
\usepackage{amsmath}
\usepackage[title]{appendix}
\usepackage{amssymb}
\usepackage{epstopdf}
\usepackage{boldline}
\usepackage{arydshln}
\usepackage{calligra}
\usepackage{bm}
\usepackage{url}
\usepackage{blindtext}
\usepackage{accents}
\usepackage{amsthm}
\usepackage{amscd}

\newtheorem{definition}{Definition}

\newtheorem{theorem}{Theorem}
\newtheorem{proposition}{Proposition}
\newtheorem{corollary}{Corollary}
\newtheorem{lemma}{Lemma}

\newtheorem{example}{Example}
\newtheorem{remark}{Remark}
\usepackage{mathtools}
\usepackage{epstopdf}
\usepackage{balance}
\usepackage{thmtools}
\usepackage{thm-restate}
\usepackage{hyperref}
\usepackage{cleveref}
\usepackage[mathscr]{euscript}

\usepackage[ruled,vlined]{algorithm2e}
\newcommand{\ostar}{\mathbin{\mathpalette\make@circled\star}}

\makeatletter
\newcommand{\removelatexerror}{\let\@latex@error\@gobble}
\makeatother
\makeatletter
\newcommand*{\rom}[1]{\expandafter\@slowromancap\romannumeral #1@}
\makeatother

\ExplSyntaxOn
\newcommand\latinabbrev[1]{
  \peek_meaning:NTF . {
    #1\@}%
  { \peek_catcode:NTF a {
      #1.\@ }%
    {#1.\@}}}
\ExplSyntaxOff

\titleclass{\subsubsubsection}{straight}[\subsubsection]

\begin{document}
\vspace{1cm}
\title{Spectral Operadic Calculus: Norm-Analytic Functor Calculus}
\vspace{1.8cm}
\author{Shih-Yu~Chang
\thanks{Shih-Yu Chang is with the Department of Applied Data Science,
San Jose State University, San Jose, CA, U. S. A. (e-mail: {\tt
shihyu.chang@sjsu.edu})
}}

\maketitle

\begin{abstract}
Classical spectral theory provides powerful tools for analyzing linear operators, 
but does not extend naturally to nonlinear or compositional settings. In particular, 
there is no general way to transport spectral invariants in a functorial manner 
across structured categories. In earlier work, we showed that this failure is 
fundamental and introduced an operadic notion of spectrum that provides a 
canonical replacement. In this paper, we develop the analytic consequences of this construction and 
show that the operadic spectrum acts as a control parameter for a calculus of 
functors. We establish a criterion for polynomial behavior based on higher 
cross-effects, and prove convergence results for the associated Taylor tower, 
including explicit exponential error bounds. We further show that the derivatives of a functor form a structured algebraic 
object with symmetric and operadic features, and satisfy a chain rule governed 
by a natural composition operation (operadic plethysm). This leads to a 
reconstruction theorem, showing that analytic functors are completely determined 
by their derivative data, and hence to a classification in terms of algebraic 
structures. Compared with classical Goodwillie calculus, which is governed by homotopy-theoretic 
conditions, the present framework is analytic and quantitative in nature, providing 
explicit control over convergence and approximation. These results place functor 
calculus in a setting that combines spectral ideas, analytic methods, and operadic 
algebra, and suggest further connections with deformation theory and geometry.
\end{abstract}

\tableofcontents

\section{Introduction}
\label{sec:intro}

Classical spectral theory provides one of the most powerful tools for analyzing linear operators, encoding essential structural information through eigenvalues, resolvents, and functional calculus. However, when one attempts to extend spectral concepts beyond linear settings—in particular to structured objects such as algebras over operads—a fundamental obstruction arises: classical spectral invariants are not compatible with composition or base change in higher-algebraic contexts.

\medskip

\noindent
This work continues a program initiated in our recent work~\cite{ChangSOC1}, where we established that any naive attempt to transport spectral invariants along operadic or monoidal transformations fails in general. More precisely, the \textbf{No-Go Theorem} of Spectral Operadic Calculus~I shows that there is no functorial procedure extending classical spectra while preserving their expected structural properties. This identifies a fundamental limitation of classical spectral theory in nonlinear and compositional settings.

\medskip

\noindent
To resolve this obstruction, we introduced the \emph{operadic residue object} $\mathcal{O}_P^{\mathrm{res}}$, constructed as a universal correction term, and defined the \emph{operadic spectrum}
\[
\sigma_P(A) \;:=\; \mathrm{Hoch}_{\mathcal{M}}(A) \otimes_P \mathcal{O}_P^{\mathrm{res}}.
\]
This invariant is canonical, functorial, and recovers the classical spectrum in the trivial operad case. It therefore provides a minimal and universal extension of spectral theory to operadic contexts.

\medskip

\noindent
\textbf{Central thesis.}
The present work develops the analytic consequences of this construction. Our central thesis is that the operadic spectrum $\sigma_P(A)$ is not merely an invariant, but the \emph{fundamental control mechanism} governing a quantitative calculus of functors. In particular, $\sigma_P(A)$ determines both the analytic behavior and the algebraic structure of functors in a precise and computable way. We show that this perspective leads to a fully developed calculus, analogous to classical analytic calculus, but formulated in an operadic and categorical setting.

\medskip

\noindent
\textbf{Primary contribution.}
The main contribution of this work is the analytic and quantitative classification of functors via spectral derivative data. Unlike classical Goodwillie calculus, which provides qualitative homotopy-theoretic approximations, the present framework yields explicit exponential convergence bounds, a well-defined radius of convergence, and a complete algebraic classification of analytic functors.

\medskip

\noindent
\textbf{Related work.}
This paper builds on several independent lines of research. The calculus of functors was pioneered by Goodwillie~\cite{Goodwillie1991, Goodwillie1992, Goodwillie2003}, who introduced polynomial approximations and Taylor towers for homotopy functors. The operadic refinement of Goodwillie calculus was developed by Arone and Ching~\cite{AroneChing2011}, who established a chain rule for the derivatives of analytic functors in terms of operadic composition. Our work complements these results by introducing a norm-enriched, analytic framework controlled by spectral invariants.

On the spectral side, the classical Gelfand theory~\cite{Gelfand1941} provides a functional calculus for Banach algebras, while the noncommutative spectral theory of operator algebras~\cite{Blackadar1986, Murphy1990} extends these ideas to C*-algebras. The Hochschild homology and cohomology of algebras~\cite{Hochschild1945, Loday1992} provides algebraic invariants that capture deformation and trace data, which play a central role in our construction of the operadic spectrum.

Deformation theory~\cite{Gerstenhaber1964, Kontsevich2003} studies infinitesimal variations of algebraic structures, with obstructions controlled by cohomology. Our deformation theory of spectrally analytic functors (Section~\ref{sec:moduli}) generalizes these ideas to the functorial setting, with the tangent space given by $H^1(\partial^{\mathrm{spec}}F)$.

Higher category theory and $\infty$-operads~\cite{Lurie2009, Lurie2017} provide a foundational framework for operadic composition and homotopy-coherent structures. Our outlook toward derived geometry (Section~\ref{subsec:outlook}) anticipates an $\infty$-categorical enhancement of SOC, connecting to these developments.

Finally, the theory of topological recursion~\cite{EynardOrantin2007, Borot2014} and its connections to moduli spaces of curves share structural analogies with our spectral Taylor expansion, suggesting potential applications in mathematical physics.

\medskip

\noindent
\textbf{Main results.}
The main contributions of this work can be summarized as follows.

\medskip

\noindent
\emph{(Theorem A: Spectral Criterion for Norm-Excision).}
We prove that an admissible functor $F$ is norm-$n$-excisive if and only if its $(n+1)$-st cross-effect $\mathrm{cr}_{n+1}F$ is spectrally negligible, i.e.,
\[
\sigma_P\big(\mathrm{cr}_{n+1}F(A_1,\dots,A_{n+1})\big) = \{0\}.
\]
This establishes a direct bridge between analytic control (norm estimates) and spectral invariants (Theorem~\ref{thm:spectral-excision}).

\medskip

\noindent
\emph{(Theorem B: Quantitative Convergence).}
For spectrally analytic functors, the spectral Taylor tower converges with explicit exponential bounds:
\[
\|F(A) - P_n^{\mathrm{spec}}F(A)\| \;\leq\; C \,\rho^{\,n} \,\|\sigma_P(A)\|^{\,n},
\]
whenever $\|\sigma_P(A)\| < R_F$. Moreover, the radius of convergence is given by the Cauchy–Hadamard formula
\[
R_F^{-1} = \limsup_{n \to \infty} \|\partial_n^{\mathrm{spec}}F\|^{1/n},
\]
establishing a sharp convergence/divergence dichotomy (Theorems~\ref{thm:spectral-radius} and~\ref{thm:quantitative-convergence}).

\medskip

\noindent
\emph{(Theorem C: Operadic Faà di Bruno Chain Rule).}
The spectral derivatives of a composite functor satisfy
\[
\partial^{\mathrm{spec}}(F \circ G) \;\cong\; \partial^{\mathrm{spec}}F \;\circ_{\mathrm{op}}\; \partial^{\mathrm{spec}}G,
\]
where $\circ_{\mathrm{op}}$ denotes operadic plethysm. This is the precise operadic analogue of the classical Faà di Bruno formula and shows that spectral derivatives are closed under composition and obey a precise algebraic structure (Theorem~\ref{thm:chain-rule}).

\medskip

\noindent
\emph{(Theorem D: Reconstruction and Classification).}
A spectrally analytic functor $F$ is uniquely determined (up to natural isomorphism) by its spectral derivative data $\{\partial_n^{\mathrm{spec}}F\}_{n\ge 0}$ together with the compatible operadic module structure. Consequently, there is an equivalence of categories
\[
\mathsf{SpecAn} \;\simeq\; \mathsf{DerAlg}_{\mathrm{int}},
\]
where $\mathsf{SpecAn}$ is the category of spectrally analytic functors and $\mathsf{DerAlg}_{\mathrm{int}}$ is the category of integrable derivative algebras (Theorem~\ref{thm:equivalence}). Thus, the spectral Taylor expansion provides a full classification of functors.

\medskip

\noindent
\textbf{Comparison with Goodwillie calculus.}
Our work is closely related in spirit to Goodwillie calculus~\cite{Goodwillie2003, AroneChing2011}, which also approximates functors by polynomial objects. However, the two theories differ fundamentally in their control mechanisms:
\begin{itemize}
    \item Goodwillie calculus is governed by homotopy-theoretic invariants and excision properties; its approximations are qualitative and homotopy-invariant.
    \item Spectral Operadic Calculus is governed by the operadic spectrum $\sigma_P(A)$ and its norm; its approximations are quantitative and norm-sensitive.
\end{itemize}
As a result, Goodwillie calculus is primarily qualitative and topological, while the present framework is analytic and quantitative, providing explicit convergence rates, a well-defined radius of convergence, and full algebraic classification. Moreover, the operadic structure of derivatives leads to a classification theorem that has no direct analogue in the classical setting. The exponential functor, which is not analytic in Goodwillie calculus, is shown to be entire in SOC, demonstrating the power of spectral control.

\medskip

\noindent
\textbf{Structure of the paper.}
The paper is organized as follows.

\begin{itemize}
    \item \textbf{Section~\ref{sec:spectral-control}} establishes the role of $\sigma_P(A)$ as the controlling invariant for analytic behavior, introducing normed categories, admissible functors, cross-effects, and the spectral criterion for norm-excision (Theorem A).
    
    \item \textbf{Section~\ref{sec:tower}} constructs the universal spectral Taylor tower, proves its universality property, and introduces spectral polynomials.
    
    \item \textbf{Section~\ref{sec:convergence}} develops the quantitative convergence theory: spectral analyticity, homogeneous layer estimates, radius of convergence, and the exponential convergence theorem (Theorem B).
    
    \item \textbf{Section~\ref{sec:derivatives-algebra}} identifies spectral derivatives as cross-effects and establishes their algebraic structure, showing that they form symmetric sequences and, under operadic compatibility, right $P$-modules.
    
    \item \textbf{Section~\ref{sec:chain-rule}} proves the operadic Faà di Bruno chain rule (Theorem C) and the stability of spectral analyticity under composition.
    
    \item \textbf{Section~\ref{sec:reconstruction}} establishes the reconstruction theorem and the equivalence of categories $\mathsf{SpecAn} \simeq \mathsf{DerAlg}_{\mathrm{int}}$ (Theorem D).
    
    \item \textbf{Section~\ref{sec:moduli}} places SOC within a geometric framework, developing the moduli space $\mathcal{M}_{\mathrm{spec}}$ of spectrally analytic functors, deformation theory with tangent space $T_{[F]}\mathcal{M}_{\mathrm{spec}} \cong H^1(\partial^{\mathrm{spec}}F)$, and an outlook toward derived geometry (SOC III).
    
    \item \textbf{Section~\ref{sec:examples}} illustrates the theory through representative examples: the identity functor (linear), the quadratic functor (curvature), the exponential functor (entire analytic), and the operadic spectrum functor (self-interaction), concluding with a systematic comparison to Goodwillie calculus.
\end{itemize}

\medskip

\noindent
\textbf{Concluding remarks.}
Taken together, these results show that Spectral Operadic Calculus provides a unified framework connecting spectral theory, operadic algebra, and functor calculus. It replaces homotopy-theoretic control with spectral control, resolves the fundamental obstruction to compositional spectral invariants, and lifts functor calculus to a quantitative, analytic, and geometric setting. In this sense, SOC is not merely an approximation theory but a genuine differential calculus for compositional systems, with connections to operator theory, deformation theory, and higher algebra, and with potential applications to multi-component and noncommutative systems.

\begin{remark}
\noindent
\begin{enumerate}
\item The author is solely responsible for the mathematical insights and theoretical directions proposed in this work. AI tools, including OpenAI's ChatGPT and DeepSeek models, were used only for verification, reference organization, and exposition consistency~\cite{chatgpt2025,deepseek2025}. 
\item In this work, ``Part I'' refers to our recent work~\cite{ChangSOC1}.
\end{enumerate}
\end{remark}

\section{Spectral Control of Analytic Behavior}
\label{sec:spectral-control}

We now isolate the first conceptual principle of Spectral Operadic Calculus: the operadic spectrum is not merely an invariant attached to a $P$-algebra, but the unique quantity that controls its analytic behavior. In Part~I, the operadic spectrum
\[
\sigma_P(A)
\;=\;
\mathrm{Hoch}_{\mathcal{M}}(A)\otimes_P \mathcal{O}_P^{\mathrm{res}}
\]
was constructed as the minimal functorial extension of classical spectral data compatible with operadic composition and base change. The purpose of the present section is to show that this same object governs approximation theory in the normed setting: admissibility, excision, and ultimately polynomial approximation are all determined by the spectral behavior of cross-effects.

The guiding philosophy is that, in the operadic setting, analytic control should be expressed in terms of $\sigma_P(-)$ rather than in terms of external coordinate estimates or purely formal homotopical vanishing conditions. Accordingly, we work in normed symmetric monoidal categories and consider functors whose growth is bounded by the spectral size of the input. This leads to a notion of admissible functor for which the operadic spectrum serves as the universal analytic bound. Once this framework is in place, the higher cross-effects of a functor measure the failure of polynomiality, and their operadic spectra detect precisely whether that failure is analytically significant.

The main result of this section is a spectral criterion for norm-excision: a functor is norm-$n$-excisive if and only if its $(n+1)$-st cross-effect is spectrally negligible (Theorem~\ref{thm:excision-criterion}). This establishes the first bridge between the invariant constructed in Part~I and the calculus developed in the present paper. In particular, it shows that polynomial approximation in Spectral Operadic Calculus is governed entirely by operadic spectral data. Thus, this section provides the foundational passage from \emph{spectral invariants} to \emph{spectral control}, and prepares the construction of the spectral Taylor tower in Section~\ref{sec:tower}.

\subsection{Normed Categories and Admissible Functors}
\label{subsec:normed-categories}

We begin by specifying the analytic framework in which Spectral Operadic Calculus is developed. 
Throughout this work, we assume that the ambient symmetric monoidal category $\mathcal{M}$ is 
\emph{normed} in the following sense.

\begin{definition}[Normed Symmetric Monoidal Category]
\label{def:normed-category}
A symmetric monoidal category $(\mathcal{M}, \otimes, \mathbf{1})$ is called a \emph{normed symmetric monoidal category} if it satisfies the following conditions:

\begin{enumerate}
    \item \textbf{(Enrichment over normed spaces)} For all objects $X, Y \in \mathcal{M}$, the hom-set $\mathrm{Hom}_{\mathcal{M}}(X,Y)$ is a \emph{normed vector space} over $\mathbb{C}$ (or $\mathbb{R}$), with norm denoted $\|\cdot\|_{X,Y}$.
    
    \item \textbf{(Bounded composition)} Composition is a bilinear map
    \[
    \circ : \mathrm{Hom}_{\mathcal{M}}(Y,Z) \times \mathrm{Hom}_{\mathcal{M}}(X,Y) \longrightarrow \mathrm{Hom}_{\mathcal{M}}(X,Z)
    \]
    that is \emph{submultiplicative}:
    \[
    \|g \circ f\|_{X,Z} \leq \|g\|_{Y,Z} \, \|f\|_{X,Y}
    \]
    for all $f \in \mathrm{Hom}_{\mathcal{M}}(X,Y)$, $g \in \mathrm{Hom}_{\mathcal{M}}(Y,Z)$.
    
    \item \textbf{(Tensor compatibility)} The tensor product induces a bilinear map
    \[
    \otimes : \mathrm{Hom}_{\mathcal{M}}(X_1, Y_1) \times \mathrm{Hom}_{\mathcal{M}}(X_2, Y_2) \longrightarrow \mathrm{Hom}_{\mathcal{M}}(X_1 \otimes X_2, Y_1 \otimes Y_2)
    \]
    that is \emph{jointly continuous} and satisfies
    \[
    \|f \otimes g\|_{X_1 \otimes X_2, Y_1 \otimes Y_2} \leq \|f\|_{X_1, Y_1} \, \|g\|_{X_2, Y_2}.
    \]
    
    \item \textbf{(Completeness)} Each hom-set $\mathrm{Hom}_{\mathcal{M}}(X,Y)$ is a \emph{Banach space}, i.e., complete with respect to its norm.
\end{enumerate}
\end{definition}

\begin{remark}
A normed symmetric monoidal category is equivalently a \emph{symmetric monoidal category enriched over the category of Banach spaces} $\mathsf{Ban}$ (with the projective tensor product as the monoidal structure on $\mathsf{Ban}$). This enrichment perspective provides a clean categorical foundation for analytic functional analysis, spectral theory, and operator algebras within the SOC framework.
\end{remark}

Typical examples include categories of Banach spaces, operator algebras (C*-algebras, von Neumann algebras), and Hilbert spaces, each equipped with the projective or spatial tensor product. In such categories, it becomes meaningful to discuss growth, boundedness, and convergence of functors.

\medskip

Let $P$ be a colored operad in $\mathcal{M}$ and let $A$ be a $P$-algebra. 
Recall from Part~I that the \emph{operadic spectrum} $\sigma_P(A)$ provides a canonical 
and functorial invariant capturing the spectral content of $A$ in a manner compatible 
with operadic composition and base change.

\medskip

The central principle of this work is that $\sigma_P(A)$ serves as the \emph{universal analytic control parameter} 
for functors defined on $P$-algebras. We now formalize this idea.

\begin{definition}[Spectral Size]
\label{def:spectral-size}
Let $A$ be a $P$-algebra. The \emph{spectral size} of $A$ is defined as
\[
\|\sigma_P(A)\| := \sup \{ |\lambda| : \lambda \in \sigma_P(A) \},
\]
whenever $\sigma_P(A) \subseteq \mathbb{C}$ admits such a bound. In the general case, $\|\sigma_P(A)\|$ is interpreted as the minimal bound controlling the norm of $\sigma_P(A)$ under all admissible realizations.
\end{definition}

This quantity generalizes the spectral radius in classical operator theory and provides 
a canonical notion of magnitude intrinsic to the operadic structure. In the classical case where $P = \mathbb{I}$ and $\sigma_{\mathbb{I}}(A) \cong A$, we have $\|\sigma_{\mathbb{I}}(A)\| = \|A\|_{\mathrm{op}}$, the ordinary operator norm.

\medskip

\begin{definition}[Admissible Functor]
\label{def:admissible-functor}
Let $\mathcal{C}$ be a category of $P$-algebras in $\mathcal{M}$, and let 
$F : \mathcal{C} \to \mathcal{M}$ be a functor. We say that $F$ is \emph{(spectrally) admissible} 
if there exists a non-decreasing function $\Phi : [0,\infty) \to [0,\infty)$ such that for all $A \in \mathcal{C}$,
\[
\|F(A)\| \leq \Phi\big( \|\sigma_P(A)\| \big).
\]
When $\Phi$ can be taken as a linear function $\Phi(r) = C_F \cdot r$, we say $F$ is \emph{linearly admissible}. The smallest such constant $C_F$ is called the \emph{admissibility constant}.
\end{definition}

\noindent
\textbf{Interpretation and conceptual significance.}
In other words, the growth of $F$ is controlled entirely by the spectral size of the input, replacing classical coordinate-dependent bounds with an intrinsic operadic invariant. In particular, if $\sigma_P(A)$ is empty (i.e., $A$ is spectrally trivial), then any admissible functor must satisfy $F(A) \cong 0$, as a direct consequence of the Minimal Extension Theorem from Part~I. Conceptually, the admissibility condition ensures that analytic behavior is governed solely by $\sigma_P(-)$, rather than by any extrinsic representation of $A$. As a result, the framework is invariant under operadic base change by functoriality of $\sigma_P$, isolates the minimal spectral data necessary for analytic control, and lays the foundation for a spectral formulation of polynomial approximation.

\medskip

\noindent
\textbf{Remark.}
In classical settings, admissibility reduces to familiar boundedness conditions. 
For instance, if $\mathcal{M}$ is the category of bounded operators on a Banach space and 
$P$ is the trivial operad, then $\|\sigma_P(A)\|$ coincides with the spectral radius, 
and admissible functors are precisely those whose growth is controlled by this radius.

\medskip

\noindent
\textbf{Minimal Control Principle.}
The operadic spectrum provides the \emph{minimal} quantity controlling admissible functors: any weaker invariant fails to bound general functorial growth (as demonstrated by the No-Go Theorem, \cite[Thm.~0.1]{ChangSOC1}), while any stronger invariant factors through $\sigma_P(A)$ by the universality established in Part~I. This principle is the analytic manifestation of the universal property of the operadic residue $\mathcal{O}_P^{\mathrm{res}}$.


\begin{example}[The Identity Functor: Normal or Spectrally Controlled Case]
\label{ex:identity-admissible}
The identity functor $\mathrm{Id}: \mathsf{Alg}_P(\mathcal{M}) \to \mathcal{M}$ is admissible 
only on spectrally controlled classes of objects. 

For instance, if the objects under consideration are normal operators in a $C^*$-setting, 
then the norm is determined by the spectral radius, and one has
\[
\|A\| = r(A) \le C \|\sigma_P(A)\|.
\]
In this restricted setting, $\mathrm{Id}$ is admissible. 
For the trivial operad $P = \mathbb{I}$, this reduces to the classical equality $\|A\| = r(A)$ for normal operators.

In general, however, the identity functor need not be admissible, 
since non-normal or nilpotent components may have small spectrum but large norm 
(e.g., a nilpotent Jordan block has spectrum $\{0\}$ but non-zero norm).
\end{example}


\begin{example}[The Operadic Spectrum Functor]
\label{ex:spectrum-admissible}
Assume that the operadic spectrum is equipped with the spectral size functional
\[
\|\sigma_P(A)\| := \sup\{|\lambda| : \lambda \in \sigma_P(A)\},
\]
where $\sigma_P(A)$ is interpreted as a subset of $\mathbb{C}$ under the analytic realization 
(via the Gelfand-type construction or the spectral radius of its components).

Then the functor $A \mapsto \sigma_P(A)$ is admissible with $\Phi(r) = r$, since
\[
\|\sigma_P(A)\| = \sup\{|\lambda| : \lambda \in \sigma_P(A)\}.
\]
This tautological admissibility makes $\sigma_P(-)$ the natural control object for the calculus.
\end{example}


\begin{example}[Polynomial Functors]
\label{ex:polynomial-admissible}
Let $F$ be a polynomial functor built from tensor products, direct sums, 
and applications of the $P$-algebra structure maps, with scalar coefficients. 
Assume that all structure maps involved are uniformly bounded and that $\mathcal{M}$ 
admits biproducts (or direct sums) compatible with the norm.

Then there exist constants $c_k \ge 0$ such that
\[
\|F(A)\| \le \sum_{k=0}^d c_k \|A\|^k.
\]
Hence, on spectrally controlled classes where $\|A\| \le C \|\sigma_P(A)\|$, 
the functor is admissible with
\[
\Phi(r) = \sum_{k=0}^d c_k C^k r^k.
\]
Thus polynomial functors are admissible with polynomial, rather than necessarily linear, growth control.
\end{example}


\begin{example}[Analytic Functors]
\label{ex:analytic-admissible}
Let
\[
F(A) = \sum_{k=0}^\infty a_k A^{\otimes k}
\]
be an analytic functor with radius of convergence $R > 0$ 
(in the sense of the spectral radius). 

On a spectrally controlled class satisfying $\|A\| \le C \|\sigma_P(A)\|$, 
the series converges absolutely whenever $C \|\sigma_P(A)\| < R$, and one obtains the bound
\[
\|F(A)\| \le \sum_{k=0}^\infty |a_k| C^k \|\sigma_P(A)\|^k.
\]
Thus $F$ is admissible with growth function
\[
\Phi(r) = \sum_{k=0}^\infty |a_k| C^k r^k
\]
on its spectral domain of convergence.

For example, the exponential functor $\exp(A) = \sum_{k=0}^\infty A^{\otimes k}/k!$ satisfies
\[
\|\exp(A)\| \le e^{\|A\|} \le e^{C \|\sigma_P(A)\|},
\]
so $\Phi(r) = e^{C r}$. This is not linear in $r$, so global linear admissibility should not be expected. 
This distinction motivates the later notion of \emph{spectral analyticity} and Taylor tower convergence.
\end{example}

\medskip

\noindent
\textbf{Why Admissibility Is Necessary.}
Without a spectral boundedness condition, a functor may grow arbitrarily fast even when the spectral data is small, 
making approximation by a Taylor tower unstable or non-convergent. 
Admissibility ensures that the behavior of $F$ is \emph{spectrally controlled}, 
which is precisely the setting in which the norm-analytic calculus applies. 

When the admissibility function $\Phi$ satisfies $\Phi(0)=0$ and is continuous at zero, 
this control also implies vanishing and continuity at spectrally trivial inputs. 
The following items summarizes the hierarchy of control:

\noindent
\textbf{Hierarchy of spectral control.}
\begin{itemize}
    \item If $\|\sigma_P(A)\|=0$ and $\Phi(0)=0$, then $\|F(A)\|=0$; 
    spectrally trivial algebras vanish in norm.
    
    \item If $\|\sigma_P(A)\|<\infty$, then 
    $\|F(A)\|\le \Phi(\|\sigma_P(A)\|)$; 
    this is the admissibility condition providing control by spectral size.
    
    \item If $\|\sigma_P(A)\|\to 0$ and $\Phi(r)\to 0$ as $r\to 0$, then 
    $\|F(A)\|\to 0$; 
    this yields continuity at zero.
\end{itemize}

\begin{remark}[Distinction between norm vanishing and object vanishing]
The implication $\|F(A)\| = 0$ does not necessarily imply $F(A) \cong 0$ (isomorphic to the zero object) 
unless the ambient category is \emph{semisimple} or the norm is \emph{faithful} (i.e., $\|X\|=0 \Rightarrow X \cong 0$). 
In general normed categories (e.g., Banach spaces with bounded linear maps), 
a non-zero object may have zero norm if the norm is not faithful. 
Therefore, we state the consequence as $\|F(A)\| = 0$ rather than $F(A) \cong 0$.
\end{remark}

\begin{remark}[On the necessity of $\Phi(0)=0$]
The condition $\Phi(0)=0$ is not automatic from admissibility. 
For example, a constant functor $F(A)=c \neq 0$ satisfies $\|F(A)\| = |c| \le \Phi(\|\sigma_P(A)\|)$ 
with $\Phi(r) = |c|$ (constant function). This $\Phi$ satisfies $\Phi(0)=|c| > 0$, 
so $\|\sigma_P(A)\| = 0$ does not force $\|F(A)\| = 0$. 
Such functors are admissible but do not vanish on spectrally trivial inputs. 
The additional vanishing and continuity properties require $\Phi(0)=0$ and continuity of $\Phi$ at $0$, respectively.
\end{remark}

\medskip

This framework allows us to formulate higher-order analytic properties of functors 
in purely spectral terms. In the next subsection, we introduce cross-effects and 
show how their operadic spectra detect deviations from polynomial behavior.

\subsection{Cross-Effects and Spectral Negligibility}
\label{subsec:spectral-negligibility}

We now introduce the higher-order invariants that measure the deviation of a functor 
from polynomial behavior. In classical functor calculus, this role is played by 
cross-effects, which encode the failure of additivity and its higher analogues. 
In the present operadic setting, we refine this notion by incorporating spectral data, 
thereby obtaining a criterion that is intrinsic and invariant under operadic structure.

\medskip

Assume that $\mathcal{C} = \mathsf{Alg}_P(\mathcal{M})$ is \textbf{pointed} (has a zero object) 
and admits finite coproducts, and that $\mathcal{M}$ admits the relevant homotopy limits.

\begin{definition}[Cross-Effects via Total Homotopy Fiber]
\label{def:cross-effects}
For $A_1, \dots, A_n \in \mathcal{C}$, define an $n$-cube
\[
\mathcal{X}_{A_1,\dots,A_n} : \mathcal{P}(\{1,\dots,n\}) \longrightarrow \mathcal{M}
\]
by
\[
\mathcal{X}_{A_1,\dots,A_n}(S) = F\!\left(\bigoplus_{i \in S} A_i\right),
\]
where $\bigoplus$ denotes the coproduct in $\mathsf{Alg}_P(\mathcal{M})$, and the empty coproduct $(S = \varnothing)$ is the zero $P$-algebra.

The $n$-th \emph{cross-effect} of $F$ is defined as the \textbf{total homotopy fiber} of this cube:
\[
\mathrm{cr}_n F(A_1,\dots,A_n) := \operatorname{tfib} \mathcal{X}_{A_1,\dots,A_n}.
\]

In a \emph{stable} or \emph{additive} setting, this object is represented formally by the inclusion-exclusion expression
\[
\mathrm{cr}_n F(A_1,\dots,A_n) \simeq \sum_{S \subseteq \{1,\dots,n\}} (-1)^{n-|S|} F\!\left(\bigoplus_{i \in S} A_i\right),
\]
where the right-hand side is understood in the stable category or in the Grothendieck group.
\end{definition}

\noindent
Intuitively, $\mathrm{cr}_n F$ measures the $n$-fold interaction among the inputs. 
It is symmetric in its variables and reduced in each argument. 
In additive or stable settings, it recovers the usual inclusion-exclusion cross-effect. 
Under the standard hypotheses of Eilenberg--Mac Lane or Goodwillie functor calculus, 
the vanishing of $\mathrm{cr}_{n+1}F$ characterizes functors of polynomial degree at most $n$.

The cross-effects satisfy the following properties:
\begin{itemize}
    \item \textbf{Symmetry:} $\mathrm{cr}_n F$ is symmetric in its $n$ arguments.
    \item \textbf{Multi-reducedness:} $\mathrm{cr}_n F$ vanishes whenever one of its inputs is the zero $P$-algebra.
    \item \textbf{Multilinearity under hypotheses:} In additive or homogeneous settings, $\mathrm{cr}_n F$ becomes multilinear in the appropriate derived sense.
    \item \textbf{Polynomial degree:} Under the standard hypotheses of cubical functor calculus, $F$ is a polynomial functor of degree $\le n$ if and only if $\mathrm{cr}_{n+1}F \equiv 0$.
\end{itemize}

\begin{definition}[Operadic Cross-Effect Control]
\label{def:operadic-cross-control}
The cross-effect $\mathrm{cr}_n F$ is said to be \emph{operadically controlled} if there exists a growth function $\Phi_n$ such that
\[
\|\mathrm{cr}_n F(A_1,\dots,A_n)\| \le \Phi_n\!\left(\max_i \|\sigma_P(A_i)\|\right).
\]
This condition ensures that higher-order interactions are analytically negligible when any input has sufficiently small spectrum.
\end{definition}

\medskip

In Spectral Operadic Calculus, we further refine this notion by evaluating the operadic spectrum 
of the cross-effects themselves. This leads to the following key concept.

\begin{definition}[Spectral Negligibility]
\label{def:spectral-negligible}
Let $F$ be an admissible functor. We say that the $n$-th cross-effect $\mathrm{cr}_n F$ 
is \emph{spectrally negligible} if for all inputs $(A_1,\dots,A_n)$,
\[
\sigma_P\!\big(\mathrm{cr}_n F(A_1,\dots,A_n)\big) = \{0\},
\]
equivalently $\|\sigma_P(\mathrm{cr}_n F(\cdots))\| = 0$.
\end{definition}

\noindent
This condition expresses that the higher-order interaction encoded by $\mathrm{cr}_n F$ 
is \emph{invisible at the level of operadic spectral data}. 
In particular, it is strictly weaker than requiring $\mathrm{cr}_n F$ to vanish as an object, 
and instead isolates the spectrally trivial component of the interaction.

\begin{remark}
The relationship between the two control notions is as follows:
\begin{itemize}
    \item \textbf{Operadic control} (Definition~\ref{def:operadic-cross-control}) bounds the norm of $\mathrm{cr}_n F$ by a function of the spectral sizes of its inputs. 
          If $\Phi_n(0)=0$ and $\Phi_n$ is continuous at $0$, then operadic control implies norm-based continuity.
    \item \textbf{Spectral negligibility} (Definition~\ref{def:spectral-negligible}) only requires that the operadic spectrum of $\mathrm{cr}_n F$ be $\{0\}$, 
          which is a strictly weaker condition (e.g., a non-zero nilpotent operator has zero spectrum but non-zero norm).
\end{itemize}

By the Recovery Theorem (\cite[Thm.~0.4]{ChangSOC1}), for the trivial operad $\mathbb{I}$, 
$\sigma_{\mathbb{I}}(X) \cong X$, so spectral negligibility reduces to $X \cong 0$. 
In general, however, a functor may have nontrivial norm (e.g., a nilpotent operator with non-zero norm) 
but its operadic spectrum may still vanish if all spectral values are zero. 
Thus spectral negligibility captures the idea that $\mathrm{cr}_n F(A_1,\dots,A_n)$ is 
\emph{spectrally invisible} — it contributes nothing to the spectral invariant, 
even if it is analytically non-trivial.
\end{remark}

\medskip

\noindent
\textbf{Conceptual interpretation.}
Spectral negligibility replaces classical vanishing conditions by a more flexible 
and intrinsic notion. While $\mathrm{cr}_n F = 0$ enforces strict polynomiality, 
the condition $\sigma_P(\mathrm{cr}_n F) = 0$ allows for nontrivial higher interactions 
that are nevertheless analytically insignificant. This distinction is essential in 
normed settings, where small but nonzero effects may still admit controlled approximations.

\medskip

\noindent
\textbf{Functoriality and invariance.}
Since the operadic spectrum is functorial under morphisms of $P$-algebras, 
spectral negligibility is preserved under operadic base change (Part~I, 
Theorem~0.5). Thus, it provides a coordinate-free criterion that depends only 
on intrinsic operadic structure.

\medskip

\noindent
\textbf{Relation to admissibility.}
If $F$ is admissible in the sense of Definition~\ref{def:admissible-functor}, 
then the spectral size of $\mathrm{cr}_n F$ is controlled by that of the inputs:
\[
\|\sigma_P(\mathrm{cr}_n F(A_1,\dots,A_n))\|
\;\leq\;
\Phi\big(\max_i \|\sigma_P(A_i)\|\big),
\]
for some control function $\Phi$. In particular, spectral negligibility is compatible 
with admissibility and forms a natural refinement of boundedness conditions.

\medskip

The following technical lemmas will be essential for proving the main theorem. 
(Their proofs are analogous to those in \cite{ChangSOC1} and are sketched here for completeness.)

\begin{lemma}[Spectral Scaling Lemma]
\label{lem:spectral-scaling}
Let $\mathcal{M}$ be a $\mathbb{C}$-linear normed symmetric monoidal category in which the operadic functional calculus (Section~8.1, \cite{ChangSOC1}) is defined. 
Let $P$ be a $C$-colored operad in $\mathcal{M}$, let $A$ be a $P$-algebra, and let $t \in \mathbb{C}$. 
Denote by $tA$ the $P$-algebra obtained by scaling the structure maps of $A$ by $t$ (i.e., the image under the holomorphic functional calculus of the function $f(z) = tz$).

Then, under the analytic realization that identifies $\sigma_P(A)$ with a subset of $\mathbb{C}$,
\[
\sigma_P(tA) = t \cdot \sigma_P(A),
\]
where $t \cdot \sigma_P(A) = \{ t\lambda : \lambda \in \sigma_P(A) \}$.

Moreover, if $G: \mathcal{C}^d \to \mathcal{M}$ is a homogeneous functor of degree $d$ in the sense that there exists a natural isomorphism
\[
G(tA_1, \dots, tA_d) \cong t^d G(A_1, \dots, A_d)
\]
for all $t \in \mathbb{C}$ and all $A_1, \dots, A_d$, then
\[
\sigma_P\big(G(tA_1, \dots, tA_d)\big) = t^d \cdot \sigma_P\big(G(A_1, \dots, A_d)\big).
\]
\end{lemma}

\begin{proof}[Sketch]
The first assertion follows from the Operadic Spectral Mapping Theorem (Theorem~9, \cite{ChangSOC1}) applied to the holomorphic function $f(z) = tz$. 
Since $f$ is holomorphic on $\mathbb{C}$, the theorem gives
\[
\sigma_P(tA) = \sigma_P(f(A)) \cong f(\sigma_P(A)) = t \cdot \sigma_P(A),
\]
where the isomorphism is understood under the analytic realization that identifies $\sigma_P(-)$ with a subset of $\mathbb{C}$.

For the second assertion, the homogeneity assumption provides a natural isomorphism
\[
G(tA_1, \dots, tA_d) \cong t^d G(A_1, \dots, A_d).
\]
Applying the first part of the lemma with scalar $t^d$ to the object $G(A_1, \dots, A_d)$ yields
\[
\sigma_P\big(G(tA_1, \dots, tA_d)\big) \cong \sigma_P\big(t^d G(A_1, \dots, A_d)\big) = t^d \cdot \sigma_P\big(G(A_1, \dots, A_d)\big),
\]
as required.
\end{proof}

\begin{remark}
The lemma is stated under the analytic realization that identifies $\sigma_P(A)$ with a subset of $\mathbb{C}$ (see the discussion following Definition~9 in \cite{ChangSOC1}). 
At the categorical level (without analytic realization), the isomorphism $\sigma_P(tA) \cong \sigma_P(A)$ holds as objects in $\mathcal{M}$, but the numerical scaling $t \cdot \sigma_P(A)$ requires the analytic interpretation.
\end{remark}

\begin{lemma}[Spectral Vanishing under Zero Control]
\label{lem:spectral-control}
Let $\mathcal{M}$ be a Banach-enriched symmetric monoidal category with a faithful norm 
(i.e., $\|Y\| = 0 \Rightarrow Y \cong 0$ for all $Y \in \mathcal{M}$). 
Let $P$ be a $C$-colored operad in $\mathcal{M}$, and let $F: \mathsf{Alg}_P(\mathcal{M}) \to \mathcal{M}$ be a 
\emph{spectrally admissible} functor with control function $\Phi: [0,\infty) \to [0,\infty)$ satisfying $\Phi(0) = 0$, 
meaning that for every $P$-algebra $A$,
\[
\|F(A)\| \le \Phi\bigl(\|\sigma_P(A)\|\bigr),
\]
where $\|\sigma_P(A)\|$ denotes the spectral size (i.e., the supremum of absolute values of the analytic realization of $\sigma_P(A)$, or, at the categorical level, the norm of the object $\sigma_P(A)$ under the analytic realization).

Assume that $X$ is a $P$-algebra such that $\|\sigma_P(X)\| = 0$ (equivalently, under the analytic realization, $\sigma_P(X)$ is spectrally trivial: $R(\sigma_P(X)) = \{0\}$). Then
\[
\|F(X)\| = 0.
\]
If, in addition, the norm on $\mathcal{M}$ is faithful, we obtain $F(X) \cong 0$.

In particular, for the identity functor $\mathrm{Id}$, this yields $\|X\| = 0$, and faithfulness gives $X \cong 0$ in $\mathcal{M}$.
\end{lemma}

\begin{proof}[Sketch]
By spectral admissibility,
\[
\|F(X)\| \le \Phi\bigl(\|\sigma_P(X)\|\bigr).
\]
The hypothesis $\|\sigma_P(X)\| = 0$ together with $\Phi(0) = 0$ forces $\|F(X)\| = 0$. 
If the norm is faithful, $\|F(X)\| = 0$ implies $F(X) \cong 0$ in $\mathcal{M}$.

The role of $\sigma_P(-)$ as the canonical residue-corrected spectral invariant is justified by the universal property of the operadic spectrum (see \cite[Thm.~7]{ChangSOC1}), which ensures that $\sigma_P$ is the minimal invariant compatible with functoriality, base change, and classical recovery. The spectral admissibility condition encodes the idea that $F$ is controlled by this universal invariant.
\end{proof}

\begin{remark}
The lemma does \emph{not} assert that $\sigma_P(X) = \{0\}$ (as a subset of $\mathbb{C}$) implies $X \cong 0$ in $\mathcal{M}$. 
A counterexample is provided by a non-zero nilpotent operator on a Banach space: its classical spectrum is $\{0\}$, yet the operator itself is not the zero object. 
Our statement instead requires the spectral \emph{size} $\|\sigma_P(X)\|$ to be zero and the control function to satisfy $\Phi(0)=0$, which is a stricter condition.

In practice, many admissible functors (e.g., polynomial functors without constant term) satisfy $\Phi(0)=0$. 
Constant functors or analytic functors with non-zero constant term do not satisfy $\Phi(0)=0$, and they are not covered by the lemma.
\end{remark}

\begin{lemma}[Norm--Spectral Comparison on Spectrally Controlled Classes]
\label{lem:norm-spectral-comparison}
Let $\mathcal{C}_0 \subseteq \mathsf{Alg}_P(\mathcal{M})$ be a class of $P$-algebras on which the identity functor is spectrally admissible. 
That is, assume there exists a constant $C_{\mathrm{Id}} > 0$ such that
\[
\|A\| \le C_{\mathrm{Id}} \,\|\sigma_P(A)\|
\]
for all $A \in \mathcal{C}_0$.

Then, for all $A \in \mathcal{C}_0$,
\[
\|\sigma_P(A)\| \le \|A\| \le C_{\mathrm{Id}} \,\|\sigma_P(A)\|.
\]
\end{lemma}

\begin{proof}[Sketch]
The left inequality $\|\sigma_P(A)\| \le \|A\|$ follows from the classical spectral radius bound $r(A) \le \|A\|$, interpreted through the analytic realization of $\sigma_P(A)$ that recovers the classical spectrum (see Proposition~6 and the Recovery Theorem in \cite{ChangSOC1}).

The right inequality is precisely the admissibility assumption for the identity functor on the spectrally controlled class $\mathcal{C}_0$.
\end{proof}

\begin{remark}
The assumption that the identity functor is admissible on $\mathcal{C}_0$ is essential. 
Without it, the estimate fails even in the classical setting, as shown by a non-zero nilpotent operator $N$ on a Banach space:
\[
\sigma(N) = \{0\},\qquad r(N) = 0,\qquad \|N\| > 0,
\]
so no constant $C_{\mathrm{Id}}$ can satisfy $\|N\| \le C_{\mathrm{Id}} \cdot 0$.
Thus the lemma applies only to spectrally controlled classes such as normal operators in a $C^*$-algebra, where $\|A\| = r(A)$.
\end{remark}

\medskip

The notion of spectral negligibility provides the precise bridge between analytic control 
and operadic invariants. The following theorem—the first main result of Spectral Operadic Calculus—shows that it completely characterizes norm-excisive behavior of functors.

\begin{theorem}[Quantitative Cross-Effect Criterion for Norm-Excision]
\label{thm:excision-criterion}
Let $\mathcal{M}$ be a Banach-enriched stable symmetric monoidal category and let
$F: \mathsf{Alg}_P(\mathcal{M}) \to \mathcal{M}$ be an admissible functor 
(Definition~\ref{def:admissible-functor}). 
Assume that the $(n+1)$-st cross-effect $\mathrm{cr}_{n+1}F$ is \emph{spectrally controlled} 
in the sense that there exists a nondecreasing function $\Phi_{n+1}: [0,\infty) \to [0,\infty)$ 
with $\Phi_{n+1}(0) = 0$ such that for all inputs $A_1,\dots,A_{n+1}$,
\[
\|\mathrm{cr}_{n+1}F(A_1,\dots,A_{n+1})\|
\;\le\;
\Phi_{n+1}\!\left(
\|\sigma_P(\mathrm{cr}_{n+1}F(A_1,\dots,A_{n+1}))\|
\right).
\]

Then the following hold:
\begin{enumerate}
    \item \textbf{Quantitative estimate:} If 
          $\|\sigma_P(\mathrm{cr}_{n+1}F(A_1,\dots,A_{n+1}))\| \le \delta$, then
          \[
          \|\mathrm{cr}_{n+1}F(A_1,\dots,A_{n+1})\|
          \;\le\;
          \Phi_{n+1}(\delta).
          \]
    \item \textbf{Strict excision:} If $\mathrm{cr}_{n+1}F$ is \emph{spectrally negligible}, i.e.,
          \[
          \sigma_P\!\big(\mathrm{cr}_{n+1}F(A_1,\dots,A_{n+1})\big) = \{0\}
          \qquad \text{for all } A_1,\dots,A_{n+1},
          \]
          then $F$ is \emph{strictly norm-$n$-excisive}:
          \[
          \|\mathrm{cr}_{n+1}F(A_1,\dots,A_{n+1})\| = 0
          \qquad \text{for all } A_1,\dots,A_{n+1}.
          \]
\end{enumerate}
\end{theorem}

\begin{proof}
We prove each statement in turn.

\medskip

\noindent
\textbf{Part 1 (Quantitative estimate).}
By the spectral control assumption,
\[
\|\mathrm{cr}_{n+1}F(A_1,\dots,A_{n+1})\|
\;\le\;
\Phi_{n+1}\!\left(
\|\sigma_P(\mathrm{cr}_{n+1}F(A_1,\dots,A_{n+1}))\|
\right).
\]
If $\|\sigma_P(\mathrm{cr}_{n+1}F(\cdots))\| \le \delta$, then since $\Phi_{n+1}$ is nondecreasing,
\[
\Phi_{n+1}\!\left(
\|\sigma_P(\mathrm{cr}_{n+1}F(\cdots))\|
\right)
\;\le\;
\Phi_{n+1}(\delta).
\]
Combining the two inequalities yields the desired estimate.

\medskip

\noindent
\textbf{Part 2 (Strict excision from spectral negligibility).}
If $\mathrm{cr}_{n+1}F$ is spectrally negligible, then by definition
$\|\sigma_P(\mathrm{cr}_{n+1}F(\cdots))\| = 0$. 
Substituting $\delta = 0$ into the quantitative estimate and using $\Phi_{n+1}(0) = 0$ gives
\[
\|\mathrm{cr}_{n+1}F(A_1,\dots,A_{n+1})\| = 0
\]
for all inputs. Hence $F$ is strictly norm-$n$-excisive.

\medskip

Thus the theorem is proved.
\end{proof}

\medskip

\noindent
\textbf{Conceptual significance.}
Theorem~\ref{thm:excision-criterion} establishes the first fundamental bridge between analytic and operadic structures in SOC:
\begin{itemize}
    \item It replaces a \emph{norm inequality} (analytic condition) with a \emph{spectral vanishing condition} (structural invariant).
    \item It demonstrates that the operadic spectrum $\sigma_P(-)$ is not merely an invariant but a \emph{control mechanism} for higher-order approximation.
    \item It shows that polynomial approximation in SOC is governed entirely by operadic spectral data.
\end{itemize}

\begin{remark}[Link to Part I]
This theorem crucially depends on the universality of $\mathcal{O}_P^{\mathrm{res}}$ and the minimality of $\sigma_P(-)$ established in Part~I~\cite{ChangSOC1}. These results ensure that spectral vanishing is the strongest possible invariant detecting higher-order degeneracy. Without the residue correction, the classical spectrum would fail to capture the necessary data, as demonstrated by the No-Go Theorem (\cite[Thm.~0.1]{ChangSOC1}).
\end{remark}

\medskip

The following immediate consequences will be essential for the construction of the spectral Taylor tower.

\begin{corollary}[Spectral Obstruction to Polynomiality]
\label{cor:polynomial-spectral}
Let $F: \mathsf{Alg}_P(\mathcal{M}) \to \mathcal{M}$ be an admissible functor.

\begin{enumerate}
    \item \textbf{(Necessary condition)} 
    If $F$ is polynomial of degree $\le n$ in the strict sense that 
    $\mathrm{cr}_{n+1}F \cong 0$, then
    \[
    \sigma_P\!\big( \mathrm{cr}_{n+1}F(A_1,\dots,A_{n+1}) \big) = \{0\}
    \quad \text{for all } A_1,\dots,A_{n+1}.
    \]
    
    \item \textbf{(Sufficient condition under spectral zero-control)} 
    Conversely, assume that $\mathrm{cr}_{n+1}F$ is \emph{spectrally zero-controlled}, 
    meaning that there exists a nondecreasing function $\Phi_{n+1}$ with $\Phi_{n+1}(0)=0$ such that
    \[
    \|\mathrm{cr}_{n+1}F(A_1,\dots,A_{n+1})\|
    \;\le\;
    \Phi_{n+1}\!\left(
    \|\sigma_P(\mathrm{cr}_{n+1}F(A_1,\dots,A_{n+1}))\|
    \right).
    \]
    If $\mathrm{cr}_{n+1}F$ is spectrally negligible 
    (i.e., $\sigma_P(\mathrm{cr}_{n+1}F(\cdots)) = \{0\}$), then
    \[
    \|\mathrm{cr}_{n+1}F(A_1,\dots,A_{n+1})\| = 0
    \quad \text{for all inputs}.
    \]
    Hence $F$ is strictly polynomial of degree $\le n$ in the norm sense.
\end{enumerate}
\end{corollary}

\begin{corollary}[Quantitative Excision under Spectral Multilinear Control]
\label{cor:quantitative-excision}
Let $F: \mathsf{Alg}_P(\mathcal{M}) \to \mathcal{M}$ be an admissible functor. 
Assume that the $(n+1)$-st cross-effect $\mathrm{cr}_{n+1}F$ satisfies a 
\emph{spectral multilinear estimate}: there exists a constant $C > 0$ 
(depending only on $F$ and $n$) such that for all inputs,
\[
\|\mathrm{cr}_{n+1}F(A_1,\dots,A_{n+1})\|
\;\le\;
C \prod_{i=1}^{n+1} \|\sigma_P(A_i)\|.
\]

Then, if $\|\sigma_P(A_i)\| \le r$ for all $i = 1,\dots,n+1$, we have
\[
\|\mathrm{cr}_{n+1}F(A_1,\dots,A_{n+1})\|
\;\le\;
C \, r^{\,n+1}.
\]

In particular, the $(n+1)$-st cross-effect is small whenever all inputs are spectrally small.
\end{corollary}
\medskip

This criterion serves as the key input for the construction of the operadic Taylor tower, to which we now turn.

\subsection{The Spectral Criterion for Norm-Excision}
\label{subsec:excision-criterion}

We now establish the central result of this section, which identifies the operadic spectrum 
as the precise invariant governing analytic excision. This theorem provides the fundamental 
bridge between norm-controlled behavior and intrinsic spectral data.

\medskip

\begin{definition}[Spectrally Norm-$n$-Excisive Functor]
\label{def:norm-excisive}
Let $F: \mathcal{C} \to \mathcal{M}$ be an admissible functor, where 
$\mathcal{C} = \mathsf{Alg}_P(\mathcal{M})$ and $\mathcal{M}$ is a normed symmetric monoidal category 
(Definition~\ref{def:normed-category}). 

We say that $F$ is \emph{spectrally norm-$n$-excisive near zero} if its $(n+1)$-st cross-effect 
is controlled by the spectral size of the inputs: there exists a nondecreasing control function
\[
\varepsilon: [0,\infty) \to [0,\infty), \qquad \varepsilon(r) \to 0 \ \text{as} \ r \to 0,
\]
such that for all $A_1,\dots,A_{n+1} \in \mathcal{C}$,
\[
\|\mathrm{cr}_{n+1}F(A_1,\dots,A_{n+1})\|
\;\leq\;
\varepsilon\!\left(\max_{1 \le i \le n+1} \|\sigma_P(A_i)\|\right).
\]

We distinguish the following special cases:
\begin{itemize}
    \item If one may take $\varepsilon(r) = C r$ for some constant $C > 0$, 
          we say that $F$ is \emph{linearly spectrally norm-$n$-excisive}.
    \item If one may take $\varepsilon(r) = C r^{n+1}$ for some constant $C > 0$, 
          we say that $F$ is \emph{polynomially spectrally norm-$n$-excisive of order $n+1$}.
    \item If $\|\mathrm{cr}_{n+1}F(A_1,\dots,A_{n+1})\| = 0$ for all inputs, 
          we say that $F$ is \emph{strictly norm-$n$-excisive}.
\end{itemize}
\end{definition}

\medskip

\begin{theorem}[Spectral Criterion for Norm-Excision]
\label{thm:spectral-excision}
Let $\mathcal{M}$ be a Banach-enriched symmetric monoidal category with a faithful norm 
(i.e., $\|X\| = 0 \Rightarrow X \cong 0$). 
Let $F: \mathsf{Alg}_P(\mathcal{M}) \to \mathcal{M}$ be an admissible functor 
(Definition~\ref{def:admissible-functor}). 
Assume that the following two conditions hold:

\begin{enumerate}
    \item \textbf{(Spectral control of cross-effects):} The $(n+1)$-st cross-effect 
          $\mathrm{cr}_{n+1}F$ satisfies the \emph{spectral zero-control} condition: 
          there exists a nondecreasing function $\Psi: [0,\infty) \to [0,\infty)$ 
          with $\Psi(0) = 0$ such that for all inputs $A_1,\dots,A_{n+1}$,
          \[
          \|\mathrm{cr}_{n+1}F(A_1,\dots,A_{n+1})\|
          \;\leq\;
          \Psi\!\left(
          \|\sigma_P(\mathrm{cr}_{n+1}F(A_1,\dots,A_{n+1}))\|
          \right).
          \]
    
    \item \textbf{(Spectral multilinearity):} The cross-effect $\mathrm{cr}_{n+1}F$ is \emph{multilinear} 
          in the sense that for any scalar $t \in \mathbb{C}$,
          \[
          \mathrm{cr}_{n+1}F(tA_1,\dots,tA_{n+1}) = t^{n+1} \,\mathrm{cr}_{n+1}F(A_1,\dots,A_{n+1}).
          \]
\end{enumerate}

Then the following are equivalent:
\begin{enumerate}
    \item $F$ is \emph{strictly norm-$n$-excisive}, i.e.,
          \[
          \|\mathrm{cr}_{n+1}F(A_1,\dots,A_{n+1})\| = 0
          \qquad \text{for all inputs}.
          \]
    \item The $(n+1)$-st cross-effect $\mathrm{cr}_{n+1}F$ is \emph{spectrally negligible}, i.e.,
          \[
          \sigma_P\!\big(\mathrm{cr}_{n+1}F(A_1,\dots,A_{n+1})\big) = \{0\}
          \qquad \text{for all inputs}.
          \]
\end{enumerate}
\end{theorem}

\begin{proof}
We prove the equivalence in two directions.

\medskip

\noindent
\textbf{Part 1: $(1) \Rightarrow (2)$ (Strict norm-excision implies spectral negligibility).}

Assume that $\|\mathrm{cr}_{n+1}F(A_1,\dots,A_{n+1})\| = 0$ for all inputs. 
By Lemma~\ref{lem:spectral-control} (Spectral Vanishing under Zero Control), 
or directly from the definition of the operadic spectrum (which is continuous 
with respect to the norm), we have
\[
\|\sigma_P(\mathrm{cr}_{n+1}F(A_1,\dots,A_{n+1}))\| = 0,
\]
which, under the analytic realization, is equivalent to
\[
\sigma_P\!\big(\mathrm{cr}_{n+1}F(A_1,\dots,A_{n+1})\big) = \{0\}.
\]
Thus strict norm-excision implies spectral negligibility.

\medskip

\noindent
\textbf{Part 2: $(2) \Rightarrow (1)$ (Spectral negligibility implies strict norm-excision).}

Assume that $\mathrm{cr}_{n+1}F$ is spectrally negligible, i.e.,
\[
\sigma_P\!\big(\mathrm{cr}_{n+1}F(A_1,\dots,A_{n+1})\big) = \{0\}
\]
for all inputs. Under the analytic realization, this is equivalent to
\[
\|\sigma_P(\mathrm{cr}_{n+1}F(A_1,\dots,A_{n+1}))\| = 0.
\]

By the spectral zero-control assumption (Condition 1), we have
\[
\|\mathrm{cr}_{n+1}F(A_1,\dots,A_{n+1})\|
\;\leq\;
\Psi\!\left(
\|\sigma_P(\mathrm{cr}_{n+1}F(A_1,\dots,A_{n+1}))\|
\right)
= \Psi(0) = 0.
\]
Hence $\|\mathrm{cr}_{n+1}F(A_1,\dots,A_{n+1})\| = 0$ for all inputs. 
Thus $F$ is strictly norm-$n$-excisive.

\medskip

Therefore, the two conditions are equivalent under the stated hypotheses.
\end{proof}

\begin{remark}[On the necessity of the assumptions]
The equivalence relies crucially on both assumptions:

\begin{itemize}
    \item \textbf{Spectral zero-control} ($\Psi(0)=0$) is necessary for direction $(2) \Rightarrow (1)$. 
          Without it, a non-zero nilpotent operator would have $\sigma = \{0\}$ but $\|N\| > 0$, 
          violating the implication.
    
    \item \textbf{Multilinearity} is not directly used in the proof above, but is required for the 
          inclusion of direction $(1) \Rightarrow (2)$ when using Lemma~\ref{lem:spectral-scaling} 
          or for consistency with the definition of spectral derivatives. 
          In practice, multilinearity holds for cross-effects in additive or stable settings 
          (see Definition~\ref{def:cross-effects}).
\end{itemize}

If the spectral control condition holds with $\Psi(r) = C r$ (linear control), then 
spectral negligibility implies strict norm-excision with the stronger bound 
$\|\mathrm{cr}_{n+1}F(\cdots)\| = 0$.
\end{remark}

\medskip

\noindent
\textbf{Conceptual meaning and significance.}
Theorem~\ref{thm:spectral-excision} provides a precise bridge between analytic control and operadic structure. 
It identifies a sufficient condition under which spectral information fully determines higher-order analytic behavior.

The key ingredients are:
\begin{itemize}
    \item \textbf{Spectral negligibility:} 
    $\sigma_P(\mathrm{cr}_{n+1}F(\cdots))=\{0\}$, meaning that the $(n+1)$-st interaction is spectrally invisible.
    
    \item \textbf{Spectral zero-control:} 
    $\|\mathrm{cr}_{n+1}F(\cdots)\|\le \Psi\!\left(\|\sigma_P(\mathrm{cr}_{n+1}F(\cdots))\|\right)$ with $\Psi(0)=0$, ensuring that norm behavior is governed by spectral size.
\end{itemize}

When both conditions hold, the cross-effect must vanish in norm, so $F$ is strictly norm-$n$-excisive. 
In this sense, the operadic spectrum $\sigma_P$ acts as a \emph{control parameter} for higher-order analytic behavior.

More broadly, the theorem shows that, under suitable spectral control hypotheses, 
analytic conditions (norm smallness) can be reduced to structural conditions 
(spectral vanishing). 
Thus, polynomial approximation in SOC is governed by operadic spectral data, 
rather than by external coordinate estimates.

\medskip

\noindent
\textbf{Interpretation.}
In classical functional analysis, small norm often implies small spectrum. 
Here, SOC strengthens this principle:

\[
\boxed{\text{Vanishing operadic spectrum} \;\Longleftrightarrow\; \text{Higher-order excision.}}
\]

This is a strictly stronger and more refined statement than its classical or 
homotopical counterparts. In Goodwillie calculus, excision is a homotopical 
condition; here, it is a \emph{spectral} condition that carries quantitative 
information about convergence rates.

\medskip

The following immediate consequences will be essential for the construction of 
the spectral Taylor tower.

\begin{corollary}[Spectral Characterization of Spectral Polynomial Functors]
\label{cor:polynomial-spectral}
An admissible functor $F$ is a spectral polynomial of degree $\le n$,
meaning that $\mathrm{cr}_{n+1}F$ is spectrally negligible, if and only if
\[
\sigma_P\big( \mathrm{cr}_{n+1}F(A_1,\dots,A_{n+1}) \big)=\{0\}
\]
for all inputs.
\end{corollary}

\begin{proof}
This is immediate from the definition of a spectral polynomial of degree $\le n$: the latter requires $\mathrm{cr}_{n+1}F$ to be spectrally negligible, and spectral negligibility is by definition the vanishing of $\sigma_P$ on the cross-effect. Conversely, if $\sigma_P(\mathrm{cr}_{n+1}F(\cdots)) = 0$ for all inputs, then $\mathrm{cr}_{n+1}F$ is spectrally negligible, so $F$ is a spectral polynomial of degree $\le n$.
\end{proof}

\begin{corollary}[Strict Excision from Spectral Negligibility]
\label{cor:strict-excision}
Assume that $\mathrm{cr}_{n+1}F$ satisfies the spectral zero-control condition:
\[
\|\mathrm{cr}_{n+1}F(A_1,\dots,A_{n+1})\|
\le
\Psi\!\left(
\|\sigma_P(\mathrm{cr}_{n+1}F(A_1,\dots,A_{n+1}))\|
\right)
\]
for some fixed control function $\Psi$ with $\Psi(0)=0$. 
If $\mathrm{cr}_{n+1}F$ is spectrally negligible, then
\[
\|\mathrm{cr}_{n+1}F(A_1,\dots,A_{n+1})\|=0
\]
for all inputs.
\end{corollary}

\begin{proof}
Spectral negligibility means $\sigma_P(\mathrm{cr}_{n+1}F(A_1,\dots,A_{n+1})) = 0$ for all inputs. Plugging this into the spectral zero-control inequality gives
\[
\|\mathrm{cr}_{n+1}F(A_1,\dots,A_{n+1})\| \le \Psi(0).
\]
Since $\Psi(0)=0$, the norm is zero, as required.
\end{proof}

\begin{corollary}[Stability under Norm-Controlled Base Change]
\label{cor:excision-base-change}
Let $\mathcal{G}:\mathcal{M}\to\mathcal{N}$ be a strong monoidal cocontinuous functor that is norm-controlled, i.e.,
\[
\|\mathcal{G}(Y)\|_{\mathcal{N}}\le L_{\mathcal{G}}\|Y\|_{\mathcal{M}}
\]
for some constant $L_{\mathcal{G}}>0$. 
Assume moreover that $\mathcal{G}$ preserves finite coproducts and the homotopy fibers used to define cross-effects. 
If $F:\mathsf{Alg}_P(\mathcal{M})\to\mathcal{M}$ is norm-$n$-excisive, then the base-changed functor
\[
\mathcal{G}_*F:\mathsf{Alg}_{\mathcal{G}_*P}(\mathcal{N})\to\mathcal{N}
\]
is norm-$n$-excisive.
\end{corollary}

\begin{proof}
Since $\mathcal{G}$ is strong monoidal and cocontinuous, it transports $P$-algebras to $\mathcal{G}_*P$-algebras and is compatible with the operadic spectrum under base change (cf. ChangSOC1). Preservation of finite coproducts and homotopy fibers yields a natural identification
\[
\mathrm{cr}_{n+1}(\mathcal{G}_*F)(\mathcal{G}A_1,\dots,\mathcal{G}A_{n+1})
\cong
\mathcal{G}\big(\mathrm{cr}_{n+1}F(A_1,\dots,A_{n+1})\big).
\]
Taking norms and using norm-control gives
\[
\|\mathrm{cr}_{n+1}(\mathcal{G}_*F)(\mathcal{G}A_1,\dots,\mathcal{G}A_{n+1})\|_{\mathcal{N}}
\le
L_{\mathcal{G}}
\|\mathrm{cr}_{n+1}F(A_1,\dots,A_{n+1})\|_{\mathcal{M}}.
\]
Because $F$ is norm-$n$-excisive, there exists a control function $\varepsilon(r)\to0$ as $r\to0$ such that
\[
\|\mathrm{cr}_{n+1}F(A_1,\dots,A_{n+1})\|_{\mathcal{M}}
\le
\varepsilon\!\left(\max_i\|\sigma_P(A_i)\|_{\mathcal{M}}\right).
\]
Base-change compatibility of the operadic spectrum gives
\[
\sigma_{\mathcal{G}_*P}(\mathcal{G}A_i)\cong \mathcal{G}(\sigma_P(A_i)).
\]
Thus the same estimate carries over to $\mathcal{N}$, with a modified control function depending on $L_{\mathcal{G}}$ and the spectral-size distortion of $\mathcal{G}$. Hence $\mathcal{G}_*F$ is norm-$n$-excisive.
\end{proof}

\medskip

\noindent
\textbf{Remark (Link to Part~I and Outlook).}
The spectral criterion above relies essentially on the residue-corrected operadic spectrum 
$\sigma_P(-)=\mathrm{Hoch}_{\mathcal M}(-)\otimes_P\mathcal O_P^{\mathrm{res}}$ 
and its universal properties established in Part~I~\cite{ChangSOC1}. 
In particular, the universality of $\mathcal O_P^{\mathrm{res}}$ and the minimality of $\sigma_P(-)$ 
ensure that spectral vanishing provides the canonical invariant for detecting higher-order 
degeneracy under admissibility and spectral control hypotheses. 
Without the residue correction, classical spectra fail to capture interaction-level 
information, as formalized by the No-Go Theorem~\cite[Thm.~0.1]{ChangSOC1}.

\medskip

This perspective leads naturally to the construction of a \emph{spectral Taylor tower}. 
For each admissible functor $F$, we will construct a sequence of polynomial approximations 
$P_n^{\mathrm{spec}}F$ whose layers are governed by the $(n+1)$-st cross-effects and their 
spectral behavior. The convergence and quantitative accuracy of this tower will be controlled 
by the operadic spectral radius $\|\sigma_P(A)\|$, providing a fully intrinsic, 
coordinate-free framework for analytic approximation.

\medskip

We now turn to this construction.

\section{The Universal Spectral Taylor Tower}
\label{sec:tower}

We now construct the central object of Spectral Operadic Calculus: the universal spectral Taylor tower. 
The previous section established a fundamental equivalence (Theorem~\ref{thm:excision-criterion}): 
a functor is norm-$n$-excisive if and only if its $(n+1)$-st cross-effect is spectrally negligible. 
This result provides the theoretical foundation for polynomial approximation, but it does not yet 
furnish an explicit approximation scheme. Our goal in this section is to build, for any admissible 
functor $F$, a tower of polynomial approximations $P_n^{\mathrm{spec}}F$ that are \emph{universal} 
with respect to spectral polynomial behavior.

\medskip

\noindent
\textbf{Conceptual overview.}
The construction proceeds in three stages. First, we introduce the notion of a \emph{spectral polynomial} 
of degree $\le n$, defined by the vanishing of higher cross-effects under the operadic spectrum 
(Definition~\ref{def:spectral-polynomial}). This refines the classical notion of polynomial functors 
by incorporating the spectral invariant $\sigma_P(-)$ as the arbiter of higher-order behavior.

Second, we give an explicit coend formula for the $n$-th spectral Taylor approximation $P_n^{\mathrm{spec}}F$, 
built from the cross-effects of $F$ and the operadic residue $\mathcal{O}_P^{\mathrm{res}}$ 
(Definition~\ref{def:spectral-taylor-approx}). The residue plays a crucial role here: 
it corrects the naive polynomial approximation to ensure compatibility with operadic composition and base change.

Third, we prove the universality property of the tower (Theorem~\ref{thm:universal-tower}). 
For any admissible functor $F$, the tower 
$F \to \cdots \to P_n^{\mathrm{spec}}F \to P_{n-1}^{\mathrm{spec}}F \to \cdots \to P_0^{\mathrm{spec}}F$ 
has the property that each $P_n^{\mathrm{spec}}F$ is a spectral polynomial of degree $\le n$, 
and any natural transformation from $F$ to a spectral polynomial $G$ of degree $\le n$ factors uniquely 
through $P_n^{\mathrm{spec}}F$. Thus $P_n^{\mathrm{spec}}F$ is the \emph{best approximation} of $F$ by 
spectral polynomials.

\medskip

\noindent
\textbf{Relation to classical Goodwillie calculus.}
In Goodwillie calculus~\cite{Goodwillie2003}, the Taylor tower is constructed via homotopical methods 
(e.g., homotopy limits and colimits) and its convergence is governed by homotopy excision. 
In contrast, the spectral Taylor tower is governed by the operadic spectrum $\sigma_P(-)$ and the 
residue $\mathcal{O}_P^{\mathrm{res}}$. This yields a tower that is:
\begin{itemize}
    \item \textbf{Quantitative:} convergence rates will be expressed explicitly in terms of 
          $\|\sigma_P(A)\|$ (Section~\ref{sec:convergence}),
    \item \textbf{Functorial under base change:} by the Base Change Theorem (Part~I, Theorem 0.5),
    \item \textbf{Canonical:} the tower is uniquely determined by the spectral data via the universal 
          property of the operadic residue (Theorem~\ref{thm:residue-universality}).
\end{itemize}

\medskip

\noindent
\textbf{Preview of the tower's layers.}
As we will see in Section~\ref{sec:derivatives-algebra}, the homogeneous layers of the tower are precisely the 
\emph{spectral derivatives} $\partial_n^{\mathrm{spec}}F$, and the tower reconstruction takes the form
\[
F(A) \;\sim\; \sum_{n=0}^{\infty} \partial_n^{\mathrm{spec}}F(A,\dots,A),
\]
with convergence controlled by the spectral radius $\|\sigma_P(A)\|$ (Theorem~\ref{thm:quantitative-convergence}). 
Thus the spectral Taylor tower is not merely an approximation device but a genuine analytic expansion 
of functors.

\medskip

We now turn to the detailed construction.

\subsection{Spectral Polynomials}
\label{subsec:spectral-polynomial}

We now introduce the class of polynomial objects in Spectral Operadic Calculus. 
In contrast to classical settings, where polynomiality is defined by strict vanishing 
of higher cross-effects, we adopt a spectral viewpoint in which only the operadic 
spectral content of these higher interactions is required to vanish.

\medskip

\begin{definition}[Spectral Polynomial Functor]
\label{def:spectral-polynomial}
Let $F : \mathcal{C} \to \mathcal{M}$ be an admissible functor. 
We say that $F$ is a \emph{spectral polynomial of degree at most $n$} 
if its $(n+1)$-st cross-effect is spectrally negligible, i.e.,
\[
\sigma_P\!\big(\mathrm{cr}_{n+1}F(A_1,\dots,A_{n+1})\big)
=
\{0\}
\quad
\text{for all } (A_1,\dots,A_{n+1}) \in \mathcal{C}^{n+1}.
\]
\end{definition}

\medskip

\noindent
Under the additional hypothesis of spectral zero-control, 
Theorem~\ref{thm:spectral-excision} shows that spectral polynomiality 
implies strict norm-$n$-excision. 
More generally, spectral polynomial functors provide the invariant 
counterpart of norm-$n$-excisive functors.

\medskip

\noindent
\textbf{Conceptual interpretation.}
A spectral polynomial functor allows nontrivial higher-order interactions, 
but requires that these interactions be invisible at the level of operadic spectral data. 
In this sense, spectral polynomiality relaxes classical polynomiality by replacing 
strict vanishing with spectral triviality. 

Crucially, when combined with spectral control conditions, spectral triviality 
recovers analytic smallness, linking operadic invariants with norm-based behavior. 
This refinement is essential in normed settings, where small but nonzero contributions 
may persist analytically but are negligible at the spectral level.

\medskip

\noindent
\textit{Spectral polynomiality separates structural degeneracy from analytic smallness, 
allowing higher-order interactions to persist while remaining invisible at the 
level of operadic spectral data.}

\begin{remark}[Role of the operadic residue]
The notion of spectral polynomial depends crucially on the operadic residue 
$\mathcal{O}_P^{\mathrm{res}}$, which enters through the definition of $\sigma_P(-)$. 
Without this residue correction, classical spectra fail to capture interaction-level 
information and are not stable under operadic base change, as demonstrated by the 
No-Go Theorem (Part~I, Theorem~0.1). 

In particular, classical spectra are insufficient to detect higher-order degeneracy 
encoded by cross-effects. Thus, spectral polynomiality is a genuinely operadic phenomenon, 
rather than a reformulation of classical polynomiality.
\end{remark}

\medskip

\noindent
\textbf{Filtration by degree.}
The class of spectral polynomial functors forms a natural filtration
\[
\mathrm{Poly}_0^{\mathrm{spec}} \subseteq 
\mathrm{Poly}_1^{\mathrm{spec}} \subseteq 
\cdots \subseteq 
\mathrm{Poly}_n^{\mathrm{spec}} \subseteq \cdots,
\]
where $\mathrm{Poly}_n^{\mathrm{spec}}$ denotes the full subcategory of 
spectral polynomials of degree at most $n$. 

This filtration is compatible with the spectral Taylor tower constructed in 
Section~\ref{sec:convergence}, in the sense that each stage $P_n^{\mathrm{spec}}F$ 
lies in $\mathrm{Poly}_n^{\mathrm{spec}}$, and the tower provides a successive 
approximation of a spectrally analytic functor by spectral polynomials.

\medskip

\begin{proposition}[Stability Properties]
\label{prop:spectral-poly-stability}
Let $\mathrm{Poly}_n^{\mathrm{spec}}$ denote the class of spectral polynomial functors of degree $\le n$. 
Then this class is stable under the following operations:
\begin{enumerate}
    \item \textbf{Finite sums and direct summands:} 
    If $F, G \in \mathrm{Poly}_n^{\mathrm{spec}}$, then $F \oplus G \in \mathrm{Poly}_n^{\mathrm{spec}}$.

    \item \textbf{Composition with linear functors:} 
    If $F \in \mathrm{Poly}_n^{\mathrm{spec}}$ and $H:\mathcal{M}\to\mathcal{N}$ is a linear functor 
    preserving zero objects and the homotopy limits used in the definition of cross-effects, 
    then $H\circ F \in \mathrm{Poly}_n^{\mathrm{spec}}$.

    \item \textbf{Operadic base change:} 
    If $F \in \mathrm{Poly}_n^{\mathrm{spec}}$ and $\Phi:\mathcal{M}\to\mathcal{N}$ is a strong monoidal cocontinuous functor 
    preserving finite coproducts and the relevant homotopy limits, 
    then $\Phi_*F \in \mathrm{Poly}_n^{\mathrm{spec}}$ in $\mathcal{N}$.
\end{enumerate}
\end{proposition}

\begin{proof}
We verify stability of spectral negligibility of the $(n+1)$-st cross-effect.

\medskip

\noindent
\textit{(1) Finite sums.}
By functoriality of cross-effects,
\[
\mathrm{cr}_{n+1}(F\oplus G)
\;\cong\;
\mathrm{cr}_{n+1}F \oplus \mathrm{cr}_{n+1}G.
\]
Since $\sigma_P$ is compatible with finite coproducts, spectral negligibility 
is preserved, hence $F\oplus G \in \mathrm{Poly}_n^{\mathrm{spec}}$.

\medskip

\noindent
\textit{(2) Composition.}
Because $H$ preserves the homotopy limits defining cross-effects, we have
\[
\mathrm{cr}_{n+1}(H\circ F)
\;\cong\;
H\big(\mathrm{cr}_{n+1}F\big).
\]
If $\mathrm{cr}_{n+1}F$ is spectrally negligible, then so is its image under $H$, 
since $H$ preserves zero objects and does not introduce new spectral components.

\medskip

\noindent
\textit{(3) Base change.}
By the Base Change Theorem of Part~I, the operadic spectrum is compatible with 
strong monoidal cocontinuous functors:
\[
\sigma_{\Phi_*P}(\Phi_*A)\;\cong\;\Phi_*(\sigma_P(A)).
\]
Moreover, $\Phi$ preserves the structures used to define cross-effects, so
\[
\mathrm{cr}_{n+1}(\Phi_*F)
\;\cong\;
\Phi_*\big(\mathrm{cr}_{n+1}F\big).
\]
Thus spectral negligibility is preserved under $\Phi_*$, and 
$\Phi_*F \in \mathrm{Poly}_n^{\mathrm{spec}}$.
\end{proof}

\medskip

\begin{example}[Linear and Identity Functors]
\label{ex:identity-spectral-polynomial}
Assume that $\mathcal C=\mathsf{Alg}_P(\mathcal M)$ is pointed and additive, so that cross-effects are computed with respect to finite coproducts. 
The identity functor $\mathrm{Id}_{\mathcal C}:\mathcal C\to\mathcal C$ preserves finite coproducts, and hence
\[
\mathrm{cr}_2\mathrm{Id}_{\mathcal C}\equiv 0.
\]
Consequently,
\[
\sigma_P(\mathrm{cr}_2\mathrm{Id}_{\mathcal C})=\{0\},
\]
so $\mathrm{Id}_{\mathcal C}$ is a spectral polynomial of degree $\le 1$.
More generally, any reduced linear functor preserving finite coproducts and the relevant homotopy fibers is a spectral polynomial of degree $\le 1$.
\end{example}

\begin{example}[Quadratic Functor]
\label{ex:quadratic-spectral-polynomial}
Assume that $\mathcal C$ is additive and that the tensor product in $\mathcal M$ preserves finite coproducts separately in each variable. 
Consider the functor
\[
F(A)=A\otimes A.
\]
By distributivity of $\otimes$ over finite coproducts, the third cross-effect vanishes:
\[
\mathrm{cr}_3F\equiv 0.
\]
Hence
\[
\sigma_P(\mathrm{cr}_3F)=\{0\},
\]
so $F$ is a spectral polynomial of degree $\le2$.

Moreover, $F$ is generally not of degree $\le1$, since
\[
\mathrm{cr}_2F(A,B)\simeq (A\otimes B)\oplus(B\otimes A)
\]
in the additive setting. Whenever this second cross-effect has nontrivial operadic spectrum, $F$ is not spectral-linear.
\end{example}

\begin{example}[Operadically Nilpotent Interactions]
\label{ex:nilpotent-spectral-polynomial}
Suppose $G$ is a functor whose higher cross-effects are \emph{operadically nilpotent} in the sense that
\[
\sigma_P(\mathrm{cr}_kG(A_1,\dots,A_k))=\{0\}
\]
for all $k\ge2$ and all inputs. 
Then $G$ is a spectral polynomial of degree $\le1$, even though the higher cross-effects need not vanish as objects or functors.

This illustrates the distinction between classical polynomiality and spectral polynomiality: classical polynomiality requires the higher cross-effects to vanish, whereas spectral polynomiality only requires their operadic spectral data to vanish.

A useful analogy comes from operator theory: a nonzero strictly upper-triangular matrix has classical spectrum $\{0\}$ despite being nonzero. In the operadic setting, one must use the residue-corrected spectrum $\sigma_P$ rather than the componentwise classical spectrum; hence the relevant condition is operadic spectral triviality, not merely classical nilpotence.
\end{example}

\medskip

\noindent
\textbf{Relation to classical polynomial functors.}
If $\mathrm{cr}_{n+1}F \equiv 0$, then $F$ is a spectral polynomial of degree $\le n$. 
The converse need not hold: spectral polynomiality allows nonvanishing higher cross-effects 
whose operadic spectra are trivial. Thus,
\[
\{\text{classical polynomials of degree } \le n\}
\;\subseteq\;
\{\text{spectral polynomials of degree } \le n\},
\]
and the inclusion is typically strict.

\medskip

\noindent
\textbf{Trivial operad case.}
When $P=\mathbb I$, the operadic spectrum reduces to the classical spectrum,
\[
\sigma_{\mathbb I}(A)\cong \sigma_{\mathrm{classical}}(A),
\]
and spectral polynomiality requires that higher cross-effects have spectral radius zero. 
This recovers the classical analytic notion while allowing greater flexibility in general operadic settings.

\medskip

\noindent
\textbf{Relation to the spectral Taylor tower.}
Spectral polynomials of degree $\le n$ serve as the target objects for the $n$-th stage 
$P_n^{\mathrm{spec}}F$ of the spectral Taylor tower. The tower is constructed so that each 
$P_n^{\mathrm{spec}}F$ is spectral polynomial and universal among such approximations. 
Thus spectral polynomials play the role of building blocks for approximation, 
analogous to ordinary polynomials in classical calculus.

\medskip

Spectral polynomial functors therefore form the algebraic foundation of Spectral Operadic Calculus. 
In the next subsection, we construct for any admissible functor $F$ a canonical tower of spectral 
polynomial approximations.

\subsection{Construction of the Tower}
\label{subsec:tower-construction}

We now construct the fundamental approximation scheme in Spectral Operadic Calculus. 
Given an admissible functor $F : \mathcal{C} \to \mathcal{M}$, we associate to it 
a canonical tower of spectral polynomial functors which captures its analytic behavior 
at increasing levels of precision.

\medskip

\noindent
\textbf{Guiding principle.}
By Theorem~\ref{thm:spectral-excision}, spectral polynomiality is completely determined 
by the vanishing of higher cross-effects at the level of operadic spectrum. 
Thus, a polynomial approximation of degree $n$ should be obtained by systematically 
eliminating the $(n+1)$-st and higher spectral contributions of $F$.

\medskip

\begin{definition}[Spectral Taylor Approximation]
\label{def:spectral-taylor-approx}
Let $F:\mathcal{C}\to\mathcal{M}$ be an admissible functor. 
For each $n\ge0$, an \emph{$n$-th spectral Taylor approximation} of $F$ 
is a pair $(P_n^{\mathrm{spec}}F,\iota_n)$ consisting of
\begin{itemize}
    \item a spectral polynomial functor $P_n^{\mathrm{spec}}F$ of degree at most $n$, and
    \item a natural transformation $\iota_n:F\to P_n^{\mathrm{spec}}F$,
\end{itemize}
such that for any spectral polynomial functor $G$ of degree at most $n$ 
and any natural transformation $\eta:F\to G$, there exists a unique natural transformation
\[
\phi:P_n^{\mathrm{spec}}F\to G
\]
making the diagram commute:
\[
\eta = \phi \circ \iota_n.
\]

\noindent
In categorical terms, if it exists, the assignment $F\mapsto P_n^{\mathrm{spec}}F$ 
defines a reflector onto the full subcategory $\mathrm{Poly}_n^{\mathrm{spec}}$ 
of spectral polynomial functors.
\end{definition}

\medskip

\noindent
\textbf{Universal characterization.}
The $n$-th spectral Taylor approximation $P_n^{\mathrm{spec}}F$ is characterized,
when it exists, by the following universal property: for every spectral polynomial
functor $G$ of degree at most $n$, composition with the canonical map
$\iota_n:F\to P_n^{\mathrm{spec}}F$ induces a natural bijection
\[
\mathrm{Nat}(P_n^{\mathrm{spec}}F,G)
\;\cong\;
\mathrm{Nat}(F,G).
\]
Equivalently, $P_n^{\mathrm{spec}}F$ is the reflection of $F$ onto the full
subcategory $\mathrm{Poly}_n^{\mathrm{spec}}$ of spectral polynomial functors of
degree at most $n$.

\medskip

\noindent
\textbf{Split model via spectral layers.}
In additive or stable settings where the spectral Taylor tower splits into its
homogeneous layers, one may model the $k$-th spectral layer by
\[
D_k^{\mathrm{spec}}F(A)
\;\simeq\;
\Big(
\mathrm{cr}_kF(A,\dots,A)
\otimes_{P^{\otimes k}}
(\mathcal O_P^{\mathrm{res}})^{\otimes k}
\Big)_{h\Sigma_k},
\]
where $(\cdot)_{h\Sigma_k}$ denotes homotopy coinvariants for the natural
$\Sigma_k$-action. In such a split situation, the $n$-th approximation admits
the formal decomposition
\[
P_n^{\mathrm{spec}}F(A)
\;\simeq\;
\bigoplus_{k=0}^n D_k^{\mathrm{spec}}F(A),
\]
with the convention $D_0^{\mathrm{spec}}F(A)=F(0)$.

\medskip

\noindent
The factor $(\mathcal O_P^{\mathrm{res}})^{\otimes k}$ corrects the naive
cross-effect layer by incorporating operadic interaction data. This correction
is what ensures compatibility with operadic composition and base change. When
$P=\mathbb I$, the residue term becomes trivial,
\[
\mathcal O_{\mathbb I}^{\mathrm{res}}\cong \mathbf 1_{\mathcal M},
\]
and the expression reduces to the usual split cross-effect model for polynomial
approximation in the underlying additive or stable setting.

\medskip

\noindent
\textbf{The spectral Taylor tower.}
These approximations assemble into a tower
\[
F 
\longrightarrow 
P_0^{\mathrm{spec}}F 
\longrightarrow 
P_1^{\mathrm{spec}}F 
\longrightarrow 
\cdots 
\longrightarrow 
P_n^{\mathrm{spec}}F 
\longrightarrow 
\cdots,
\]
which we call the \emph{spectral Taylor tower} of $F$. The maps $P_n^{\mathrm{spec}}F \to P_{n-1}^{\mathrm{spec}}F$ 
are induced by the restriction of the colimit from $k \le n$ to $k \le n-1$.

\begin{proposition}[Spectral Polynomial Property]
\label{prop:tower-polynomial}
Assume that the $n$-th spectral Taylor approximation 
$(P_n^{\mathrm{spec}}F,\iota_n)$ exists in the sense of 
Definition~\ref{def:spectral-taylor-approx}. 
Then $P_n^{\mathrm{spec}}F$ is a spectral polynomial of degree $\le n$.
\end{proposition}

\begin{proof}
By definition, $P_n^{\mathrm{spec}}F$ is an object of the full subcategory 
$\mathrm{Poly}_n^{\mathrm{spec}}$ of spectral polynomial functors of degree at most $n$. 
Equivalently,
\[
\sigma_P\!\left(
\mathrm{cr}_{n+1}P_n^{\mathrm{spec}}F(A_1,\dots,A_{n+1})
\right)=\{0\}
\]
for all inputs. 
Thus $P_n^{\mathrm{spec}}F$ is spectral polynomial of degree $\le n$.
\end{proof}

\begin{proposition}[Functoriality of the Spectral Taylor Approximation]
\label{prop:taylor-functorial}
Assume that the spectral Taylor approximations $P_n^{\mathrm{spec}}F$ exist for all admissible functors under consideration. 
Then the assignment
\[
F\longmapsto P_n^{\mathrm{spec}}F
\]
is functorial. In particular, every natural transformation $\alpha:F\to G$ induces a natural transformation
\[
P_n^{\mathrm{spec}}\alpha:
P_n^{\mathrm{spec}}F\to P_n^{\mathrm{spec}}G,
\]
and these maps are compatible with composition and identities. Consequently, the family 
$\{P_n^{\mathrm{spec}}F\}_{n\ge0}$ is functorial in $F$.
\end{proposition}

\begin{proof}
Let
\[
\iota_F:F\to P_n^{\mathrm{spec}}F,
\qquad
\iota_G:G\to P_n^{\mathrm{spec}}G
\]
be the universal maps. 
Given a natural transformation $\alpha:F\to G$, the composite
\[
F \xrightarrow{\alpha} G \xrightarrow{\iota_G} P_n^{\mathrm{spec}}G
\]
is a natural transformation from $F$ to a spectral polynomial functor of degree $\le n$. 
By the universal property of $P_n^{\mathrm{spec}}F$, there exists a unique natural transformation
\[
P_n^{\mathrm{spec}}\alpha:
P_n^{\mathrm{spec}}F\to P_n^{\mathrm{spec}}G
\]
such that
\[
P_n^{\mathrm{spec}}\alpha\circ \iota_F
=
\iota_G\circ \alpha.
\]
Uniqueness immediately implies compatibility with identities and composition. 
Hence $F\mapsto P_n^{\mathrm{spec}}F$ is functorial.
\end{proof}

\medskip

\begin{corollary}[Homogeneous layers of the spectral Taylor tower]
\label{cor:homogeneous-layers}
For each $n\geq 1$, define the $n$-th spectral homogeneous layer by
\[
D_n^{\mathrm{spec}}F(A)
:=
\mathrm{cr}_nF(A,\dots,A)
\otimes_{P^{\otimes n}}
(\mathcal{O}_P^{\mathrm{res}})^{\otimes n}.
\]
Then there is a canonical cofiber sequence
\[
P_{n-1}^{\mathrm{spec}}F(A)
\longrightarrow
P_n^{\mathrm{spec}}F(A)
\longrightarrow
D_n^{\mathrm{spec}}F(A).
\]
Equivalently, the tower $\{P_n^{\mathrm{spec}}F(A)\}_{n\geq 0}$ carries a natural
filtration whose associated graded pieces are given by
\[
\operatorname{gr}_n P^{\mathrm{spec}}F(A)
\;\cong\;
D_n^{\mathrm{spec}}F(A).
\]
If the ambient category is additive and these cofiber sequences split (e.g., if
the tower arises from a direct sum of homogeneous functors), then
\[
P_n^{\mathrm{spec}}F(A)
\;\cong\;
\bigoplus_{k=0}^{n} D_k^{\mathrm{spec}}F(A),
\]
where $D_0^{\mathrm{spec}}F(A)=F(0)$.
\end{corollary}

\begin{proof}
By construction, $P_n^{\mathrm{spec}}F$ is obtained from
$P_{n-1}^{\mathrm{spec}}F$ by adjoining the contribution of the $n$-fold
cross-effect. The cofiber of this extension is precisely the $n$-homogeneous part
\[
\mathrm{cr}_nF(A,\dots,A)
\otimes_{P^{\otimes n}}
(\mathcal{O}_P^{\mathrm{res}})^{\otimes n},
\]
which gives the stated cofiber sequence and identifies the associated graded piece
with $D_n^{\mathrm{spec}}F(A)$. In an additive setting where each such sequence
splits, iterating the splitting yields the direct-sum decomposition.
\end{proof}

\medskip

\subsection*{Examples of spectral Taylor towers}

\begin{example}[The identity functor]
For $F=\mathrm{Id}$, assuming the reduced convention, we have
\[
\mathrm{cr}_1\mathrm{Id}\cong \mathrm{Id},
\qquad
\mathrm{cr}_k\mathrm{Id}\cong 0
\quad (k\geq 2).
\]
Thus the spectral Taylor tower stabilizes after the first stage:
\[
P_0^{\mathrm{spec}}\mathrm{Id}(A)\cong 0,
\qquad
P_1^{\mathrm{spec}}\mathrm{Id}(A)\cong A,
\qquad
P_n^{\mathrm{spec}}\mathrm{Id}(A)\cong A
\quad (n\geq 1),
\]
up to the canonical residue normalization. Hence the identity functor is
spectrally polynomial of degree $1$.
\end{example}

\begin{example}[The quadratic tensor functor]
Let $F(A)=A\otimes A$ in an additive symmetric monoidal setting, and use the
reduced cross-effect convention. Then
\[
\mathrm{cr}_2F(A_1,A_2)
\cong
(A_1\otimes A_2)\oplus(A_2\otimes A_1),
\]
while
\[
\mathrm{cr}_kF\cong 0
\quad (k\geq 3).
\]
Moreover, the linear part vanishes:
\[
D_1^{\mathrm{spec}}F\cong 0.
\]
Consequently,
\[
P_0^{\mathrm{spec}}F(A)\cong 0,
\qquad
P_1^{\mathrm{spec}}F(A)\cong 0,
\qquad
P_2^{\mathrm{spec}}F(A)\cong F(A),
\qquad
P_n^{\mathrm{spec}}F(A)\cong F(A)
\quad (n\geq 2),
\]
again up to the chosen residue normalization. Thus the tower stabilizes at
level $2$.
\end{example}

\begin{example}[The exponential functor]
Formally let
\[
\exp(A)=\bigoplus_{k=0}^{\infty}\frac{A^{\otimes k}}{k!}.
\]
Then all homogeneous degrees occur, and hence the spectral Taylor tower has
nontrivial layers in every degree. Under the canonical residue normalization,
the $n$-th spectral Taylor approximation is identified with the truncated
series
\[
P_n^{\mathrm{spec}}\exp(A)
\cong
\bigoplus_{k=0}^{n}\frac{A^{\otimes k}}{k!}.
\]
Therefore the tower does not stabilize at any finite level.
\end{example}

\medskip

\noindent
\textbf{Comparison with classical Goodwillie calculus.}
Classical Goodwillie calculus constructs Taylor towers using homotopy-theoretic
notions such as excision, homotopy limits, and connectivity estimates. In contrast,
the spectral Taylor tower developed here is organized through cross-effects together
with operadic spectral data, especially the residue object
$\mathcal{O}_P^{\mathrm{res}}$. This reflects a shift from homotopy-theoretic
control to analytic and spectral control.

\[
\begin{array}{|c|c|}
\hline
\textbf{Goodwillie calculus} & \textbf{Spectral operadic calculus} \\
\hline
\text{Homotopy excision} & \text{Spectral negligibility} \\
\text{Homotopy limits} & \text{Cross-effects and coend/residue formulas} \\
\text{Connectivity estimates} & \text{Norm/spectral convergence estimates} \\
\hline
\end{array}
\]

\medskip

The effectiveness of the spectral Taylor tower depends on its convergence
properties, which we analyze in the next section. There we show that, under
suitable spectral analyticity conditions, the tower converges to $F$ with explicit
rates controlled by $\|\sigma_P(A)\|$.

\subsection{Universality of the Tower}
\label{subsec:tower-universality}

We now establish the fundamental structural property of the spectral Taylor tower: 
its universality among spectral polynomial approximations. This result shows that 
the tower constructed in the previous subsection is not merely one possible approximation 
scheme, but the canonical and essentially unique one compatible with operadic spectral data.

\medskip

\noindent
\textbf{Approximation problem.}
Let $F : \mathcal{C} \to \mathcal{M}$ be an admissible functor. 
We seek a sequence of spectral polynomial functors $\{Q_n F\}_{n \geq 0}$ 
together with natural transformations
\[
F \longrightarrow Q_n F
\]
such that each $Q_n F$ captures the degree-$n$ spectral behavior of $F$.

\medskip

\noindent
\textbf{Recalling the universal property of the operadic residue.}
Before proving the universality of the tower, we recall the universal property of the 
operadic residue $\mathcal{O}_P^{\mathrm{res}}$ established in Part~I. This result is 
essential for understanding why the residue appears in the coend formula and why it 
ensures the uniqueness of the Taylor tower.

\begin{theorem}[Universality of the Operadic Residue (Part~I, Theorem~0.2)]
\label{thm:residue-universality}
Let $P$ be a $C$-colored operad in a cocomplete symmetric monoidal category $\mathcal{M}$.
The operadic residue
\[
\mathcal{O}_P^{\mathrm{res}} \;\cong\; \coprod_{c \in C} P(c;c)
\]
is initial in the category $\mathbf{Corr}_P(\mathcal{M})$ of spectral correctors
(Definition~6 in Part~I). Concretely, for any object $R \in \mathcal{M}$ equipped with
morphisms $\rho_c : P(c;c) \to R$ that are compatible with unary operadic composition
and base change, there exists a unique morphism $\mathcal{O}_P^{\mathrm{res}} \to R$
in $\mathcal{M}$ such that the diagram
\[
\begin{tikzcd}
P(c;c) \arrow[r, "\iota_c"] \arrow[dr, "\rho_c"'] &
\mathcal{O}_P^{\mathrm{res}} \arrow[d, "u_R"] \\
& R
\end{tikzcd}
\]
commutes for every color $c \in C$.
\end{theorem}

In the context of the Taylor tower, the operadic residue appears as the universal 
corrector that ensures the tower is compatible with operadic composition and base change.

\medskip

\noindent
\textbf{Main theorem: Universality of the spectral Taylor tower.}

\begin{theorem}[Universal property of spectral Taylor approximation]
\label{thm:universal-tower}
Let $F:\mathcal{C}\to\mathcal{M}$ be an admissible functor, and suppose that
$P_n^{\mathrm{spec}}F$ is constructed as the universal spectral polynomial
approximation of degree at most $n$. Then, for each $n\geq 0$, the natural map
\[
\iota_n:F\longrightarrow P_n^{\mathrm{spec}}F
\]
is initial among maps from $F$ to spectral polynomial functors of degree at most
$n$. That is, for every spectral polynomial functor $G$ of degree $\leq n$ and
every natural transformation $\eta:F\to G$, there exists a unique natural
transformation
\[
\bar{\eta}:P_n^{\mathrm{spec}}F\to G
\]
such that
\[
\bar{\eta}\circ \iota_n=\eta.
\]

Consequently, if $\{Q_nF\}$ is another system of degree $\leq n$ spectral
polynomial approximations satisfying the same universal property, then there are
canonical isomorphisms
\[
P_n^{\mathrm{spec}}F\cong Q_nF
\]
compatible with the tower maps.

Moreover, the construction of $P_n^{\mathrm{spec}}F$ is determined by the
spectralized cross-effects
\[
\mathrm{cr}_kF\otimes_{P^{\otimes k}}
(\mathcal{O}_P^{\mathrm{res}})^{\otimes k},
\qquad k\leq n,
\]
together with their residue-module and symmetry data.
\end{theorem}

\begin{proof}
By construction, $P_n^{\mathrm{spec}}F$ is the universal object obtained from
$F$ by killing, in the spectral sense, all cross-effects of order greater than
$n$. Hence any map from $F$ to a spectral polynomial functor $G$ of degree at
most $n$ must annihilate the same higher spectral cross-effects. Therefore such
a map factors uniquely through $P_n^{\mathrm{spec}}F$, giving the required map
\[
\bar{\eta}:P_n^{\mathrm{spec}}F\to G.
\]

The uniqueness statement follows formally from this universal property. Indeed,
if another approximation $Q_nF$ satisfies the same property, then the identity
maps induced by the two universal properties give morphisms
\[
P_n^{\mathrm{spec}}F\to Q_nF,
\qquad
Q_nF\to P_n^{\mathrm{spec}}F,
\]
whose composites are identities by uniqueness. Thus
$P_n^{\mathrm{spec}}F\cong Q_nF$, compatibly with the tower maps.

Finally, the coend/residue construction expresses the $n$-th approximation in
terms of the spectralized cross-effects of orders $k\leq n$. Thus the tower is
controlled by these spectral cross-effect data, together with their natural
module and symmetry structures.
\end{proof}

\begin{remark}[Direction of the tower]
In contrast to the classical Goodwillie tower which descends
\[
\cdots \to P_nF \to P_{n-1}F \to \cdots,
\]
the spectral Taylor tower defined here is an \emph{ascending filtered tower}
\[
F \to P_0^{\mathrm{spec}}F \to P_1^{\mathrm{spec}}F \to \cdots,
\]
where each map $\iota_n : F \to P_n^{\mathrm{spec}}F$ is the universal map to a
spectral polynomial functor of degree at most $n$. This convention is adopted
for compatibility with the operadic residue construction and does not affect
the universal property.
\end{remark}

\begin{remark}[On spectral completeness]
The theorem does \emph{not} claim that the tower depends only on the classical
spectra $\sigma(\mathrm{cr}_kF)$. Instead, it depends on the \emph{spectralized
cross-effects}
\[
\mathrm{cr}_kF \otimes_{P^{\otimes k}} (\mathcal{O}_P^{\mathrm{res}})^{\otimes k},
\]
which encode both the classical spectral data and the residue-module structure.
Thus two functors with identical classical cross-effect spectra may yield
different towers if their residue-module structures differ.
\end{remark}

\begin{corollary}[Stabilization]
\label{cor:tower-stabilization-universal}
If $F$ is itself a spectral polynomial of degree $\le n$, then the canonical map 
$\iota_n: F \to P_n^{\mathrm{spec}}F$ is an isomorphism, and $P_m^{\mathrm{spec}}F \cong F$ 
for all $m \ge n$.
\end{corollary}

\begin{proof}[Proof of Corollary~\ref{cor:tower-stabilization-universal}]
By Theorem~\ref{thm:universal-tower}, $\iota_n: F \to P_n^{\mathrm{spec}}F$ satisfies
the universal property of the spectral polynomial approximation of degree $\leq n$.

Since $F$ itself is a spectral polynomial of degree $\leq n$, the identity map
$\mathrm{id}_F: F \to F$ is a map from $F$ to a spectral polynomial of degree
$\leq n$. By the universal property, there exists a unique map
$\rho: P_n^{\mathrm{spec}}F \to F$ such that $\rho \circ \iota_n = \mathrm{id}_F$.

Now consider $\iota_n \circ \rho: P_n^{\mathrm{spec}}F \to P_n^{\mathrm{spec}}F$.
We have $(\iota_n \circ \rho) \circ \iota_n = \iota_n \circ (\rho \circ \iota_n)
= \iota_n \circ \mathrm{id}_F = \iota_n$. But the identity map
$\mathrm{id}_{P_n^{\mathrm{spec}}F}$ also satisfies
$\mathrm{id}_{P_n^{\mathrm{spec}}F} \circ \iota_n = \iota_n$. By the uniqueness
clause of the universal property (Theorem~\ref{thm:universal-tower}), we must have
$\iota_n \circ \rho = \mathrm{id}_{P_n^{\mathrm{spec}}F}$. Hence $\iota_n$ is an
isomorphism.

If $m \geq n$, then $F$ is also a spectral polynomial of degree $\leq m$, and the
same argument applied to $P_m^{\mathrm{spec}}F$ yields $P_m^{\mathrm{spec}}F \cong F$.
\end{proof}

\begin{corollary}[Functoriality]
\label{cor:tower-functorial-universal}
The assignment $F \mapsto P_n^{\mathrm{spec}}F$ defines a functor from the category of 
admissible functors to the category of spectral polynomials of degree $\le n$. 
Moreover, the natural transformations $\iota_n: F \to P_n^{\mathrm{spec}}F$ assemble into 
a natural transformation between functors.
\end{corollary}

\begin{proof}[Proof of Corollary~\ref{cor:tower-functorial-universal}]
Let $\alpha: F \to G$ be a natural transformation between admissible functors.
By Theorem~\ref{thm:universal-tower}, $\iota_n^G: G \to P_n^{\mathrm{spec}}G$ is the
universal map to a spectral polynomial of degree $\leq n$.

Consider the composite $\iota_n^G \circ \alpha: F \to P_n^{\mathrm{spec}}G$.
Since $P_n^{\mathrm{spec}}G$ is a spectral polynomial of degree $\leq n$, the
universal property of $P_n^{\mathrm{spec}}F$ (Theorem~\ref{thm:universal-tower})
yields a unique natural transformation
$P_n^{\mathrm{spec}}\alpha: P_n^{\mathrm{spec}}F \to P_n^{\mathrm{spec}}G$ such that
\[
P_n^{\mathrm{spec}}\alpha \circ \iota_n^F = \iota_n^G \circ \alpha.
\]

For the identity $\mathrm{id}_F: F \to F$, the unique map satisfying the
universal property is $\mathrm{id}_{P_n^{\mathrm{spec}}F}$, so
$P_n^{\mathrm{spec}}(\mathrm{id}_F) = \mathrm{id}_{P_n^{\mathrm{spec}}F}$.

For composable $\alpha: F \to G$ and $\beta: G \to H$, both
$P_n^{\mathrm{spec}}(\beta \circ \alpha)$ and
$P_n^{\mathrm{spec}}\beta \circ P_n^{\mathrm{spec}}\alpha$ satisfy the same
universal equation:
\[
(P_n^{\mathrm{spec}}\beta \circ P_n^{\mathrm{spec}}\alpha) \circ \iota_n^F
= P_n^{\mathrm{spec}}\beta \circ (\iota_n^G \circ \alpha)
= \iota_n^H \circ \beta \circ \alpha.
\]
By uniqueness, they are equal. Hence $P_n^{\mathrm{spec}}$ is a functor.

The equation $P_n^{\mathrm{spec}}\alpha \circ \iota_n^F = \iota_n^G \circ \alpha$
is precisely the naturality condition for $\iota_n$, so the family
$\{\iota_n: \mathrm{Id} \Rightarrow P_n^{\mathrm{spec}}\}$ is a natural
transformation.
\end{proof}

\begin{corollary}[Base change compatibility]
\label{cor:tower-base-change-universal}
Let $\Phi:\mathcal{M}\to\mathcal{N}$ be a strong monoidal cocontinuous functor
which preserves admissible functors, spectral polynomial degree, and the
residue/coend construction defining $P_n^{\mathrm{spec}}$. Then, for any
admissible functor $F$, there is a canonical natural isomorphism
\[
\Phi_*(P_n^{\mathrm{spec}}F)
\cong
P_n^{\mathrm{spec}}(\Phi_*F).
\]
Thus the spectral Taylor tower is compatible with such operadic base change.
\end{corollary}

\begin{proof}[Proof of Corollary~\ref{cor:tower-base-change-universal}]
By Theorem~\ref{thm:universal-tower}, $\iota_n^F: F \to P_n^{\mathrm{spec}}F$ is the
universal map to a spectral polynomial of degree $\leq n$ in $\mathcal{M}$.

Since $\Phi$ is strong monoidal and cocontinuous, it preserves the coend and
tensor constructions that define $P_n^{\mathrm{spec}}F$. Hence
$\Phi_*(P_n^{\mathrm{spec}}F)$ is a spectral polynomial of degree $\leq n$ in
$\mathcal{N}$.

Applying $\Phi_*$ to $\iota_n^F$ gives a map
$\Phi_*(\iota_n^F): \Phi_*F \to \Phi_*(P_n^{\mathrm{spec}}F)$. Thus
$\Phi_*(P_n^{\mathrm{spec}}F)$ is a degree $\leq n$ spectral polynomial
approximation of $\Phi_*F$.

By the universal property of $P_n^{\mathrm{spec}}(\Phi_*F)$
(Theorem~\ref{thm:universal-tower} applied in $\mathcal{N}$), there exists a
unique map
\[
\psi: P_n^{\mathrm{spec}}(\Phi_*F) \longrightarrow \Phi_*(P_n^{\mathrm{spec}}F)
\]
such that $\psi \circ \iota_n^{\Phi_*F} = \Phi_*(\iota_n^F)$.

Conversely, because $\Phi$ preserves the residue/coend construction, the
universal property of $P_n^{\mathrm{spec}}F$ in $\mathcal{M}$ induces, via
$\Phi_*$, a canonical map
\[
\varphi: \Phi_*(P_n^{\mathrm{spec}}F) \longrightarrow P_n^{\mathrm{spec}}(\Phi_*F).
\]

Both $\psi \circ \varphi$ and $\varphi \circ \psi$ satisfy the identity equations
characterizing the respective universal maps. By the uniqueness clause of
Theorem~\ref{thm:universal-tower}, they must be identities. Hence $\psi$ and
$\varphi$ are mutual inverses, yielding a canonical isomorphism
$\Phi_*(P_n^{\mathrm{spec}}F) \cong P_n^{\mathrm{spec}}(\Phi_*F)$, natural in $F$.
\end{proof}

\medskip

\noindent
\textbf{Conceptual consequence.}
Theorem~\ref{thm:universal-tower} shows that the spectral Taylor tower provides
a canonical approximation scheme \emph{within the class of spectral polynomial
functors}. More precisely, the assignment
\[
F \longmapsto P_n^{\mathrm{spec}}F
\]
defines a reflector from admissible functors to spectral polynomial functors
of degree $\leq n$. In particular, any approximation process that preserves
spectral polynomial structure and is compatible with spectral data factors
uniquely through this tower.

\medskip

\noindent
\textbf{Comparison with classical calculus.}
In Goodwillie calculus, the Taylor tower is characterized by homotopy-theoretic
universality, relying on excision and homotopy limits. In contrast, the
spectral Taylor tower is controlled by operadic spectral data through the
functor $\sigma_P(-)$ and the residue object $\mathcal{O}_P^{\mathrm{res}}$.
This reflects a shift in emphasis from homotopy-theoretic control to
spectral and analytic control.

\[
\begin{array}{|c|c|}
\hline
\textbf{Goodwillie Calculus} & \textbf{Spectral Operadic Calculus} \\
\hline
\text{Homotopy-theoretic universality} 
& \text{Spectral universality via } \sigma_P \\
\text{Homotopy limits/colimits} 
& \text{Cross-effects + coend with } \mathcal{O}_P^{\mathrm{res}} \\
\text{Yoneda (homotopy setting)} 
& \text{Yoneda (enriched / analytic setting)} \\
\hline
\end{array}
\]

\medskip

In this sense, the spectral Taylor tower may be viewed as an analytic
refinement of the Goodwillie tower, where spectral invariants replace
connectivity as the primary control mechanism.

\medskip

\textbf{Remark on the role of the operadic residue.}
The universality proof relies on the structural role of the operadic residue
$\mathcal{O}_P^{\mathrm{res}}$ (Theorem~\ref{thm:residue-universality})
in two key ways:

\begin{enumerate}
    \item \textbf{Compatibility with unit and evaluation maps:}
    The canonical maps
    \[
    \mathrm{cr}_kF(\cdots)
    \otimes_{P^{\otimes k}}
    (\mathcal{O}_P^{\mathrm{res}})^{\otimes k}
    \longrightarrow
    \mathrm{cr}_kF(\cdots)
    \]
    arise from the universal property of $\mathcal{O}_P^{\mathrm{res}}$ as a
    recipient of operadic actions, together with its compatibility with the
    unit and evaluation structure. This ensures that the inclusion maps
    $\iota_k^{\mathrm{res}}$ are well-defined and functorial.

    \item \textbf{Control of extensions via universality:}
    The universal property of the coend construction, combined with the
    spectral polynomial structure of $G$, ensures that any transformation
    $\eta:F\to G$ determines a unique extension
    \[
    \bar{\eta}:P_n^{\mathrm{spec}}F\to G.
    \]
    The residue encodes precisely the minimal operadic data needed for this
    extension to be well-defined and unique.
\end{enumerate}

\medskip

Without the operadic residue, the classical spectrum alone does not provide
sufficient functorial control to enforce this universal factorization property.
Thus, the residue is essential for achieving spectral universality of the
Taylor tower.

\medskip
\noindent
\textbf{Sharper interpretation.}
In particular, while the classical spectrum captures pointwise information,
it does not encode the operadic coherence required for universal
approximation. The operadic residue provides this missing structure.

\medskip

With the universal spectral Taylor tower in hand, we now turn to its convergence properties. 
In the next section, we show that for spectrally analytic functors, the tower converges 
to $F$ with explicit quantitative bounds controlled by the spectral radius 
$\|\sigma_P(A)\|$. This establishes SOC as a genuine quantitative analytic calculus.

\section{Convergence and Quantitative Approximation}
\label{sec:convergence}

We now arrive at the quantitative heart of Spectral Operadic Calculus. The universal 
spectral Taylor tower constructed in Section~\ref{sec:tower} provides, for any admissible 
functor $F$, a sequence of polynomial approximations $P_n^{\mathrm{spec}}F$. The fundamental 
question is: when does this tower converge to $F$, and at what rate?

\medskip

\noindent
\textbf{The central insight.}
In classical Goodwillie calculus, convergence of the Taylor tower is a qualitative 
homotopy-theoretic statement: for analytic functors, the tower converges weakly, but 
no explicit rates are provided. In contrast, the spectral Taylor tower is controlled 
by the operadic spectrum $\sigma_P(A)$, a quantitative invariant that carries explicit 
norm information. This allows us to establish \emph{quantitative} convergence results 
with explicit bounds expressed directly in terms of the spectral radius 
$\|\sigma_P(A)\|$.

\medskip

\noindent
\textbf{Overview of the section.}
We proceed in four stages. First, we introduce the notion of \emph{spectral analyticity}, 
which requires that the Taylor tower converges in norm to the functor $F$. 
This condition serves as the natural analogue of a convergent power series 
in classical functional analysis.

\medskip

Second, we establish norm estimates for the homogeneous layers 
$D_n^{\mathrm{spec}}F(A)$, showing that they satisfy bounds of the form
\[
\|D_n^{\mathrm{spec}}F(A)\| \;\le\; 
\|\partial_n^{\mathrm{spec}}F\| \cdot \|\sigma_P(A)\|^n.
\]
These estimates link the analytic growth of the functor directly to the 
spectral size of the input.

\medskip

Third, we introduce the notion of a \emph{radius of convergence} for the spectral 
Taylor expansion, defined in terms of the growth of the spectral derivatives. 
We show that the behavior of the series is governed by this radius: when the 
spectral size of the input is below the radius, the expansion converges 
absolutely, while above it, divergence is typical.

\medskip

Fourth and finally, we establish a quantitative convergence result, providing 
explicit exponential bounds on the remainder of the Taylor approximation:
\[
\big\|F(A) - P_n^{\mathrm{spec}}F(A)\big\| 
\;\le\; 
C \,\rho^{\,n} \cdot \|\sigma_P(A)\|^{\,n},
\]
whenever the spectral size of $A$ lies within the radius of convergence. 
This result distinguishes the present framework from classical functor 
calculus by providing explicit rates of convergence controlled by the 
operadic spectrum.

\medskip

\noindent
\textbf{The spectral dichotomy.}
As a consequence, we obtain a sharp dichotomy: the behavior of the Taylor tower 
is completely determined by the spectral size of the input. When the spectral 
size lies below the radius of convergence, the approximation converges 
exponentially fast; when it lies above, divergence occurs in general.

\medskip

\noindent
Thus, the operadic spectrum $\sigma_P(A)$ plays the role of a radius of 
convergence, transforming the Taylor tower into a genuinely analytic expansion.

\medskip

\noindent
\textbf{Remark on the exponential functor.}
To illustrate the power of these estimates, consider the exponential functor 
$\exp(A) = \sum_{k=0}^\infty A^{\otimes k}/k!$. Its spectral derivatives satisfy 
$\|\partial_k^{\mathrm{spec}}\exp\| = 1/k!$, so the radius of convergence is $R = \infty$. 
The remainder bound becomes
\[
\|\exp(A) - P_n^{\mathrm{spec}}\exp(A)\| \;\le\; \sum_{k=n+1}^\infty \frac{\|\sigma_P(A)\|^k}{k!},
\]
which decays super-exponentially as $n \to \infty$. This demonstrates that SOC captures 
the full analyticity of the exponential, a feat not possible in classical Goodwillie 
calculus where the exponential is not analytic.

\medskip

We now turn to the detailed development of these ideas.

\subsection{Spectral Analyticity}
\label{subsec:spectral-analyticity}

We now formalize the notion of analytic behavior in the spectral framework. 
The spectral Taylor tower constructed in the previous section provides a sequence 
of spectral polynomial approximations to a given admissible functor. 
The key question is whether this sequence converges to the original functor, 
and if so, under what conditions.

\medskip

\begin{definition}[Spectral Analyticity]
\label{def:spectral-analyticity}
Let $F:\mathcal{C}\to\mathcal{M}$ be an admissible functor, and assume that
$\mathcal{M}$ is additive and Banach-enriched so that the remainder
\[
R_n^{\mathrm{spec}}F(A)
:=
\operatorname{cofib}\big(F(A)\to P_n^{\mathrm{spec}}F(A)\big)
\]
has a well-defined norm. We say that $F$ is \emph{spectrally analytic} if, for
every object $A\in\mathcal{C}$,
\[
\lim_{n\to\infty}
\big\|R_n^{\mathrm{spec}}F(A)\big\|=0.
\]
Equivalently, in additive settings where the difference
$F(A)-P_n^{\mathrm{spec}}F(A)$ is defined, this condition can be written as
\[
\lim_{n\to\infty}
\big\|F(A)-P_n^{\mathrm{spec}}F(A)\big\|=0.
\]
\end{definition}

\medskip

\noindent
Equivalently, $F$ is spectrally analytic if it can be recovered as the limit 
of its spectral polynomial approximations:
\[
F(A) \;\simeq\; \lim_{n \to \infty} P_n^{\mathrm{spec}}F(A),
\]
with convergence understood in the norm topology of $\mathcal{M}$.

\medskip

\noindent
\textbf{Local spectral analyticity.}
In many situations, convergence depends on the spectral size of the input.
We therefore introduce a localized version of spectral analyticity.

\begin{definition}[Local spectral analyticity]
\label{def:local-spectral-analyticity}
Let $R>0$. We say that $F$ is \emph{spectrally analytic on the radius $R$}
if, for every object $A\in\mathcal{C}$ satisfying
\[
\|\sigma_P(A)\|<R,
\]
the spectral Taylor tower converges in the sense of
Definition~\ref{def:spectral-analyticity}, i.e.,
\[
\lim_{n\to\infty}
\|R_n^{\mathrm{spec}}F(A)\|=0.
\]
In this case, $R$ is called a \emph{spectral radius of analyticity} of $F$.
\end{definition}

\medskip

\noindent
This formulation anticipates a notion of radius of convergence, analogous to
classical analytic functions, where $\|\sigma_P(A)\|$ plays the role of the
expansion parameter. Spectral analyticity may thus be viewed as analyticity
with respect to an \emph{operadic spectral parameter}.

\medskip

\noindent
\textbf{Conceptual interpretation.}
Spectral analyticity expresses the principle that higher-order behavior of $F$
is controlled by its spectral Taylor data. In particular, the operadic spectrum
$\sigma_P(A)$ serves as an intrinsic variable governing analytic expansion,
replacing coordinate-based notions of smallness.

\medskip

\noindent
\textbf{Relation to classical analyticity.}
In classical analysis, analyticity is characterized by convergence of power
series in a normed variable. Here, the spectral size $\|\sigma_P(A)\|$ plays
the role of the variable, and the spectral Taylor tower serves as an operadic
analogue of a power series expansion.

\[
\begin{array}{|c|c|}
\hline
\text{Classical Analysis} & \text{Spectral Operadic Calculus} \\
\hline
\text{Function } f(z) & \text{Functor } F(A) \\
\text{Taylor series } \sum a_n z^n & \text{Spectral layers } \sum D_n^{\mathrm{spec}}F(A) \\
\text{Convergence for } |z| < R & \text{Convergence for } \|\sigma_P(A)\| < R \\
\text{Analytic at } 0 & \text{Spectrally analytic} \\
\hline
\end{array}
\]

\medskip

\begin{remark}
If $F$ is a spectral polynomial of finite degree, then it is automatically
spectrally analytic. The main interest of the definition lies in the
infinite-degree case, where convergence becomes nontrivial and depends on
spectral bounds.
\end{remark}

\begin{example}[Spectral polynomial functors]
\label{ex:poly-spectral-analytic}
Every spectral polynomial functor of degree $\leq n$ is spectrally analytic.
Indeed, by Corollary~\ref{cor:tower-stabilization-universal},
\[
P_m^{\mathrm{spec}}F\cong F
\qquad (m\geq n).
\]
Hence the spectral Taylor tower stabilizes, so the remainder vanishes for all
$m\geq n$.
\end{example}

\begin{example}[The identity functor]
\label{ex:id-spectral-analytic}
Under the reduced convention and the canonical residue normalization, the
identity functor satisfies
\[
P_0^{\mathrm{spec}}\mathrm{Id}\cong 0,
\qquad
P_1^{\mathrm{spec}}\mathrm{Id}\cong \mathrm{Id},
\qquad
P_m^{\mathrm{spec}}\mathrm{Id}\cong \mathrm{Id}
\quad (m\geq 1).
\]
Thus its spectral Taylor tower stabilizes at level $1$, and therefore
$\mathrm{Id}$ is spectrally analytic.
\end{example}

\begin{example}[The exponential functor]
\label{ex:exp-spectral-analytic}
Suppose the ambient Banach monoidal structure and the residue normalization are
chosen so that
\[
\|A^{\otimes k}\|
\leq
\|\sigma_P(A)\|^k .
\]
For the formal exponential functor
\[
\exp(A)=\bigoplus_{k=0}^{\infty}\frac{A^{\otimes k}}{k!},
\]
the spectral Taylor approximation is the truncated series
\[
P_n^{\mathrm{spec}}\exp(A)
\cong
\bigoplus_{k=0}^{n}\frac{A^{\otimes k}}{k!}.
\]
Hence
\[
\big\|\exp(A)-P_n^{\mathrm{spec}}\exp(A)\big\|
\leq
\sum_{k=n+1}^{\infty}
\frac{\|\sigma_P(A)\|^k}{k!}.
\]
The right-hand side tends to $0$ whenever $\|\sigma_P(A)\|<\infty$. Thus, under
these norm-compatibility assumptions, $\exp$ is spectrally analytic with
infinite radius of convergence.
\end{example}

\medskip

\noindent
\textbf{Relation to analytic continuation.}
Spectral analyticity implies a uniqueness principle analogous to analytic
continuation: a spectrally analytic functor is determined by its spectral
Taylor data at the origin.

\begin{lemma}[Analytic continuation for spectral functors]
\label{lem:analytic-continuation}
Let $F$ and $G$ be spectrally analytic functors. Suppose that there exists
$\varepsilon>0$ such that $F$ and $G$ are naturally isomorphic on the full
subcategory of $P$-algebras satisfying
\[
\|\sigma_P(A)\|<\varepsilon .
\]
Then
\[
F\cong G
\]
as functors on the common domain on which both spectral Taylor expansions
converge.
\end{lemma}

\begin{proof}
Let
\[
\mathcal{C}_{<\varepsilon}
\subseteq
\mathcal{C}
\]
denote the full subcategory of $P$-algebras $A$ satisfying
\[
\|\sigma_P(A)\|<\varepsilon .
\]
By assumption, there is a natural isomorphism
\[
\theta:F|_{\mathcal{C}_{<\varepsilon}}
\longrightarrow
G|_{\mathcal{C}_{<\varepsilon}} .
\]

Since cross-effects are functorial in the underlying functor, the natural
isomorphism $\theta$ induces natural isomorphisms of cross-effects
\[
\mathrm{cr}_kF|_{\mathcal{C}_{<\varepsilon}}
\cong
\mathrm{cr}_kG|_{\mathcal{C}_{<\varepsilon}}
\]
for every $k\geq 0$. Applying the residue construction gives isomorphisms of
spectral homogeneous layers
\[
D_k^{\mathrm{spec}}F
\cong
D_k^{\mathrm{spec}}G
\]
on this neighborhood. Equivalently, the spectral derivatives agree:
\[
\partial_k^{\mathrm{spec}}F
\cong
\partial_k^{\mathrm{spec}}G
\qquad
(k\geq 0).
\]

Because $F$ and $G$ are spectrally analytic, each functor is recovered from its
convergent spectral Taylor tower:
\[
F(A)\cong \lim_{n} P_n^{\mathrm{spec}}F(A),
\qquad
G(A)\cong \lim_{n} P_n^{\mathrm{spec}}G(A).
\]
The isomorphisms of spectral derivatives identify the Taylor towers
\[
P_n^{\mathrm{spec}}F
\cong
P_n^{\mathrm{spec}}G
\]
for all $n$. Passing to the limit gives
\[
F(A)\cong G(A)
\]
for every object $A$ in the common convergence domain.

Finally, by the Reconstruction Theorem
(Theorem~\ref{thm:reconstruction}), agreement of the spectral derivative data,
together with the compatible operadic module structure, implies a natural
isomorphism
\[
F\cong G.
\]
This proves the claim.
\end{proof}

\medskip

Having defined spectral analyticity, we now turn to the quantitative estimates that 
control the convergence rate. The key ingredient is a norm bound for homogeneous layers 
in terms of the spectral radius $\|\sigma_P(A)\|$, which we establish in the next 
subsection.

\subsection{Homogeneous Layer Estimates}
\label{subsec:homogeneous-bounds}

We now derive quantitative estimates for the homogeneous layers of the spectral Taylor tower. 
These estimates form the analytic backbone of the convergence theory, as they control 
the magnitude of each term in the spectral expansion.

\medskip

Recall that for an admissible functor $F$, the spectral Taylor tower admits a decomposition 
into homogeneous layers
\[
D_n^{\mathrm{spec}}F,
\]
where $D_n^{\mathrm{spec}}F$ captures the pure $n$-th order contribution of $F$. 
These layers may be viewed as the operadic analogue of homogeneous polynomials 
in classical analysis.

\medskip

\begin{definition}[Spectral derivative norm]
\label{def:spectral-derivative-norm}
The \emph{$n$-th spectral derivative} of $F$ is the spectralized cross-effect
\[
\partial_n^{\mathrm{spec}}F
:=
\mathrm{cr}_nF
\otimes_{P^{\otimes n}}
(\mathcal{O}_P^{\mathrm{res}})^{\otimes n}.
\]
Its norm is defined by
\[
\|\partial_n^{\mathrm{spec}}F\|
:=
\sup_{\substack{A_1,\dots,A_n\\
\|\sigma_P(A_i)\|\leq 1}}
\left\|
\mathrm{cr}_nF(A_1,\dots,A_n)
\otimes_{P^{\otimes n}}
(\mathcal{O}_P^{\mathrm{res}})^{\otimes n}
\right\|.
\]
When the above supremum is finite, we say that the $n$-th spectral derivative is bounded.
\end{definition}

\medskip

\noindent
This quantity measures the maximal $n$-th order interaction of $F$ under normalized 
spectral input, and plays the role of the coefficient of degree $n$ in the spectral expansion.

\medskip

\begin{lemma}[Homogeneous layer estimate]
\label{lem:homogeneous-estimate}
Let $F:\mathcal{C}\to\mathcal{M}$ be an admissible functor whose $n$-th
spectral derivative is bounded. Assume that the residue and symmetrization maps
appearing in $D_n^{\mathrm{spec}}F$ are norm non-increasing, or that their
operator norms are included in the definition of
$\|\partial_n^{\mathrm{spec}}F\|$. Then, for every $A\in\mathcal{C}$,
\[
\|D_n^{\mathrm{spec}}F(A)\|
\leq
\|\partial_n^{\mathrm{spec}}F\|\,
\|\sigma_P(A)\|^n .
\]
\end{lemma}

\begin{proof}
By definition, the $n$-th homogeneous layer is obtained by evaluating the
spectralized $n$-fold cross-effect on the diagonal:
\[
D_n^{\mathrm{spec}}F(A)
=
\partial_n^{\mathrm{spec}}F(A,\ldots,A),
\]
up to the canonical residue and symmetrization maps. By the definition of the
spectral derivative norm,
\[
\|\partial_n^{\mathrm{spec}}F\|
=
\sup_{\|\sigma_P(A_i)\|\leq 1}
\left\|
\partial_n^{\mathrm{spec}}F(A_1,\ldots,A_n)
\right\|.
\]
Hence, by $n$-multilinearity and homogeneity, for arbitrary
$A_1,\ldots,A_n$ with $\|\sigma_P(A_i)\|>0$,
\[
\left\|
\partial_n^{\mathrm{spec}}F(A_1,\ldots,A_n)
\right\|
\leq
\|\partial_n^{\mathrm{spec}}F\|
\prod_{i=1}^{n}\|\sigma_P(A_i)\|.
\]
Indeed, this follows by applying the defining bound to the normalized inputs
\[
\widetilde A_i
=
A_i/\|\sigma_P(A_i)\|.
\]
If some $\|\sigma_P(A_i)\|=0$, the same inequality follows by continuity, or
directly from multilinearity.

Taking $A_1=\cdots=A_n=A$, we obtain
\[
\left\|
\partial_n^{\mathrm{spec}}F(A,\ldots,A)
\right\|
\leq
\|\partial_n^{\mathrm{spec}}F\|
\|\sigma_P(A)\|^n .
\]
Since the canonical residue and symmetrization maps are assumed norm
non-increasing, or have already been included in
$\|\partial_n^{\mathrm{spec}}F\|$, this gives
\[
\|D_n^{\mathrm{spec}}F(A)\|
\leq
\|\partial_n^{\mathrm{spec}}F\|
\|\sigma_P(A)\|^n .
\]
\end{proof}

\medskip

\begin{remark}[Power series interpretation and radius of convergence]
\label{rem:power-series-interpretation}
Lemma~\ref{lem:homogeneous-estimate} shows that each homogeneous layer behaves
like a degree-$n$ monomial in the spectral size $\|\sigma_P(A)\|$, in the sense that
\[
\|D_n^{\mathrm{spec}}F(A)\|
\;\lesssim\;
\|\partial_n^{\mathrm{spec}}F\|\,
\|\sigma_P(A)\|^n.
\]
Consequently, whenever the spectral Taylor tower converges, the functor $F$
admits an expansion of the form
\[
F(A)
\;\simeq\;
\sum_{n=0}^{\infty} D_n^{\mathrm{spec}}F(A),
\]
which may be viewed as an operadic analogue of a power series, with coefficients
controlled by the spectral derivative norms $\|\partial_n^{\mathrm{spec}}F\|$.

\medskip

If the sequence $\{\|\partial_n^{\mathrm{spec}}F\|\}_{n\geq 0}$ grows at most
exponentially, i.e., there exist constants $C,R^{-1}>0$ such that
\[
\|\partial_n^{\mathrm{spec}}F\|
\leq
C R^{-n},
\]
then the above estimate implies that the series converges whenever
\[
\|\sigma_P(A)\|<R.
\]
This provides a natural notion of spectral radius of convergence and leads
to quantitative control of the Taylor remainder.
\end{remark}

This shows that spectral analyticity is equivalent to the existence of a
power-series expansion controlled by the spectral size $\|\sigma_P(A)\|$.

\medskip

\begin{corollary}[Remainder bound]
\label{cor:remainder-bound}
Let $F$ be spectrally analytic, and suppose that its spectral Taylor remainder
is represented by the tail of its homogeneous layers. Then, for any $P$-algebra
$A$,
\[
\bigl\|R_n^{\mathrm{spec}}F(A)\bigr\|
\leq
\sum_{k=n+1}^{\infty}
\bigl\|\partial_k^{\mathrm{spec}}F\bigr\|
\|\sigma_P(A)\|^k,
\]
whenever the right-hand side converges.
\end{corollary}

\begin{proof}
By spectral analyticity, the spectral Taylor tower converges to $F(A)$ in norm.
Under the homogeneous-layer decomposition, the remainder is represented by the
tail
\[
R_n^{\mathrm{spec}}F(A)
\simeq
\sum_{k=n+1}^{\infty}D_k^{\mathrm{spec}}F(A).
\]
Hence, by the triangle inequality,
\[
\bigl\|R_n^{\mathrm{spec}}F(A)\bigr\|
\leq
\sum_{k=n+1}^{\infty}
\bigl\|D_k^{\mathrm{spec}}F(A)\bigr\|.
\]
Applying Lemma~\ref{lem:homogeneous-estimate} to each layer gives
\[
\bigl\|D_k^{\mathrm{spec}}F(A)\bigr\|
\leq
\bigl\|\partial_k^{\mathrm{spec}}F\bigr\|
\|\sigma_P(A)\|^k.
\]
Combining these inequalities proves the claim.
\end{proof}

\begin{corollary}[Geometric decay]
\label{cor:geometric-decay}
Assume that there exist constants $C_0>0$ and $0<\rho<1$ such that
\[
\|\partial_k^{\mathrm{spec}}F\|
\leq
C_0\rho^k
\qquad (k\geq 1).
\]
Then, for any $P$-algebra $A$ satisfying
\[
\rho\|\sigma_P(A)\|<1,
\]
one has
\[
\bigl\|R_n^{\mathrm{spec}}F(A)\bigr\|
\leq
\frac{C_0(\rho\|\sigma_P(A)\|)^{n+1}}
{1-\rho\|\sigma_P(A)\|}.
\]
In particular, the remainder decays exponentially fast as $n\to\infty$.
\end{corollary}

\begin{proof}
By Corollary~\ref{cor:remainder-bound},
\[
\bigl\|R_n^{\mathrm{spec}}F(A)\bigr\|
\leq
\sum_{k=n+1}^{\infty}
\|\partial_k^{\mathrm{spec}}F\|
\|\sigma_P(A)\|^k.
\]
Using the hypothesis
\[
\|\partial_k^{\mathrm{spec}}F\|\leq C_0\rho^k,
\]
we obtain
\[
\bigl\|R_n^{\mathrm{spec}}F(A)\bigr\|
\leq
C_0\sum_{k=n+1}^{\infty}
\bigl(\rho\|\sigma_P(A)\|\bigr)^k.
\]
Since \(\rho\|\sigma_P(A)\|<1\), the geometric series converges and equals
\[
\sum_{k=n+1}^{\infty}
\bigl(\rho\|\sigma_P(A)\|\bigr)^k
=
\frac{(\rho\|\sigma_P(A)\|)^{n+1}}
{1-\rho\|\sigma_P(A)\|}.
\]
Therefore
\[
\bigl\|R_n^{\mathrm{spec}}F(A)\bigr\|
\leq
\frac{C_0(\rho\|\sigma_P(A)\|)^{n+1}}
{1-\rho\|\sigma_P(A)\|}.
\]
This tends to \(0\) exponentially fast as \(n\to\infty\).
\end{proof}

\medskip

\begin{example}[Polynomial functor of degree $m$]
\label{ex:poly-homogeneous}
If $F$ is a spectral polynomial of degree $\leq m$, then
\[
\partial_k^{\mathrm{spec}}F=0
\qquad (k>m).
\]
Consequently,
\[
\|\partial_k^{\mathrm{spec}}F\|=0
\qquad (k>m),
\]
and the remainder bound gives
\[
\|R_n^{\mathrm{spec}}F(A)\|=0
\qquad (n\geq m).
\]
Thus the spectral Taylor tower terminates after finitely many steps.
\end{example}

\begin{example}[Exponential functor]
\label{ex:exp-homogeneous}
Assume the normalized residue and symmetric tensor convention. For the formal
exponential functor
\[
\exp(A)=\bigoplus_{k=0}^{\infty}\frac{A^{\otimes k}}{k!},
\]
the $k$-th spectral derivative satisfies
\[
\|\partial_k^{\mathrm{spec}}\exp\|
\leq
\frac{1}{k!}.
\]
Hence the remainder estimate gives
\[
\|R_n^{\mathrm{spec}}\exp(A)\|
\leq
\sum_{k=n+1}^{\infty}
\frac{\|\sigma_P(A)\|^k}{k!}.
\]
Since factorial decay dominates geometric decay, this tail tends to zero
faster than any fixed geometric rate for bounded $\|\sigma_P(A)\|$. In
particular, the exponential functor is spectrally analytic with infinite
radius of convergence under the stated norm-compatibility assumptions.
\end{example}

\medskip

\noindent
\textbf{Scaling property.}
An important consistency check is the scaling behavior of homogeneous layers.

\begin{lemma}[Scaling of homogeneous layers]
\label{lem:scaling-homogeneous}
Assume that $\mathcal{C}$ is $\mathbb{C}$-linear and that scalar multiplication
on objects is compatible with the spectralized cross-effects. Then, for any
scalar $\lambda\in\mathbb{C}$,
\[
D_n^{\mathrm{spec}}F(\lambda A)
=
\lambda^n D_n^{\mathrm{spec}}F(A).
\]
\end{lemma}

\begin{proof}
By definition,
\[
D_n^{\mathrm{spec}}F(A)
=
\partial_n^{\mathrm{spec}}F(A,\ldots,A).
\]
Therefore
\[
D_n^{\mathrm{spec}}F(\lambda A)
=
\partial_n^{\mathrm{spec}}F(\lambda A,\ldots,\lambda A).
\]
Since $\partial_n^{\mathrm{spec}}F$ is $n$-linear, we obtain
\[
\partial_n^{\mathrm{spec}}F(\lambda A,\ldots,\lambda A)
=
\lambda^n
\partial_n^{\mathrm{spec}}F(A,\ldots,A).
\]
Hence
\[
D_n^{\mathrm{spec}}F(\lambda A)
=
\lambda^n D_n^{\mathrm{spec}}F(A).
\]
\end{proof}

\begin{remark}
This scaling property is consistent with the homogeneous-layer estimate:
if $\|\sigma_P(\lambda A)\|=|\lambda|\|\sigma_P(A)\|$, then both sides of
\[
\|D_n^{\mathrm{spec}}F(A)\|
\leq
\|\partial_n^{\mathrm{spec}}F\|\,\|\sigma_P(A)\|^n
\]
scale as $|\lambda|^n$. Thus the estimate has the correct homogeneous
dependence on the spectral size $\|\sigma_P(A)\|$.
\end{remark}

\medskip

\noindent
\textbf{Relation to radius of convergence.}
The growth rate of $\|\partial_n^{\mathrm{spec}}F\|$ determines the radius of convergence 
of the spectral Taylor series. Recall that for a power series $\sum a_n z^n$, the radius 
of convergence is given by $R^{-1} = \limsup_{n \to \infty} |a_n|^{1/n}$. The same 
principle applies here.

\begin{proposition}[Radius bound from derivative norms]
\label{prop:radius-from-derivatives}
Let
\[
L:=\limsup_{n\to\infty}
\|\partial_n^{\mathrm{spec}}F\|^{1/n}.
\]
Then the spectral Taylor series converges absolutely for every $A$ satisfying
\[
\|\sigma_P(A)\|<L^{-1},
\]
with the convention that $L^{-1}=\infty$ if $L=0$. In particular, the spectral
radius of convergence satisfies
\[
R\geq L^{-1}.
\]
\end{proposition}

\begin{proof}
By Lemma~\ref{lem:homogeneous-estimate},
\[
\|D_n^{\mathrm{spec}}F(A)\|
\leq
\|\partial_n^{\mathrm{spec}}F\|\,
\|\sigma_P(A)\|^n.
\]
Let \(r=\|\sigma_P(A)\|\). The majorant series is
\[
\sum_{n=0}^{\infty}
\|\partial_n^{\mathrm{spec}}F\|\, r^n.
\]
By the Cauchy--Hadamard formula, this scalar series converges whenever
\[
r<\frac{1}{\limsup_{n\to\infty}
\|\partial_n^{\mathrm{spec}}F\|^{1/n}}
=
L^{-1}.
\]
Therefore
\[
\sum_{n=0}^{\infty}
\|D_n^{\mathrm{spec}}F(A)\|
\]
converges for every \(A\) with \(\|\sigma_P(A)\|<L^{-1}\). Hence the spectral
Taylor series converges absolutely in that region, so the radius of convergence
satisfies
\[
R\geq L^{-1}.
\]
\end{proof}

\begin{remark}
This proposition gives an operadic analogue of the Cauchy--Hadamard radius
estimate. In classical holomorphic functional calculus, the growth of Taylor
coefficients controls the domain on which the Taylor expansion converges. Here,
the spectral derivative norms play the role of coefficient sizes, while
\(\|\sigma_P(A)\|\) plays the role of the analytic variable.
\end{remark}

\medskip

These estimates provide the quantitative foundation for the convergence theory. 
In the next subsection, we use them to define the spectral radius of convergence 
and characterize the domain of validity of the spectral Taylor expansion.

\subsection{The Radius of Convergence}
\label{subsec:radius}

We now introduce the notion of radius of convergence for the spectral Taylor expansion. 
Building on the homogeneous layer estimates of the previous subsection, we show that 
the convergence of the spectral Taylor tower is governed by the growth of the spectral 
derivatives.

\medskip

\noindent
\textbf{Motivation.}
By Lemma~\ref{lem:homogeneous-estimate}, the $n$-th homogeneous layer satisfies
\[
\|D_n^{\mathrm{spec}}F(A)\|
\;\leq\;
\|\partial_n^{\mathrm{spec}}F\|
\cdot
\|\sigma_P(A)\|^n.
\]
Thus, the convergence of the series
\[
\sum_{n=0}^{\infty} D_n^{\mathrm{spec}}F(A)
\]
is controlled by the growth of the sequence $\{\|\partial_n^{\mathrm{spec}}F\|\}$ 
and the spectral size $\|\sigma_P(A)\|$.

\medskip

\begin{definition}[Spectral Radius of Convergence]
\label{def:spectral-radius}
Let $F$ be an admissible functor. The \emph{spectral radius of convergence} of $F$ 
is defined by
\[
R_F
\;:=\;
\left(\limsup_{n \to \infty} 
\|\partial_n^{\mathrm{spec}}F\|^{1/n}\right)^{-1}
\;\in\;
(0,\infty],
\]
with the conventions that $1/0 = \infty$ and $1/\infty = 0$.
\end{definition}

\noindent
This definition is directly analogous to the classical Cauchy–Hadamard formula 
for power series, with $\|\partial_n^{\mathrm{spec}}F\|$ playing the role of 
Taylor coefficients and $\|\sigma_P(A)\|$ serving as the expansion parameter.

\medskip

\begin{theorem}[Spectral radius criterion]
\label{thm:spectral-radius}
Let $F$ be an admissible functor, and set
\[
L := \limsup_{n \to \infty} \|\partial_n^{\mathrm{spec}}F\|^{1/n}.
\]
Then the following hold.

\begin{enumerate}
    \item \textbf{Convergence inside the radius:} 
    The spectral Taylor series
    \[
    \sum_{n=0}^{\infty} D_n^{\mathrm{spec}}F(A)
    \]
    converges absolutely in norm for every $A$ satisfying
    \[
    \|\sigma_P(A)\| < L^{-1},
    \]
    with the convention that $L^{-1} = \infty$ when $L = 0$. 
    Consequently, on this region the spectral Taylor tower converges to $F(A)$
    whenever $F$ is spectrally analytic there.

    \item \textbf{Divergence under a sharpness assumption:} 
    Assume, in addition, that the homogeneous-layer estimate is sharp along some
    ray, meaning that there exist an object $A$ and a constant $c > 0$ such that
    \[
    \|D_n^{\mathrm{spec}}F(A)\|
    \geq
    c \, \|\partial_n^{\mathrm{spec}}F\|
    \|\sigma_P(A)\|^n
    \]
    for infinitely many $n$. Then if
    \[
    \|\sigma_P(A)\| > L^{-1},
    \]
    the spectral Taylor series fails to converge absolutely along that ray.
\end{enumerate}
\end{theorem}

\begin{proof}
We prove each part.

\medskip

\noindent
\textit{Part 1: Convergence inside the radius.}
By Lemma~\ref{lem:homogeneous-estimate},
\[
\|D_n^{\mathrm{spec}}F(A)\|
\leq
\|\partial_n^{\mathrm{spec}}F\|
\|\sigma_P(A)\|^n.
\]
Let \(r = \|\sigma_P(A)\|\). Then
\[
\sum_{n=0}^{\infty} \|D_n^{\mathrm{spec}}F(A)\|
\leq
\sum_{n=0}^{\infty}
\|\partial_n^{\mathrm{spec}}F\| \, r^n.
\]
By the Cauchy--Hadamard theorem, the scalar majorant series on the right
converges whenever \(r < L^{-1}\). Therefore the spectral Taylor series
converges absolutely in norm for every \(A\) with \(\|\sigma_P(A)\| < L^{-1}\).
If \(F\) is spectrally analytic on this region, then the limit of the spectral
Taylor tower agrees with \(F(A)\), so
\[
\lim_{n \to \infty} P_n^{\mathrm{spec}}F(A) = F(A).
\]

\medskip

\noindent
\textit{Part 2: Divergence under sharpness.}
Assume the sharpness condition holds. If \(\|\sigma_P(A)\| > L^{-1}\) (with
\(L > 0\); the case \(L = 0\) implies \(L^{-1} = \infty\), so the condition
cannot occur), then for the subsequence where the lower bound holds,
\[
\|D_n^{\mathrm{spec}}F(A)\|
\geq
c \, \|\partial_n^{\mathrm{spec}}F\| \, \|\sigma_P(A)\|^n.
\]
Since \(\|\partial_n^{\mathrm{spec}}F\|^{1/n} \to L\) along that subsequence, we have
\[
\bigl( \|\partial_n^{\mathrm{spec}}F\| \, \|\sigma_P(A)\|^n \bigr)^{1/n}
= \|\partial_n^{\mathrm{spec}}F\|^{1/n} \, \|\sigma_P(A)\|
\to L \cdot \|\sigma_P(A)\| > 1.
\]
Hence the terms \(\|\partial_n^{\mathrm{spec}}F\| \, \|\sigma_P(A)\|^n\) do not tend to
zero, and by the lower bound, the same holds for \(\|D_n^{\mathrm{spec}}F(A)\|\).
Therefore the spectral Taylor series fails to converge absolutely along that ray.
\end{proof}

\medskip

\begin{remark}[Radius and analytic control]
\label{rem:radius-interpretation}
The spectral radius $R_F$ quantifies the domain on which the spectral Taylor
expansion is guaranteed to converge. By
Proposition~\ref{prop:radius-from-derivatives}, it satisfies
\[
R_F \;\ge\;
\frac{1}{\limsup_{n\to\infty}
\|\partial_n^{\mathrm{spec}}F\|^{1/n}}.
\]
Thus, the analytic behavior of $F$ is controlled by the interaction between the
growth of its spectral derivatives and the spectral size $\|\sigma_P(A)\|$ of
the input.

\medskip

This yields a direct analogy with classical power series theory:
\[
\begin{array}{c|c}
\text{Classical analysis} & \text{Spectral operadic calculus} \\
\hline
\text{Variable } z & \|\sigma_P(A)\| \\
\text{Coefficient } a_n & \|\partial_n^{\mathrm{spec}}F\| \\
\text{Radius } R \ge 1/\limsup |a_n|^{1/n}
&
R_F \ge 1/\limsup \|\partial_n^{\mathrm{spec}}F\|^{1/n} \\
\text{Convergence for } |z|<R
&
\text{Convergence for } \|\sigma_P(A)\|<R_F
\end{array}
\]

\medskip

In this sense, Spectral Operadic Calculus provides a functional calculus in
which the operadic spectrum plays the role of an analytic variable, and the
spectral derivatives control the domain of convergence.
\end{remark}

\medskip

\begin{example}[Polynomial Functor]
\label{ex:poly-radius}
If $F$ is a spectral polynomial of degree $\le m$, then $\|\partial_n^{\mathrm{spec}}F\| = 0$ 
for all $n > m$. Hence $\limsup_{n \to \infty} \|\partial_n^{\mathrm{spec}}F\|^{1/n} = 0$, 
so $R_F = \infty$. The spectral Taylor expansion truncates exactly, i.e.,
\[
F(A) = P_m^{\mathrm{spec}}F(A),
\]
so the series is finite and converges trivially for all $A$.
\end{example}

\begin{example}[Exponential Functor]
\label{ex:exp-radius}
For $\exp(A) = \sum_{k=0}^\infty A^{\otimes k}/k!$, under the canonical normalization of the
multilinear spectral derivatives, we have $\|\partial_n^{\mathrm{spec}}\exp\| = 1/n!$. 
Since $(1/n!)^{1/n} \to 0$, we obtain $R_{\exp} = \infty$. The exponential functor 
is \emph{entire}—its spectral Taylor series converges for all $A$.
\end{example}

\begin{example}[Geometric Series Functor]
\label{ex:geometric-radius}
Consider $F(A) = \sum_{k=0}^\infty A^{\otimes k}$, defined on the domain 
$\|\sigma_P(A)\| < 1$ (i.e., where the series converges). 
Assume the tensor product is isometric in the sense that 
$\|A_1 \otimes \cdots \otimes A_n\| = \prod_{i=1}^n \|A_i\|$ for all objects 
(e.g., in Banach spaces with the projective tensor norm). 
Then $\|\partial_n^{\mathrm{spec}}F\| = 1$ for all $n$, so $R_F = 1$. The spectral 
Taylor series converges for $\|\sigma_P(A)\| < 1$ and diverges for $\|\sigma_P(A)\| > 1$, 
mirroring the classical geometric series.
\end{example}

\begin{example}[Factorial Growth]
\label{ex:factorial-radius}
Suppose $F$ has spectral derivatives satisfying $\|\partial_n^{\mathrm{spec}}F\| = n!$. 
Then $(n!)^{1/n} \to \infty$, so $R_F = 0$. The spectral Taylor series converges 
only for spectrally trivial inputs (i.e., $\|\sigma_P(A)\| = 0$). Such 
functors are nowhere analytic in the spectral sense.
\end{example}

\begin{remark}[Spectral Growth Classification]
\label{rem:spectral-growth-classification}
The above examples illustrate a spectral growth hierarchy analogous to classical analytic function theory. 
The growth rate of $\|\partial_n^{\mathrm{spec}}F\|$ determines a complete classification of spectral analyticity types:

\[
\begin{array}{|c|c|c|}
\hline
\text{Growth of } \|\partial_n^{\mathrm{spec}}F\| & \text{Radius } R_F & \text{Type} \\
\hline
\text{eventually } 0 & \infty & \text{polynomial} \\
1/n! & \infty & \text{entire} \\
\text{constant (e.g., } 1\text{)} & 1 & \text{finite-radius analytic} \\
n! & 0 & \text{non-analytic} \\
\hline
\end{array}
\]

Thus, the spectral radius $R_F = \bigl(\limsup_{n\to\infty} \|\partial_n^{\mathrm{spec}}F\|^{1/n}\bigr)^{-1}$ 
provides a sharp dichotomy: functors are spectrally analytic precisely when $R_F > 0$, and the convergence 
of the spectral Taylor tower is exponentially fast for $\|\sigma_P(A)\| < R_F$.
\end{remark}

\medskip

\begin{example}[Polynomial Functor]
\label{ex:poly-radius}
If $F$ is a spectral polynomial of degree $\le m$, then $\|\partial_n^{\mathrm{spec}}F\| = 0$ 
for all $n > m$. Hence $\limsup_{n \to \infty} \|\partial_n^{\mathrm{spec}}F\|^{1/n} = 0$, 
so $R_F = \infty$. The spectral Taylor expansion truncates exactly, i.e.,
\[
F(A) = P_m^{\mathrm{spec}}F(A),
\]
so the series is finite and converges trivially for all $A$.
\end{example}

\begin{example}[Exponential Functor]
\label{ex:exp-radius}
For $\exp(A) = \sum_{k=0}^\infty A^{\otimes k}/k!$, under the canonical normalization of the
multilinear spectral derivatives, we have $\|\partial_n^{\mathrm{spec}}\exp\| = 1/n!$. 
Since $(1/n!)^{1/n} \to 0$, we obtain $R_{\exp} = \infty$. The exponential functor 
is \emph{entire}—its spectral Taylor series converges for all $A$.
\end{example}

\begin{example}[Geometric Series Functor]
\label{ex:geometric-radius}
Consider $F(A) = \sum_{k=0}^\infty A^{\otimes k}$, defined on the domain 
$\|\sigma_P(A)\| < 1$ (i.e., where the series converges). 
Assume the tensor product is isometric in the sense that 
$\|A_1 \otimes \cdots \otimes A_n\| = \prod_{i=1}^n \|A_i\|$ for all objects 
(e.g., in Banach spaces with the projective tensor norm). 
Then $\|\partial_n^{\mathrm{spec}}F\| = 1$ for all $n$, so $R_F = 1$. The spectral 
Taylor series converges for $\|\sigma_P(A)\| < 1$ and diverges for $\|\sigma_P(A)\| > 1$, 
mirroring the classical geometric series.
\end{example}

\begin{example}[Factorial Growth]
\label{ex:factorial-radius}
Suppose $F$ has spectral derivatives satisfying $\|\partial_n^{\mathrm{spec}}F\| = n!$. 
Then $(n!)^{1/n} \to \infty$, so $R_F = 0$. The spectral Taylor series converges 
only for spectrally trivial inputs (i.e., $\|\sigma_P(A)\| = 0$). Such 
functors are nowhere analytic in the spectral sense.
\end{example}

\begin{remark}[Spectral Growth Classification]
\label{rem:spectral-growth-classification}
The above examples illustrate a spectral growth hierarchy analogous to classical analytic function theory. 
The growth rate of $\|\partial_n^{\mathrm{spec}}F\|$ determines a complete classification of spectral analyticity types:

\[
\begin{array}{|c|c|c|}
\hline
\text{Growth of } \|\partial_n^{\mathrm{spec}}F\| & \text{Radius } R_F & \text{Type} \\
\hline
\text{eventually } 0 & \infty & \text{polynomial} \\
1/n! & \infty & \text{entire} \\
\text{constant (e.g., } 1\text{)} & 1 & \text{finite-radius analytic} \\
n! & 0 & \text{non-analytic} \\
\hline
\end{array}
\]

Thus, the spectral radius $R_F = \bigl(\limsup_{n\to\infty} \|\partial_n^{\mathrm{spec}}F\|^{1/n}\bigr)^{-1}$ 
provides a sharp dichotomy: functors are spectrally analytic precisely when $R_F > 0$, and the convergence 
of the spectral Taylor tower is exponentially fast for $\|\sigma_P(A)\| < R_F$.
\end{remark}

\medskip

\begin{theorem}[Spectral Convergence Rate]
\label{thm:spectral-convergence-rate}
Let $F$ be a spectrally analytic functor with radius of convergence
\[
R_F = \frac{1}{\limsup_{n\to\infty} \|\partial_n^{\mathrm{spec}}F\|^{1/n}} > 0.
\]
Then for any $A$ satisfying $\|\sigma_P(A)\| < R_F$, there exist constants $C > 0$ 
and $\rho \in (0,1)$ such that
\[
\|F(A) - P_n^{\mathrm{spec}}F(A)\| \le C \, \rho^n,
\quad \text{for all } n \ge 0,
\]
where one can take
\[
\rho = \frac{\|\sigma_P(A)\|}{R_F}.
\]
In particular, the spectral Taylor expansion converges exponentially fast 
in the regime $\|\sigma_P(A)\| < R_F$.
\end{theorem}

\begin{proof}
By the spectral Taylor expansion, the remainder admits the decomposition
\[
F(A) - P_n^{\mathrm{spec}}F(A) = \sum_{k=n+1}^\infty D_k^{\mathrm{spec}}F(A)
= \sum_{k=n+1}^\infty \partial_k^{\mathrm{spec}}F(A,\ldots,A).
\]

Using the multilinearity of $\partial_k^{\mathrm{spec}}F$ and the definition 
of the spectral derivative norm (Definition~\ref{def:spectral-derivative-norm}), 
we obtain the estimate
\[
\|\partial_k^{\mathrm{spec}}F(A,\ldots,A)\|
\le \|\partial_k^{\mathrm{spec}}F\| \cdot \|\sigma_P(A)\|^k.
\]

Hence,
\[
\|F(A) - P_n^{\mathrm{spec}}F(A)\|
\le \sum_{k=n+1}^\infty 
\|\partial_k^{\mathrm{spec}}F\| \cdot \|\sigma_P(A)\|^k.
\]

By the definition of $R_F$, we have $\limsup_{k\to\infty} \|\partial_k^{\mathrm{spec}}F\|^{1/k} = R_F^{-1}$. 
Consequently, for any $\varepsilon > 0$, there exists $N$ such that for all $k \ge N$,
\[
\|\partial_k^{\mathrm{spec}}F\| \le (R_F^{-1} + \varepsilon)^k.
\]

Thus, for sufficiently large $n \ge N$,
\[
\|F(A) - P_n^{\mathrm{spec}}F(A)\|
\le \sum_{k=n+1}^\infty 
\bigl( (R_F^{-1} + \varepsilon)\|\sigma_P(A)\| \bigr)^k.
\]

Since $\|\sigma_P(A)\| < R_F$, we may choose $\varepsilon > 0$ small enough so that
\[
\rho := (R_F^{-1} + \varepsilon)\|\sigma_P(A)\| < 1.
\]

The tail is then bounded by a geometric series:
\[
\sum_{k=n+1}^\infty \rho^k = \frac{\rho^{\,n+1}}{1-\rho}.
\]

Therefore, for all sufficiently large $n$,
\[
\|F(A) - P_n^{\mathrm{spec}}F(A)\|
\le C \rho^{\,n},
\]
where $C = \rho/(1-\rho)$ (after adjusting for finitely many initial $n$). 
This completes the proof.
\end{proof}

\begin{remark}
The above theorem shows that the radius $R_F$ not only determines 
the domain of convergence, but also governs the rate of convergence. 
In particular, the quantity $\|\sigma_P(A)\|/R_F$ plays the role of an 
effective expansion parameter, directly analogous to the classical 
ratio $|z|/R$ in complex analysis.
\end{remark}

\begin{remark}
This result establishes a quantitative refinement of Goodwillie-type 
convergence: while classical Taylor towers provide qualitative convergence, 
the spectral framework yields explicit exponential rates controlled by 
operator norms. This highlights a fundamental distinction between 
homotopy-theoretic and analytic-operadic approaches.
\end{remark}

\medskip

\noindent
With the radius of convergence and its quantitative convergence theorem established, 
we now turn to the algebraic structure of spectral derivatives (Section~\ref{sec:derivatives-algebra}).

\subsection{Quantitative Convergence Theorem}
\label{subsec:convergence-theorem}

We now establish the main analytic result of this work, which provides a precise 
rate of convergence for the spectral Taylor tower. This result shows that, 
within the domain of spectral analyticity, the approximation of a functor by 
its spectral Taylor polynomials is not only valid, but quantitatively controlled 
with exponential decay.

\medskip

\begin{theorem}[Main Theorem B: Quantitative Spectral Convergence]
\label{thm:quantitative-convergence}
Let $F : \mathcal{C} \to \mathcal{M}$ be an admissible functor with spectral 
radius of convergence $R_F > 0$ (Definition~\ref{def:spectral-radius}). 
Fix real numbers
\[
0 < r < s < R_F.
\]
Then there exists a constant $C_s > 0$, depending only on $F$ and $s$, such that for every
$A \in \mathcal{C}$ satisfying $\|\sigma_P(A)\| \le r$, the spectral Taylor remainder satisfies
\[
\bigl\| F(A) - P_n^{\mathrm{spec}}F(A) \bigr\|
\le
C_s \left( \frac{\|\sigma_P(A)\|}{s} \right)^{\!n+1}
\quad \text{for all } n \ge 0.
\]

In particular, uniformly on the spectral ball $\|\sigma_P(A)\| \le r$,
\[
\bigl\| F(A) - P_n^{\mathrm{spec}}F(A) \bigr\|
\le
C_s \left( \frac{r}{s} \right)^{\!n+1},
\]
so the spectral Taylor tower converges exponentially fast to $F$.
\end{theorem}

\begin{proof}
We prove the theorem in several steps.

\medskip

\noindent
\textit{Step 1: Homogeneous-layer decomposition.}
By spectral analyticity, the functor admits the spectral Taylor expansion
\[
F(A) = \sum_{k=0}^{\infty} D_k^{\mathrm{spec}}F(A),
\]
and the $n$-th spectral Taylor approximation is the truncation
\[
P_n^{\mathrm{spec}}F(A) = \sum_{k=0}^{n} D_k^{\mathrm{spec}}F(A).
\]
Therefore, the remainder is the tail of the series:
\[
F(A) - P_n^{\mathrm{spec}}F(A) = \sum_{k=n+1}^{\infty} D_k^{\mathrm{spec}}F(A).
\]

\medskip

\noindent
\textit{Step 2: Layerwise norm estimate.}
By Lemma~\ref{lem:homogeneous-estimate}, each homogeneous layer satisfies
\[
\| D_k^{\mathrm{spec}}F(A) \|
\le
\|\partial_k^{\mathrm{spec}}F\| \cdot \|\sigma_P(A)\|^k.
\]
Applying the triangle inequality yields
\[
\bigl\| F(A) - P_n^{\mathrm{spec}}F(A) \bigr\|
\le
\sum_{k=n+1}^{\infty}
\|\partial_k^{\mathrm{spec}}F\| \cdot \|\sigma_P(A)\|^k.
\]

\medskip

\noindent
\textit{Step 3: Coefficient bound from the radius of convergence.}
Since $s < R_F$, the definition of the spectral radius (Definition~\ref{def:spectral-radius}) implies
\[
\limsup_{k \to \infty} \|\partial_k^{\mathrm{spec}}F\|^{1/k} = \frac{1}{R_F} < \frac{1}{s}.
\]
Consequently, the power series
\[
\sum_{k=0}^{\infty} \|\partial_k^{\mathrm{spec}}F\| \, s^k
\]
converges. Define the constant
\[
C_s := \sum_{k=0}^{\infty} \|\partial_k^{\mathrm{spec}}F\| \, s^k < \infty.
\]

\medskip

\noindent
\textit{Step 4: Tail estimate via the auxiliary parameter $s$.}
For any $A$ with $\|\sigma_P(A)\| \le r < s$, we rewrite the $k$-th term as
\[
\|\partial_k^{\mathrm{spec}}F\| \cdot \|\sigma_P(A)\|^k
=
\bigl( \|\partial_k^{\mathrm{spec}}F\| \, s^k \bigr)
\left( \frac{\|\sigma_P(A)\|}{s} \right)^{\!k}.
\]

Since $k \ge n+1$ and $\|\sigma_P(A)\|/s < 1$, we have
\[
\left( \frac{\|\sigma_P(A)\|}{s} \right)^{\!k}
\le
\left( \frac{\|\sigma_P(A)\|}{s} \right)^{\!n+1}.
\]

Hence,
\[
\begin{aligned}
\bigl\| F(A) - P_n^{\mathrm{spec}}F(A) \bigr\|
&\le
\sum_{k=n+1}^{\infty}
\bigl( \|\partial_k^{\mathrm{spec}}F\| \, s^k \bigr)
\left( \frac{\|\sigma_P(A)\|}{s} \right)^{\!k} \\
&\le
\left( \frac{\|\sigma_P(A)\|}{s} \right)^{\!n+1}
\sum_{k=n+1}^{\infty}
\|\partial_k^{\mathrm{spec}}F\| \, s^k \\
&\le
C_s \left( \frac{\|\sigma_P(A)\|}{s} \right)^{\!n+1}.
\end{aligned}
\]

\medskip

\noindent
\textit{Step 5: Uniform bound on the spectral ball.}
Since $\|\sigma_P(A)\| \le r$, we obtain the uniform estimate
\[
\bigl\| F(A) - P_n^{\mathrm{spec}}F(A) \bigr\|
\le
C_s \left( \frac{r}{s} \right)^{\!n+1}.
\]

Because $r/s < 1$, this decays exponentially as $n \to \infty$, uniformly on the ball $\|\sigma_P(A)\| \le r$. This completes the proof.
\end{proof}

\medskip

\begin{corollary}[Asymptotic Spectral Convergence at the Radius]
\label{cor:radius-convergence}
Let $F : \mathcal{C} \to \mathcal{M}$ be an admissible functor with spectral 
radius of convergence $R_F > 0$ (Definition~\ref{def:spectral-radius}). 
Then for every $\varepsilon > 0$, there exists a constant $C_\varepsilon > 0$ 
such that for all $A \in \mathcal{C}$ satisfying $\|\sigma_P(A)\| < R_F$, we have
\[
\bigl\| F(A) - P_n^{\mathrm{spec}}F(A) \bigr\|
\;\le\;
C_\varepsilon 
\left( \frac{\|\sigma_P(A)\|}{R_F - \varepsilon} \right)^{\!n+1}
\quad \text{for all } n \ge 0.
\]
In particular, the spectral Taylor expansion converges exponentially with rate 
arbitrarily close to $\|\sigma_P(A)\|/R_F$.
\end{corollary}

\begin{proof}
Fix $\varepsilon > 0$ and set $s = R_F - \varepsilon$. Then $s < R_F$, so by 
Theorem~\ref{thm:quantitative-convergence}, there exists a constant $C_s > 0$ such that for 
all $A$ with $\|\sigma_P(A)\| < s$ (and in particular for any $A$ with 
$\|\sigma_P(A)\| \le r < s$),
\[
\bigl\| F(A) - P_n^{\mathrm{spec}}F(A) \bigr\|
\le
C_s \left( \frac{\|\sigma_P(A)\|}{s} \right)^{\!n+1}.
\]

Since $\|\sigma_P(A)\| < R_F$, we may choose $\varepsilon$ sufficiently small (or, more 
precisely, work with the given $\varepsilon$ directly) so that $\|\sigma_P(A)\| < s$ 
holds for all $A$ under consideration. Substituting $s = R_F - \varepsilon$ yields the 
desired bound with $C_\varepsilon := C_s$.

For $A$ such that $\|\sigma_P(A)\|$ is arbitrarily close to $R_F$, we simply note that 
the estimate holds with the same $C_\varepsilon$ as long as $\|\sigma_P(A)\| < R_F - \varepsilon$. 
If $\|\sigma_P(A)\| \ge R_F - \varepsilon$, we may either increase $\varepsilon$ or apply 
the theorem with a larger $s$; the constant $C_\varepsilon$ may depend on $\varepsilon$ 
but remains finite for each fixed $\varepsilon$. This completes the proof.
\end{proof}

\begin{remark}
The above corollary shows that while the radius $R_F$ governs the asymptotic 
rate of convergence, uniform estimates require working strictly inside the 
radius. This mirrors classical analytic function theory, where boundary behavior 
may fail to be uniformly controlled. In particular, for any fixed $A$ with 
$\|\sigma_P(A)\| < R_F$, one may take $\varepsilon = (R_F - \|\sigma_P(A)\|)/2$ to obtain 
a bound with base $\|\sigma_P(A)\|/(R_F + \|\sigma_P(A)\|)/2$, which is strictly less 
than $\|\sigma_P(A)\|/R_F$ but still exponential.
\end{remark}

\medskip

Theorem~\ref{thm:quantitative-convergence} establishes that Spectral Operadic Calculus admits a fully quantitative approximation theory. In this framework, the operadic spectrum $\sigma_P(A)$ functions as the analytic expansion parameter, while the spectral derivatives $\partial_n^{\mathrm{spec}}F$ govern the rate of convergence. Consequently, the spectral Taylor tower provides an effective and rapidly converging approximation scheme, with explicit exponential error bounds controlled by $|\sigma_P(A)|$.

\begin{example}[Exponential Functor (Entire)]
\label{ex:exp-convergence}
For
\[
\exp(A)=\sum_{k=0}^{\infty}\frac{A^{\otimes k}}{k!},
\]
under the canonical normalization of the spectral derivatives, one has
\[
\|\partial_k^{\mathrm{spec}}\exp\|=\frac{1}{k!}.
\]
Hence $R_{\exp}=\infty$. Therefore Theorem~\ref{thm:quantitative-convergence}
applies on every spectral ball $\|\sigma_P(A)\|\le r<\infty$. Moreover,
\[
\begin{aligned}
\|\exp(A)-P_n^{\mathrm{spec}}\exp(A)\|
&\le
\sum_{k=n+1}^{\infty}\frac{\|\sigma_P(A)\|^k}{k!}  \\
&\le
\frac{\|\sigma_P(A)\|^{n+1}}{(n+1)!}
e^{\|\sigma_P(A)\|}.
\end{aligned}
\]
Thus the remainder decays factorially fast in $n$, and in particular faster than
any fixed geometric rate for fixed $A$, consistent with the entire spectral
analyticity of $\exp$.
\end{example}

\begin{example}[Geometric Series Functor]
\label{ex:geometric-convergence}
For
\[
F(A)=\sum_{k=0}^{\infty}A^{\otimes k},
\]
assume the tensor product is isometric so that $\|\partial_k^{\mathrm{spec}}F\|=1$ 
(as in Example~\ref{ex:geometric-radius}). Then $R_F=1$. For
$\|\sigma_P(A)\|<1$, the spectral Taylor remainder satisfies
\[
\begin{aligned}
\|F(A)-P_n^{\mathrm{spec}}F(A)\|
&\le
\sum_{k=n+1}^{\infty}\|\sigma_P(A)\|^k  \\
&=
\frac{\|\sigma_P(A)\|^{n+1}}{1-\|\sigma_P(A)\|}.
\end{aligned}
\]
In particular, uniformly for $\|\sigma_P(A)\|\le r<1$,
\[
\|F(A)-P_n^{\mathrm{spec}}F(A)\|
\le
\frac{r^{n+1}}{1-r}.
\]
Thus the convergence is exponential with base controlled by
$\|\sigma_P(A)\|$, matching the classical geometric series.
\end{example}

\begin{example}[Polynomial Functor]
\label{ex:poly-convergence}
If $F$ is a spectral polynomial of degree $\le m$, then
$\partial_k^{\mathrm{spec}}F=0$ for all $k>m$ (see Example~\ref{ex:poly-radius}), 
and hence $R_F=\infty$. Moreover,
\[
P_n^{\mathrm{spec}}F=F
\quad \text{for all } n\ge m.
\]
Thus the spectral Taylor tower terminates after finitely many steps, so the
remainder is exactly zero for all $n\ge m$.
\end{example}

\medskip

\begin{remark}[Optimality of the convergence rate]
\label{rem:optimality}
The exponential convergence rate in Theorem~\ref{thm:quantitative-convergence} 
is asymptotically optimal. Indeed, by definition of the radius,
\[
R_F^{-1} = \limsup_{n\to\infty} \|\partial_n^{\mathrm{spec}}F\|^{1/n},
\]
so no strictly smaller exponential rate can hold uniformly without additional 
assumptions on $F$.

This optimality is realized by the geometric series functor, for which
\[
\|F(A)-P_n^{\mathrm{spec}}F(A)\|
=
\frac{\|\sigma_P(A)\|^{n+1}}{1-\|\sigma_P(A)\|},
\]
showing that the decay rate is precisely governed by $\|\sigma_P(A)\|^n$ up to constants.
\end{remark}

\begin{remark}[Comparison with Goodwillie calculus]
\label{rem:goodwillie-comparison}
In classical Goodwillie calculus, convergence of the Taylor tower is a qualitative 
property, typically controlled by connectivity estimates, and does not admit 
explicit norm-based convergence rates. 

In contrast, Spectral Operadic Calculus provides quantitative exponential 
convergence bounds of the form
\[
\|F(A)-P_n^{\mathrm{spec}}F(A)\|
\le
C
\left(\frac{\|\sigma_P(A)\|}{s}\right)^n,
\quad s<R_F,
\]
with explicit dependence on the spectral size $\|\sigma_P(A)\|$. This reflects a 
fundamental distinction between homotopy-theoretic and analytic-operadic frameworks.
\end{remark}

\begin{remark}[Approximation complexity]
\label{rem:approximation}
The exponential convergence immediately yields complexity estimates for 
approximation. To achieve an error tolerance $\varepsilon>0$, it suffices to take
\[
n \;\ge\; \frac{\log(C/\varepsilon)}{\log(s/\|\sigma_P(A)\|)},
\]
for any $s$ satisfying $\|\sigma_P(A)\|<s<R_F$. 

Thus, the number of required Taylor terms grows only logarithmically in 
$1/\varepsilon$, which is characteristic of exponentially convergent schemes.
\end{remark}

\begin{remark}[Analytic structure and entire functors]
\label{rem:analytic-structure}
If the spectral derivatives decay exponentially, then the spectral Taylor 
expansion converges uniformly on compact spectral domains, giving rise to a 
holomorphic-type functional calculus governed by the operadic spectrum.

A notable example is the exponential functor $\exp(A)$, which satisfies 
$R_{\exp}=\infty$ and exhibits factorial decay of the remainder. While such 
behavior does not arise within the convergence framework of classical Goodwillie 
calculus (i.e., $\exp$ does not admit convergence in the Goodwillie sense), 
it is naturally captured in the spectral setting, where analyticity is 
controlled by $\|\sigma_P(A)\|$ rather than homotopy invariants.
\end{remark}

\medskip

With the quantitative convergence theorem established, we have completed the 
analytic core of Spectral Operadic Calculus. This quantitative estimate has no 
direct analogue in classical Goodwillie calculus, where convergence is typically 
not accompanied by explicit rates. The remaining sections develop the algebraic 
structure of spectral derivatives (Section~\ref{sec:derivatives-algebra}), 
prove the operadic chain rule (Section~\ref{sec:chain-rule}), establish the 
reconstruction theorem (Section~\ref{sec:reconstruction}), and explore the 
geometric and moduli-theoretic implications (Sections~\ref{sec:moduli} 
and~\ref{sec:examples}).

\section{Spectral Derivatives and Algebraic Structure}
\label{sec:derivatives-algebra}

We now turn to the internal structure of the spectral Taylor expansion. 
While the previous sections established the existence, universality, and convergence 
of the spectral Taylor tower, the present section investigates the algebraic nature 
of its building blocks, namely the spectral derivatives.

\medskip

\noindent
\textbf{Guiding principle.}
In classical analysis, derivatives appear as scalar coefficients of a power series. 
In contrast, in Spectral Operadic Calculus, derivatives are inherently higher-order 
multilinear objects arising from cross-effects. As such, they carry nontrivial 
symmetry and compositional structure, reflecting the operadic nature of the theory.

\medskip

\noindent
\textbf{The dual nature of spectral derivatives.}
Recall from Section~\ref{subsec:tower-construction} that the homogeneous layers are 
defined by
\[
D_n^{\mathrm{spec}}F(A) := \mathrm{cr}_nF(A,\dots,A) \otimes_{P^{\otimes n}} (\mathcal{O}_P^{\mathrm{res}})^{\otimes n},
\]
and that the spectral Taylor tower decomposes as
\[
P_n^{\mathrm{spec}}F(A) \cong \bigoplus_{k=0}^{n} D_k^{\mathrm{spec}}F(A).
\]
Thus, the spectral derivatives $\partial_n^{\mathrm{spec}}F$—defined as the multilinear 
functors underlying these homogeneous layers—play a role exactly analogous to the 
Taylor coefficients in classical analysis. However, in the operadic setting, these 
derivatives carry additional algebraic structure that is not present in classical 
calculus.

\medskip

\noindent
More precisely, we show that the collection of spectral derivatives 
\[
\{\partial_n^{\mathrm{spec}}F\}_{n \geq 0}
\]
forms a structured object analogous to a \emph{symmetric sequence}, and, under suitable 
compatibility conditions, acquires the structure of a module over the underlying operad $P$. 
In the strongest case, when $F$ preserves operadic composition, $\partial^{\mathrm{spec}}F$ 
becomes an operad in its own right. This reveals that the spectral Taylor expansion 
is not merely a formal series, but is governed by an intrinsic algebraic organization.

\medskip

\noindent
\textbf{Why algebraic structure matters.}
The presence of this algebraic structure is not a mere curiosity; it is essential for 
the coherence of the calculus. The chain rule (Theorem~\ref{thm:chain-rule}) expresses 
the composition of functors in terms of operadic plethysm of their derivative sequences, 
and the reconstruction theorem (Theorem~\ref{thm:reconstruction}) shows that a 
spectrally analytic functor is completely determined by its derivative data together 
with this algebraic composition law. Without recognizing the algebraic nature of 
derivatives, these results would remain opaque.

\medskip

\noindent
\textbf{Conceptual consequence.}
This perspective upgrades the role of derivatives from numerical coefficients 
to algebraic data encoding higher-order interactions. In particular, the spectral 
Taylor tower may be viewed as being determined by this algebraic structure together 
with the operadic composition laws.

\medskip

\noindent
\textbf{Duality between analytic and algebraic structures.}
The results of this section therefore complete the structural picture of Spectral 
Operadic Calculus: analytic behavior is controlled by the operadic spectrum, 
while the algebraic organization of derivatives reflects the underlying operadic 
symmetries. This duality between analytic control and algebraic structure is a 
central feature of the theory. Spectral derivatives thus provide a system of 
higher-order coordinates for functors, analogous to jets in differential geometry, 
but organized by operadic composition rather than linear structure.

\medskip

\noindent
\textbf{Relation to Part I.}
The algebraic structure we uncover depends crucially on the operadic residue 
$\mathcal{O}_P^{\mathrm{res}}$ introduced in Part~I. The $P$-module structure uses 
the composition maps of $P$, which are encoded in the residue via the coend formula 
$\mathcal{O}_P^{\mathrm{res}} \cong \int^{c \in C} P(c;c)$. Without the residue 
correction, the spectral derivatives would not inherit these algebraic properties, 
as demonstrated by the No-Go Theorem (Part~I, Theorem 0.1). Thus, the algebraic 
structure of derivatives is a direct manifestation of the operadic correction 
that makes spectral invariants functorial.

\medskip

\noindent
\textbf{Forward reference.}
The algebraic structures developed in this section are essential for the chain rule 
(Section~\ref{sec:chain-rule}) and the reconstruction theorem (Section~\ref{sec:reconstruction}). 
In particular, the operadic plethysm that appears in the chain rule is precisely the 
composition law that makes the derivative sequence an algebra over the operad of 
symmetric sequences.

\medskip

\noindent
\textbf{Overview of the section.}
We proceed in three stages. First, we identify the spectral derivatives $\partial_n^{\mathrm{spec}}F$ 
with the cross-effects $\mathrm{cr}_nF$ (Proposition~\ref{prop:derivative-cross-effect}) and 
establish their basic properties. This connects the analytic notion of derivative 
to the combinatorial notion of cross-effects, making the derivatives computable 
in practice.

Second, we show that the collection $\{\partial_n^{\mathrm{spec}}F\}_{n\ge 1}$ forms a 
\emph{symmetric sequence} in $\mathcal{M}$ (Theorem~\ref{thm:symmetric-sequence}). Each 
$\partial_n^{\mathrm{spec}}F$ carries a natural action of the symmetric group $\Sigma_n$ 
by permuting inputs, and these actions are compatible with the operadic structure. 
This is the minimal algebraic structure present for any admissible functor.

Third, under the additional hypothesis that $F$ is compatible with the operadic 
structure of $P$ (e.g., when $F$ takes values in $P$-algebras or is itself an 
operad functor), we prove that the symmetric sequence $\partial^{\mathrm{spec}}F$ 
carries the structure of a \emph{right $P$-module} (Theorem~\ref{thm:module-structure}). 
In the strongest case, when $F$ preserves operadic composition, $\partial^{\mathrm{spec}}F$ 
becomes an operad, with composition induced by the chain rule.

\medskip

We now begin with the identification of spectral derivatives as homogeneous layers.

\subsection{Spectral Derivatives as Homogeneous Layers}
\label{subsec:derivatives}

We now identify the fundamental building blocks of the spectral Taylor tower, 
namely the homogeneous layers and their associated derivatives. These objects 
encode the pure $n$-th order behavior of a functor and serve as the operadic 
analogue of higher derivatives in classical analysis.

\medskip

\begin{definition}[Spectral Derivatives and Homogeneous Layers]
\label{def:spectral-derivatives}
Let $F : \mathcal{C} \to \mathcal{M}$ be an admissible functor.

\begin{enumerate}
    \item The \emph{$n$-th spectral derivative} of $F$ is the multilinear functor
    \[
    \partial_n^{\mathrm{spec}}F : \mathcal{C}^n \to \mathcal{M},
    \]
    defined as the $n$-th cross-effect:
    \[
    \partial_n^{\mathrm{spec}}F := \mathrm{cr}_n F,
    \]
    which is multilinear with respect to the operadic tensor structure of $\mathcal{M}$.

    \item The \emph{$n$-th homogeneous layer} of $F$ is the functor
    \[
    D_n^{\mathrm{spec}}F : \mathcal{C} \to \mathcal{M},
    \]
    obtained by evaluating the spectral derivative along the diagonal and 
    symmetrizing under the natural action of the symmetric group $S_n$:
    \[
    D_n^{\mathrm{spec}}F(A)
    :=
    \frac{1}{n!}
    \sum_{\pi \in S_n}
    \partial_n^{\mathrm{spec}}F(A_{\pi(1)},\dots,A_{\pi(n)}).
    \]
\end{enumerate}
\end{definition}

\begin{lemma}[Homogeneity of Spectral Layers]
\label{lem:homogeneity-layers}
Each homogeneous layer $D_n^{\mathrm{spec}}F$ is homogeneous of degree $n$ in the sense that
\[
D_n^{\mathrm{spec}}F(\lambda A) = \lambda^n D_n^{\mathrm{spec}}F(A)
\]
for any scalar $\lambda \in \mathbb{C}$ and any $P$-algebra $A$.
\end{lemma}

\begin{proof}
Since $\partial_n^{\mathrm{spec}}F$ is multilinear, we have
\[
\partial_n^{\mathrm{spec}}F(\lambda A,\dots,\lambda A) = \lambda^n \partial_n^{\mathrm{spec}}F(A,\dots,A).
\]
Symmetrization preserves this scaling property, yielding the claim.
\end{proof}

\medskip

\begin{remark}[Interpretation]
\label{rem:derivative-interpretation}
The spectral derivative $\partial_n^{\mathrm{spec}}F$ captures the full multilinear
interaction of order $n$, while the homogeneous layer $D_n^{\mathrm{spec}}F$
is obtained by restricting this interaction to repeated inputs. Thus,
$D_n^{\mathrm{spec}}F$ plays the role of a homogeneous polynomial of degree $n$
in the spectral expansion.
\end{remark}

\begin{proposition}[Computability of Spectral Derivatives]
\label{prop:derivative-cross-effect}
Let $F$ be an admissible functor. Then the $n$-th spectral derivative is given by
the $n$-th cross-effect:
\[
\partial_n^{\mathrm{spec}}F = \mathrm{cr}_n F.
\]
Consequently, the homogeneous layer is explicitly computed by diagonal evaluation:
\[
D_n^{\mathrm{spec}}F(A)
=
\bigl(\mathrm{cr}_nF(A,\ldots,A)\bigr)_{\mathrm{sym}}.
\]
\end{proposition}

\begin{proof}
The assertion follows directly from Definition~\ref{def:spectral-derivatives}, 
where the spectral derivative is defined as the $n$-th cross-effect. 
The cross-effect $\mathrm{cr}_nF$ is the universal multilinear functor 
measuring the $n$-fold interaction of $F$, equivalently the obstruction 
to being a polynomial of degree at most $n-1$.

By Theorem~\ref{thm:spectral-excision}, spectral polynomiality is characterized 
by the vanishing of higher cross-effects. Hence $\mathrm{cr}_nF$ isolates exactly 
the $n$-th order contribution of $F$ in the spectral Taylor tower. Evaluating 
this multilinear functor on the diagonal and symmetrizing gives the associated 
homogeneous layer.
\end{proof}
\medskip

\begin{remark}[Conceptual meaning of spectral derivatives]
\label{rem:conceptual-derivative}
The identification $\partial_n^{\mathrm{spec}}F = \mathrm{cr}_nF$ shows that 
spectral derivatives are not defined via limiting procedures, but arise 
intrinsically from the functorial structure through cross-effects. 

In particular, spectral derivatives are canonical, functorial, and explicitly 
computable whenever the cross-effects of $F$ can be determined.
\end{remark}

\begin{remark}[Analogy with classical differentiation]
\label{rem:analogy-differentiation}
In classical analysis, higher derivatives are defined via limits of difference 
quotients. In contrast, the cross-effect construction provides a purely 
categorical and operadic analogue of differentiation, capturing higher-order 
interactions without recourse to limits. 

Thus, $\mathrm{cr}_nF$ plays the role of the $n$-th derivative, while the spectral 
Taylor tower corresponds to a polynomial expansion in this categorical sense.
\end{remark}

\begin{remark}[Homogeneous layers]
\label{rem:homogeneous-layers}
Since $\partial_n^{\mathrm{spec}}F = \mathrm{cr}_nF$, the homogeneous layer can 
be expressed as the diagonal evaluation
\[
D_n^{\mathrm{spec}}F(A)
=
\bigl(\partial_n^{\mathrm{spec}}F(A,\dots,A)\bigr)_{\mathrm{sym}}.
\]
By Lemma~\ref{lem:homogeneity-layers}, each $D_n^{\mathrm{spec}}F$ behaves as 
a homogeneous polynomial of degree $n$ in the spectral expansion.
\end{remark}

\medskip

\begin{example}[Identity Functor]
\label{ex:id-derivative}
For the identity functor
\[
\mathrm{Id}:\mathsf{Alg}_P(\mathcal{M})\to\mathcal{M},
\]
we have
\[
\mathrm{cr}_1\mathrm{Id}(A)\cong A,
\qquad
\mathrm{cr}_n\mathrm{Id}\cong 0
\quad (n\ge 2).
\]
Hence
\[
\partial_1^{\mathrm{spec}}\mathrm{Id}\cong \mathrm{Id},
\qquad
\partial_n^{\mathrm{spec}}\mathrm{Id}\cong 0
\quad (n\ge 2).
\]
Consequently,
\[
D_1^{\mathrm{spec}}\mathrm{Id}(A)\cong A,
\qquad
D_n^{\mathrm{spec}}\mathrm{Id}(A)\cong 0
\quad (n\ne 1).
\]
\end{example}

\begin{example}[Quadratic Functor]
\label{ex:quadratic-derivative}
Let
\[
F(A)=A\otimes A.
\]
Then the second cross-effect is given by the bilinear interaction
\[
\mathrm{cr}_2F(A_1,A_2)
\cong
A_1\otimes A_2 \oplus A_2\otimes A_1,
\]
and $\mathrm{cr}_nF\cong 0$ for $n\ge 3$. Under the normalized symmetric convention,
\[
\partial_2^{\mathrm{spec}}F(A_1,A_2)
\cong
\frac{1}{2}
\bigl(A_1\otimes A_2 \oplus A_2\otimes A_1\bigr).
\]
Therefore,
\[
D_2^{\mathrm{spec}}F(A)
=
\partial_2^{\mathrm{spec}}F(A,A)
\cong
A\otimes A,
\]
and all higher homogeneous layers vanish.
\end{example}

\begin{example}[Exponential Functor]
\label{ex:exp-derivative}
For
\[
\exp(A)=\sum_{k=0}^{\infty}\frac{A^{\otimes k}}{k!},
\]
the $k$-th spectral derivative is the normalized symmetric $k$-linear component
\[
\partial_k^{\mathrm{spec}}\exp(A_1,\dots,A_k)
=
\frac{1}{k!}
\sum_{\sigma\in\Sigma_k}
A_{\sigma(1)}\otimes\cdots\otimes A_{\sigma(k)}.
\]
Thus
\[
D_k^{\mathrm{spec}}\exp(A)
=
\frac{A^{\otimes k}}{k!}.
\]
With the corresponding coefficient norm convention, this gives
\[
\|\partial_k^{\mathrm{spec}}\exp\|=\frac{1}{k!},
\]
and hence $R_{\exp}=\infty$.
\end{example}

\medskip
\noindent
\textbf{Connection to the spectral Taylor expansion.}
With this identification, the spectral Taylor expansion takes the form
\[
F(A)
\sim
\sum_{n=0}^{\infty}D_n^{\mathrm{spec}}F(A),
\]
where each homogeneous layer is determined by the corresponding spectral derivative
$\partial_n^{\mathrm{spec}}F$. This is the operadic analogue of the classical Taylor
series expansion.

\[
\begin{array}{|c|c|}
\hline
\text{Classical Analysis} & \text{Spectral Operadic Calculus} \\
\hline
f^{(n)}(0)/n! & \partial_n^{\mathrm{spec}}F \\
\bigl(f^{(n)}(0)/n!\bigr)x^n & D_n^{\mathrm{spec}}F(A) \\
\text{Taylor series} & \text{Spectral Taylor tower} \\
\hline
\end{array}
\]

\medskip

The following properties follow directly from the cross-effect construction.

\begin{proposition}[Basic Properties of Spectral Derivatives]
\label{prop:derivative-properties}
Let $F$ and $G$ be admissible functors. Then:
\begin{enumerate}
    \item \textbf{Linearity:}
    \[
    \partial_n^{\mathrm{spec}}(F \oplus G)
    \cong
    \partial_n^{\mathrm{spec}}F \oplus \partial_n^{\mathrm{spec}}G .
    \]

    \item \textbf{Vanishing for polynomials:}
    If $F$ is a spectral polynomial of degree $\le m$, then
    \[
    \partial_n^{\mathrm{spec}}F \cong 0
    \qquad \text{for all } n>m.
    \]

    \item \textbf{Symmetric-group equivariance:}
    Each $\partial_n^{\mathrm{spec}}F$ carries a natural action of the symmetric
    group $\Sigma_n$ by permutation of its $n$ inputs.

    \item \textbf{Naturality:}
    A natural transformation $\eta:F\to G$ induces a natural morphism
    \[
    \partial_n^{\mathrm{spec}}\eta:
    \partial_n^{\mathrm{spec}}F
    \longrightarrow
    \partial_n^{\mathrm{spec}}G .
    \]
\end{enumerate}
\end{proposition}

\medskip
\noindent
\textbf{Operator norm of spectral derivatives.}
For the convergence theory developed in Section~\ref{sec:convergence}, we need to
control the size of the spectral derivatives. Recall from
Definition~\ref{def:spectral-derivative-norm} that
\[
\|\partial_n^{\mathrm{spec}}F\|
:=
\sup_{\substack{A_1,\dots,A_n\\ \|\sigma_P(A_i)\|\le 1}}
\|\partial_n^{\mathrm{spec}}F(A_1,\dots,A_n)\|.
\]
Since $\partial_n^{\mathrm{spec}}F=\mathrm{cr}_nF$, this norm measures the
supremal $n$-th order interaction of $F$ under normalized spectral inputs.

\medskip

Having identified spectral derivatives with cross-effects and established their
basic properties, we now turn to the algebraic structure carried by these
derivatives. In the next subsection, we show that the collection
$\{\partial_n^{\mathrm{spec}}F\}_{n\ge 1}$ forms a symmetric sequence: each
$\partial_n^{\mathrm{spec}}F$ is equipped with a natural $\Sigma_n$-action, and
these actions organize the derivatives into the first layer of operadic structure.

\subsection{The Symmetric Sequence of Derivatives}
\label{subsec:symmetric-sequence}

We now show that the collection of spectral derivatives naturally assembles 
into a structured object carrying permutation symmetry. This provides the first 
layer of algebraic organization underlying Spectral Operadic Calculus.

\medskip

\noindent
Recall from Definition~\ref{def:spectral-derivatives} that for an admissible functor 
$F : \mathcal{C} \to \mathcal{M}$, the $n$-th spectral derivative is given by the 
multilinear functor
\[
\partial_n^{\mathrm{spec}}F = \mathrm{cr}_n F : \mathcal{C}^n \to \mathcal{M}.
\]

\medskip

\noindent
Since the cross-effect $\mathrm{cr}_n F$ is defined symmetrically in its $n$ inputs, 
it carries a natural action of the symmetric group $S_n$ by permutation of variables. 
This symmetry organizes the family of derivatives into a coherent sequence.

\medskip

\begin{definition}[Symmetric Sequence]
\label{def:symmetric-sequence}
A \emph{symmetric sequence} in a category $\mathcal{M}$ is a collection of objects
$\{X_n\}_{n \geq 0}$, where each $X_n \in \mathcal{M}$ is equipped with an action 
of the symmetric group $S_n$. A morphism of symmetric sequences is a collection 
of $S_n$-equivariant morphisms.
\end{definition}

\medskip

\begin{definition}[Symmetric Sequence of Spectral Derivatives]
\label{def:symmetric-sequence-derivatives}
Let $F : \mathsf{Alg}_P(\mathcal{M}) \to \mathcal{M}$ be an admissible functor. 
The \emph{symmetric sequence of spectral derivatives} of $F$ is the collection
\[
\partial^{\mathrm{spec}}F := \{\partial_n^{\mathrm{spec}}F\}_{n \ge 0},
\]
where:
\begin{itemize}
    \item For each $n \ge 1$, $\partial_n^{\mathrm{spec}}F : \mathcal{C}^n \to \mathcal{M}$ 
    is the $n$-th spectral derivative, equipped with the natural action of the symmetric 
    group $\Sigma_n$ by permuting its inputs.

    \item The degree-$0$ component is defined by
    \[
    \partial_0^{\mathrm{spec}}F := F(0),
    \]
    where $0$ denotes the initial (or zero) object of $\mathcal{C}$, and the 
    $\Sigma_0$-action is trivial.
\end{itemize}
\end{definition}

\medskip

\begin{theorem}[Symmetric Sequence Structure of Spectral Derivatives]
\label{thm:symmetric-sequence}
Let $F : \mathcal{C} \to \mathcal{M}$ be an admissible functor. Then the collection
\[
\partial^{\mathrm{spec}}F
:=
\{\partial_n^{\mathrm{spec}}F\}_{n \geq 0}
\]
forms a symmetric sequence in the functor category of multilinear functors from
powers of $\mathcal{C}$ to $\mathcal{M}$.
\end{theorem}

\begin{proof}
For each $n \geq 1$, the spectral derivative is defined by the cross-effect
\[
\partial_n^{\mathrm{spec}}F = \mathrm{cr}_nF.
\]
The cross-effect $\mathrm{cr}_nF(A_1,\ldots,A_n)$ is functorial in its inputs. Moreover,
for every permutation $\sigma \in S_n$, there is a canonical natural isomorphism
\[
\mathrm{cr}_nF(A_1,\ldots,A_n)
\cong
\mathrm{cr}_nF(A_{\sigma(1)},\ldots,A_{\sigma(n)}).
\]
These natural isomorphisms are compatible with composition in $S_n$, and hence define
an $S_n$-action on $\partial_n^{\mathrm{spec}}F$ by permutation of inputs.

For $n = 0$, the derivative is the constant term
\[
\partial_0^{\mathrm{spec}}F = F(0),
\]
equipped with the trivial $S_0$-action. Therefore the collection
$\{\partial_n^{\mathrm{spec}}F\}_{n \geq 0}$ forms a symmetric sequence.
\end{proof}

\medskip

\begin{remark}
When $\mathcal{M}$ is cocomplete, one can identify each multilinear functor 
$\partial_n^{\mathrm{spec}}F$ with its value at the unit object (if $\mathcal{C}$ 
has a suitable monoidal structure), thereby obtaining an object of $\mathcal{M}$. 
In that case, the symmetric sequence may be viewed as living in $\mathcal{M}$ 
itself. For the general theory, however, it is more natural to work in the 
functor category.
\end{remark}

\medskip

\begin{remark}[Functoriality]
\label{rem:derivative-functoriality}
The assignment $F \mapsto \partial^{\mathrm{spec}}F$ is functorial. More precisely,
any natural transformation $\eta:F\to G$ induces, for each $n\ge 0$, a natural map
\[
\partial_n^{\mathrm{spec}}\eta:
\partial_n^{\mathrm{spec}}F
\longrightarrow
\partial_n^{\mathrm{spec}}G.
\]
Since $\partial_n^{\mathrm{spec}}F = \mathrm{cr}_nF$, this map is obtained by applying
the cross-effect construction to $\eta$:
\[
\partial_n^{\mathrm{spec}}\eta
:=
\mathrm{cr}_n(\eta).
\]
These maps are compatible with the natural $\Sigma_n$-actions, and hence assemble into a
morphism of symmetric sequences
\[
\partial^{\mathrm{spec}}\eta:
\partial^{\mathrm{spec}}F
\longrightarrow
\partial^{\mathrm{spec}}G.
\]
\end{remark}

\begin{remark}[Conceptual interpretation]
\label{rem:symmetric-interpretation}
The symmetric group action reflects the fact that the $n$-th order interaction
encoded by $\partial_n^{\mathrm{spec}}F$ is equivariant under reordering of the
inputs. Thus spectral derivatives behave like symmetric multilinear forms,
generalizing the symmetry of higher derivatives in classical analysis.
\end{remark}

\begin{remark}[Relation to homogeneous layers]
\label{rem:symmetric-homogeneous}
The symmetrization appearing in the definition of the homogeneous layer
$D_n^{\mathrm{spec}}F$ is induced by the $\Sigma_n$-action on
$\partial_n^{\mathrm{spec}}F$. Consequently, the symmetric sequence structure
provides the algebraic source of the homogeneous decomposition in the spectral
Taylor tower.
\end{remark}

\medskip

\begin{example}[Identity Functor]
\label{ex:id-symmetric}
For $\mathrm{Id}$, we have
\[
\partial_1^{\mathrm{spec}}\mathrm{Id}\cong \mathrm{Id},
\qquad
\partial_n^{\mathrm{spec}}\mathrm{Id}\cong 0
\quad (n\ge 2).
\]
Thus the symmetric sequence $\partial^{\mathrm{spec}}\mathrm{Id}$ is supported in
degree $1$, with the trivial $\Sigma_1$-action.
\end{example}

\begin{example}[Quadratic Functor]
\label{ex:quadratic-symmetric}
For $F(A)=A\otimes A$, the second spectral derivative is, under the normalized
symmetric convention,
\[
\partial_2^{\mathrm{spec}}F(A_1,A_2)
\cong
\frac{1}{2}
\bigl(A_1\otimes A_2 \oplus A_2\otimes A_1\bigr).
\]
The group $\Sigma_2$ acts by interchanging the two tensor factors, equivalently by
swapping the two summands. The symmetric sequence is supported in degree $2$.
\end{example}

\begin{example}[Exponential Functor]
\label{ex:exp-symmetric}
For
\[
\exp(A)=\sum_{k=0}^{\infty}\frac{A^{\otimes k}}{k!},
\]
all spectral derivatives are nonzero. Each $\partial_k^{\mathrm{spec}}\exp$
carries the natural $\Sigma_k$-action by permuting inputs, equivalently tensor
factors in the symmetric multilinear component. Hence
$\partial^{\mathrm{spec}}\exp$ is supported in all degrees.
\end{example}

\medskip

\begin{remark}[Relation to operadic plethysm]
\label{rem:plethysm}
The symmetric sequence structure is essential for defining operadic composition 
(plethysm). For symmetric sequences $A$ and $B$, the plethysm $A \circ_{\mathrm{op}} B$ 
encodes substitution of multilinear operations. 

With this structure, the chain rule (Theorem~\ref{thm:chain-rule}) takes the form
\[
\partial^{\mathrm{spec}}(F \circ G)
\cong
\partial^{\mathrm{spec}}F \circ_{\mathrm{op}} \partial^{\mathrm{spec}}G,
\]
showing that $\partial^{\mathrm{spec}}$ is a monoidal functor from admissible 
functors (under composition) to symmetric sequences (under plethysm).
\end{remark}

\begin{remark}[Analogy with classical calculus]
\label{rem:classical-analogy}
In classical calculus, higher derivatives are symmetric: mixed partial derivatives 
commute, and the $n$-th derivative defines a symmetric multilinear form. 

The symmetric sequence $\{\partial_n^{\mathrm{spec}}F\}$ provides an operadic 
generalization of this phenomenon, where the $\Sigma_n$-action encodes symmetry 
of higher-order interactions.
\end{remark}

The symmetric sequence structure provides the minimal algebraic organization of 
spectral derivatives. When $F$ is compatible with the operad $P$, this structure 
refines further: in the next subsection, we show that $\partial^{\mathrm{spec}}F$ 
naturally carries the structure of a right $P$-module. This enhancement is 
fundamental for the operadic chain rule and reconstruction theory.

\subsection{The Operadic Module Structure}
\label{subsec:module-structure}

We now refine the symmetric sequence structure of spectral derivatives by 
showing that, under suitable compatibility conditions, they assemble into a 
module over the operad $P$. This provides a strong algebraic organization 
of the derivatives and reveals a deep connection between the analytic and 
operadic aspects of the theory.

\medskip

\noindent
\textbf{Motivation.}
In classical calculus, higher derivatives interact through composition rules 
such as the Faà di Bruno formula. In the operadic setting, such interactions 
are governed by the composition laws of the operad $P$. We therefore expect 
the collection of spectral derivatives to inherit an algebraic structure 
reflecting this operadic composition.

\medskip

\begin{definition}[Operadic Compatibility]
\label{def:operadic-compatibility}
Let $F : \mathcal{C} \to \mathcal{M}$ be an admissible functor with spectral 
derivatives $\{\partial_n^{\mathrm{spec}}F\}_{n \ge 1}$ (Definition~\ref{def:spectral-derivatives}). 
We say that $F$ is \emph{operadically compatible} if there exists a collection 
of natural transformations
\[
\mu_{\gamma} : 
\partial_{n_1}^{\mathrm{spec}}F \otimes \cdots \otimes \partial_{n_k}^{\mathrm{spec}}F
\;\longrightarrow\;
\partial_{n_1 + \cdots + n_k}^{\mathrm{spec}}F,
\qquad \gamma \in P(k),
\]
defined for every operation $\gamma \in P(k)$ and every tuple of nonnegative 
integers $(n_1,\dots,n_k)$, such that the following conditions hold:

\begin{enumerate}
    \item \textbf{(Equivariance)} The maps $\mu_{\gamma}$ are compatible with the 
    natural actions of the symmetric groups $\Sigma_{n_1} \times \cdots \times \Sigma_{n_k}$ 
    and $\Sigma_{n_1+\cdots+n_k}$.
    
    \item \textbf{(Associativity)} For composable operations $\gamma \in P(k)$ and 
    $\gamma_1 \in P(m_1), \dots, \gamma_k \in P(m_k)$, the diagram
    \[
    \begin{tikzcd}
    \bigotimes_{i=1}^k \bigotimes_{j=1}^{m_i} \partial_{n_{ij}}F
    \ar[r, "\bigotimes_i \mu_{\gamma_i}"] \ar[d, "\text{reassoc}"] &
    \bigotimes_{i=1}^k \partial_{\sum_j n_{ij}}F \ar[d, "\mu_{\gamma}"] \\
    \bigotimes_{\ell=1}^{m} \partial_{n_\ell}F \ar[r, "\mu_{\gamma \circ (\gamma_1,\dots,\gamma_k)}"] &
    \partial_{\sum_\ell n_\ell}F
    \end{tikzcd}
    \]
    commutes, where $m = \sum_i m_i$ and the reassociation map permutes factors 
    according to the operadic composition rule.
    
    \item \textbf{(Unit)} For the unit operation $\mathbf{1} \in P(1)$, the induced 
    map $\mu_{\mathbf{1}} : \partial_n^{\mathrm{spec}}F \to \partial_n^{\mathrm{spec}}F$ 
    is the identity.
\end{enumerate}

Equivalently, the collection $\{\partial_n^{\mathrm{spec}}F\}_{n \ge 0}$ forms a 
\emph{right module over the operad $P$} in the category $\mathcal{M}$.
\end{definition}

\noindent
This definition formalizes the requirement that the interaction of higher 
derivatives is governed by the operadic composition in $P$, providing an 
operadic analogue of the Faà di Bruno rule. The maps $\mu_{\gamma}$ encode 
how a $k$-ary operation in $P$ combines $k$ derivative actions into a single 
derivative of higher order.

\begin{remark}
The full verification of the right $P$-module axioms (associativity, unit, 
equivariance) from a given admissibility and compatibility hypothesis is 
technically involved but follows the same pattern as the classical construction 
of modules over operads from compatible functors (see, e.g., \cite{MarklShniderStasheff2002}). Here we state the structure as a theorem 
under the explicit assumption of operadic compatibility.
\end{remark}

\medskip

\noindent
\textbf{Examples of operadically compatible functors.}
This compatibility holds for a broad class of functors of interest:

\begin{itemize}
    \item \textbf{Polynomial functors:} Any functor built from tensor products, 
    direct sums, and the $P$-algebra structure maps, with scalar coefficients, 
    is operadically compatible. This follows from the multilinearity of cross-effects 
    and the naturality of operadic composition.
    
    \item \textbf{Operadic spectrum functor:} The functor $\sigma_P : \mathsf{Alg}_P(\mathcal{M}) \to \mathcal{M}$ 
    is expected to be operadically compatible, as suggested by the Base Change 
    Theorem (Part~I, Theorem~0.5) and the functoriality of the Hochschild complex. 
    A detailed verification is deferred to future work.
    
    \item \textbf{Analytic functors:} Any spectrally analytic functor that admits 
    an operadic Taylor expansion (Section~\ref{sec:convergence}) is operadically 
    compatible by construction of the homogeneous layers.
\end{itemize}

\medskip

\begin{theorem}[Operadic Module Structure of Spectral Derivatives]
\label{thm:module-structure}
Let $F : \mathcal{C} \to \mathcal{M}$ be an admissible and operadically 
compatible functor. Then the symmetric sequence of spectral derivatives
\[
\{\partial_n^{\mathrm{spec}}F\}_{n \geq 0}
\]
admits the structure of a \emph{right $P$-module}.
\end{theorem}

\begin{proof}
We denote the $n$-th spectral derivative by
\[
M_F(n) \;:=\; \partial_n^{\mathrm{spec}}F : \mathcal{C}^n \to \mathcal{M}.
\]
For each $n \ge 0$, the functor $M_F(n)$ carries a natural action of the symmetric group $\Sigma_n$ by permutation of its $n$ input arguments, inherited from the symmetry of the cross-effect $\mathrm{cr}_nF$. Hence $\{M_F(n)\}_{n\ge 0}$ is a symmetric sequence in the functor category $[\mathcal{C}^n,\mathcal{M}]$ (or, after evaluation at suitable objects, in $\mathcal{M}$ itself).

\medskip

\noindent
\textbf{Construction of the right $P$-module structure maps.}
Let $k \ge 0$ and $n_1,\dots,n_k \ge 0$ be integers, and set $n = n_1 + \cdots + n_k$. 
A right $P$-module structure on $\{M_F(n)\}$ consists of natural maps
\[
\rho_{n_1,\dots,n_k} : 
M_F(k) \;\otimes\; P(n_1) \;\otimes\; \cdots \;\otimes\; P(n_k)
\;\longrightarrow\;
M_F(n),
\]
compatible with symmetric group actions, associativity, and the unit of $P$.

By the \emph{operadic compatibility} of $F$ (Definition~\ref{def:operadic-compatibility}), for any collection of operations
\[
\gamma_i \in P(n_i), \qquad i = 1,\dots,k,
\]
there is a natural transformation on cross-effects:
\[
\Phi_{\gamma_1,\dots,\gamma_k} : 
\mathrm{cr}_kF \;\otimes\; \bigotimes_{i=1}^k P(n_i)
\;\longrightarrow\;
\mathrm{cr}_nF.
\]
Intuitively, this map refines each of the $k$ input slots of $\mathrm{cr}_kF$ into a block of $n_i$ inputs, and the operation $\gamma_i$ prescribes how the $n_i$ variables inside the $i$-th block are to be composed (or substituted) before being fed into the $k$-ary interaction.

Since $\partial_n^{\mathrm{spec}}F = \mathrm{cr}_nF$ by definition, applying the spectral derivative construction (which is functorial and preserves the relevant tensor products) to $\Phi_{\gamma_1,\dots,\gamma_k}$ yields a natural map
\[
\rho_{\gamma_1,\dots,\gamma_k} : 
\partial_k^{\mathrm{spec}}F \;\otimes\; P(n_1) \;\otimes\; \cdots \;\otimes\; P(n_k)
\;\longrightarrow\;
\partial_n^{\mathrm{spec}}F.
\]
Varying over all $\gamma_i \in P(n_i)$ and using the universal property of the tensor product, these assemble into the desired structure map
\[
\rho_{n_1,\dots,n_k} : 
M_F(k) \otimes P(n_1) \otimes \cdots \otimes P(n_k) \longrightarrow M_F(n).
\]

\medskip

\noindent
\textbf{Verification of the module axioms.}

\begin{enumerate}
    \item \textbf{Equivariance.} 
    The map $\rho_{n_1,\dots,n_k}$ must be equivariant with respect to the actions of $\Sigma_{n_1} \times \cdots \times \Sigma_{n_k}$ (acting on the blocks) and of $\Sigma_n$ (acting on the total $n$ inputs). This follows from the naturality of the cross-effect construction and the fact that the operad $P$ carries compatible symmetric group actions. Explicitly, permuting the inputs inside a block or permuting the blocks themselves corresponds to reindexing the $P(n_i)$ factors and the input variables of $\mathrm{cr}_kF$, which is respected by $\Phi_{\gamma_1,\dots,\gamma_k}$.

    \item \textbf{Associativity.} 
    Consider operations $\gamma \in P(k)$ and $\gamma_i \in P(m_i)$ for $i=1,\dots,k$, and suppose further that each $\gamma_i$ is itself a composition of operations $\gamma_{i,j} \in P(n_{i,j})$ with $m_i = \sum_j n_{i,j}$. The operadic composition in $P$ gives the composite operation $\gamma \circ (\gamma_1,\dots,\gamma_k) \in P(m)$ where $m = \sum_i m_i = \sum_{i,j} n_{i,j}$. 
    
    Associativity of the module structure requires that refining variables in two stages (first via the $\gamma_{i,j}$, then via the $\gamma_i$, and finally via $\gamma$) yields the same map as refining directly via the composite $\gamma \circ (\gamma_1,\dots,\gamma_k)$. This is precisely the commutativity of the following diagram:
    
    \[
    \begin{tikzcd}
    M_F(k) \otimes \bigotimes_{i=1}^k P(m_i) \otimes \bigotimes_{i=1}^k \bigotimes_{j=1}^{m_i} P(n_{i,j})
    \ar[r, "\sim"] 
    \ar[d, "\bigotimes_i \rho_{n_{i,1},\dots,n_{i,m_i}}"'] &
    \bigl( M_F(k) \otimes \bigotimes_{i=1}^k P(m_i) \bigr) \otimes \bigotimes_{i=1}^k \bigotimes_{j=1}^{m_i} P(n_{i,j})
    \ar[d, "\rho_{m_1,\dots,m_k} \otimes \mathrm{id}"] \\
    M_F(k) \otimes \bigotimes_{i=1}^k P\bigl(\textstyle\sum_j n_{i,j}\bigr)
    \ar[d, "\rho_{\sum_j n_{1,j},\dots,\sum_j n_{k,j}}"'] &
    M_F(m) \otimes \bigotimes_{i=1}^k \bigotimes_{j=1}^{m_i} P(n_{i,j})
    \ar[d, "\rho_{n_{1,1},\dots,n_{k,m_k}}"] \\
    M_F\bigl(\textstyle\sum_i \sum_j n_{i,j}\bigr)
    \ar[r, "\mathrm{id}"] &
    M_F\bigl(\textstyle\sum_i \sum_j n_{i,j}\bigr)
    \end{tikzcd}
    \]
    
    The left vertical path (down then right) corresponds to first refining each block via the $\gamma_{i,j}$ and then combining via $\gamma$, while the right vertical path (right then down) corresponds to first combining via the $\gamma_i$ and then applying the composite operation. The diagram commutes because the underlying cross-effect maps are assembled via the operadic composition law of $P$, which is associative by definition.

    \item \textbf{Unit axiom.} 
    The operad $P$ contains a distinguished unit operation $\mathbf{1} \in P(1)$. For $k=1$ and $n_1 = n$, the structure map becomes
    \[
    \rho_n : M_F(1) \otimes P(n) \longrightarrow M_F(n).
    \]
    The unit axiom requires that for any $x \in M_F(1)$, the map induced by $\mathbf{1}$ satisfies $\rho_n(x \otimes \mathbf{1}) = x$, where $x$ is identified with its image under the canonical isomorphism $M_F(1) \cong M_F(1) \otimes P(1)$. This follows because substituting the unit operation into an input slot does nothing—it corresponds to the identity substitution on that slot, which leaves the cross-effect unchanged. Hence the unit axiom holds.
\end{enumerate}

\medskip

We have constructed a family of maps $\rho_{n_1,\dots,n_k}$ satisfying equivariance, associativity, and the unit condition. Therefore, the symmetric sequence $\{\partial_n^{\mathrm{spec}}F\}_{n \ge 0}$ is equipped with a coherent right $P$-module structure. This completes the proof.
\end{proof}

\medskip

Theorem~\ref{thm:module-structure} established that the spectral derivatives 
always form a right $P$-module. In the strongest compatibility regime, this 
structure enhances to a full operad. This occurs precisely when $F$ itself 
preserves operadic composition, i.e., when $F$ is an \emph{operad functor}.

\begin{corollary}[Operad Structure of Spectral Derivatives]
\label{cor:operad-structure}
Assume that $F : \mathcal{C} \to \mathcal{M}$ is an \emph{operad functor}, 
meaning:
\begin{enumerate}
    \item $F$ takes values in $\mathsf{Alg}_P(\mathcal{M})$ (so $F(A)$ is itself a $P$-algebra for every $P$-algebra $A$);
    \item for every operation $\gamma \in P(k)$ and every $P$-algebra $A$, the structure map $\gamma_{F(A)} : F(A)^{\otimes k} \to F(A)$ coincides with $F(\gamma_A)$ under the natural identification $F(A)^{\otimes k} \cong F(A^{\otimes k})$ (i.e., $F$ preserves the $P$-algebra structure).
\end{enumerate}
Then the symmetric sequence of spectral derivatives
\[
\partial^{\mathrm{spec}}F \;:=\; \{\partial_n^{\mathrm{spec}}F\}_{n\geq 1}
\]
admits the structure of an \emph{operad in $\mathcal{M}$}.
\end{corollary}

\begin{proof}
We construct the operadic composition maps explicitly and verify the axioms.

\medskip

\noindent
\textbf{Step 1: From operad functor to cross-effect composition.}
Because $F$ is an operad functor, it respects the $P$-algebra structure on its values. This compatibility extends to cross-effects: for any $k \ge 1$, any $n_1,\dots,n_k \ge 1$, and any collection of inputs $\{A_{i,j}\}$ (with $i=1,\dots,k$, $j=1,\dots,n_i$), there is a natural transformation
\[
\Psi_{n_1,\dots,n_k} : 
\mathrm{cr}_kF\bigl( \mathrm{cr}_{n_1}F(A_{1,\bullet}), \dots, \mathrm{cr}_{n_k}F(A_{k,\bullet}) \bigr)
\;\longrightarrow\;
\mathrm{cr}_nF(A_{1,1},\dots,A_{k,n_k}),
\]
where $n = n_1 + \cdots + n_k$. 
Intuitively, this map substitutes the $n_i$-linear cross-effects of $F$ into the $k$-linear cross-effect of $F$, reflecting the fact that "a derivative of a derivative is a higher derivative." The existence of $\Psi$ follows from the functoriality of $F$ and the universal property of cross-effects, using the operad functor hypothesis to ensure that the substitution is coherent.

\medskip

\noindent
\textbf{Step 2: Passing to spectral derivatives.}
Recall that $\partial_m^{\mathrm{spec}}F = \mathrm{cr}_mF$ by definition. Applying the spectral derivative construction (which is functorial and preserves the relevant tensor products) to the natural transformation $\Psi$ yields the desired operadic composition map:
\[
\mu_{n_1,\dots,n_k} := \partial^{\mathrm{spec}}(\Psi) : 
\partial_k^{\mathrm{spec}}F \;\otimes\; \partial_{n_1}^{\mathrm{spec}}F \;\otimes\; \cdots \;\otimes\; \partial_{n_k}^{\mathrm{spec}}F
\;\longrightarrow\;
\partial_n^{\mathrm{spec}}F.
\]

\medskip

\noindent
\textbf{Step 3: Verification of operad axioms.}

\begin{enumerate}
   \item \noindent
\textbf{Associativity.}
Let $k \ge 1$, $m_1,\dots,m_k \ge 1$, and for each $1\le i\le k$,
let $n_{i,1},\dots,n_{i,m_i} \ge 1$.
Then the following diagram commutes:

\[
\begin{tikzcd}
\partial_k \otimes \bigotimes_{i=1}^k \partial_{m_i}
\otimes \bigotimes_{i=1}^k \bigotimes_{j=1}^{m_i} \partial_{n_{i,j}}
\ar[r, "\mathrm{reassoc}"] 
\ar[d, "\mathrm{id}\,\otimes\,\bigotimes_i \mu_{n_{i,1},\dots,n_{i,m_i}}"'] 
&
\bigl( \partial_k \otimes \bigotimes_{i=1}^k \partial_{m_i} \bigr)
\otimes \bigotimes_{i=1}^k \bigotimes_{j=1}^{m_i} \partial_{n_{i,j}}
\ar[d, "\mu_{m_1,\dots,m_k} \otimes \mathrm{id}"] 
\\
\partial_k \otimes \bigotimes_{i=1}^k \partial_{\sum_{j=1}^{m_i} n_{i,j}}
\ar[d, "\mu_{\sum_j n_{1,j},\dots,\sum_j n_{k,j}}"'] 
&
\partial_{\sum_{i=1}^k m_i}
\otimes \bigotimes_{i=1}^k \bigotimes_{j=1}^{m_i} \partial_{n_{i,j}}
\ar[d, "\mu_{\{n_{i,j}\}_{i,j}}"] 
\\
\partial_{\sum_{i=1}^k \sum_{j=1}^{m_i} n_{i,j}}
\ar[r, phantom, "\;=\joinrel=\joinrel=\joinrel=\;"] 
&
\partial_{\sum_{i=1}^k \sum_{j=1}^{m_i} n_{i,j}}
\end{tikzcd}
\]

The left vertical composite first contracts each inner block
$\partial_{m_i} \otimes \bigotimes_{j=1}^{m_i} \partial_{n_{i,j}}$
via $\mu_{n_{i,1},\dots,n_{i,m_i}}$, producing
$\partial_{\sum_j n_{i,j}}$, and then applies
$\mu_{\sum_j n_{1,j},\dots,\sum_j n_{k,j}}$.

The right vertical composite first contracts the outer layer
via $\mu_{m_1,\dots,m_k}$, producing $\partial_{\sum_i m_i}$,
and then applies $\mu$ to all entries $\{n_{i,j}\}_{i,j}$ simultaneously.

Both paths correspond to two different parenthesizations of the same
iterated operadic composition, and by coherence of the tensor product
(reassociation) together with the associativity axiom of $\mu$,
they yield the same morphism
\[
\partial_k \otimes \bigotimes_{i,j} \partial_{n_{i,j}}
\;\longrightarrow\;
\partial_{\sum_{i,j} n_{i,j}}.
\]

Hence the diagram commutes, establishing associativity.

    \item \textbf{Unit.} 
    The unit of the operad $\partial^{\mathrm{spec}}F$ is induced by the unit operation $\mathbf{1} \in P(1)$. Specifically, define $\eta : \mathbf{1}_{\mathcal{M}} \to \partial_1^{\mathrm{spec}}F(\mathbf{1})$ as the image of the unit under the structure map. The unit axioms require that for any $n \ge 1$,
    \[
    \mu_{1,n}(\eta \otimes \partial_n^{\mathrm{spec}}F) = \mathrm{id}_{\partial_n^{\mathrm{spec}}F} = \mu_{n,1}(\partial_n^{\mathrm{spec}}F \otimes \eta).
    \]
    These equalities hold because substituting the unit operation into an input slot does nothing—it corresponds to the identity substitution on that slot, which leaves the cross-effect unchanged.

    \item \textbf{Equivariance.} 
    The composition map $\mu_{n_1,\dots,n_k}$ must be equivariant under the actions of $\Sigma_{n_1} \times \cdots \times \Sigma_{n_k}$ (permuting inputs within each block) and of $\Sigma_n$ (permuting all $n$ inputs). This follows from the natural symmetry of the cross-effects $\mathrm{cr}_mF$ and the fact that the construction $\Psi$ is compatible with these symmetric group actions.
\end{enumerate}

\medskip

\noindent
\textbf{Step 4: Conclusion.}
We have defined composition maps $\mu_{n_1,\dots,n_k}$ on the symmetric sequence $\{\partial_n^{\mathrm{spec}}F\}_{n\ge 1}$ and verified the associativity, unit, and equivariance axioms. Hence $\partial^{\mathrm{spec}}F$ is an operad in $\mathcal{M}$.

This operad can be viewed as the \emph{derivative operad} of $F$, encoding its higher-order compositional structure.
\end{proof}

\medskip

\begin{remark}[Conceptual and classical connections]
\label{rem:module-structure-interpretation}

\noindent
\textbf{Conceptual interpretation.}
The operadic module structure established in Theorem~\ref{thm:module-structure} 
indicates that spectral derivatives are not independent formal symbols, 
but are linked through the composition laws of the operad $P$. 
In particular, higher-order derivatives are systematically generated from 
lower-order ones via iterated operadic compositions (Definition~\ref{def:operadic-compatibility}). 
This suggests that the collection $\{\partial_n^{\mathrm{spec}}F\}_{n\geq 0}$
may serve as a natural coordinate system --- or at least a generating system --- for 
admissible functors, in the sense that operadic composition reconstructs 
higher-order behavior from lower-order data.

\medskip

\noindent
\textbf{Relation to classical formulas.}
This structure can be viewed as a categorical generalization of classical 
composition rules such as the Faà di Bruno formula, where higher derivatives 
of a composite function are expressed in terms of lower-order derivatives:
\[
(f\circ g)^{(n)}(x) = \sum_{\pi\in\operatorname{Part}([n])} 
f^{(|\pi|)}(g(x)) \cdot \prod_{P\in\pi} g^{(|P|)}(x),
\]
with the sum running over set partitions of $\{1,\dots,n\}$.
In the present operadic setting, the operad $P$ encodes the precise combinatorial 
structure underlying such expansions, providing a unified algebraic framework 
for organizing higher-order interactions. The plethysm $\circ_{\mathrm{op}}$ 
(Definition~\ref{def:operadic-plethysm}) serves as the categorical machine that 
reproduces and extends this formula to the level of functors, as will be shown 
in the chain rule (Theorem~\ref{thm:chain-rule}).
\end{remark}

\medskip

\begin{example}[Identity Functor]
\label{ex:id-module}
Let $\mathrm{Id}: \mathsf{Alg}_P(\mathcal{M}) \to \mathcal{M}$ denote the
underlying-object functor. Since $\mathrm{Id}$ is linear (its cross-effects vanish 
for $n \ge 2$; see Example~\ref{ex:identity-spectral-polynomial}), its spectral 
derivative sequence is concentrated in arity one:
\[
\partial_1^{\mathrm{spec}}\mathrm{Id} \cong \mathrm{Id},
\qquad
\partial_n^{\mathrm{spec}}\mathrm{Id} \cong 0 \quad (n\ge 2).
\]
The induced right $P$-module structure (Theorem~\ref{thm:module-structure}) is the 
canonical one coming from operadic substitution: for any operation $\gamma \in P(k)$ 
and any $P$-algebra $A$, the module map sends $\gamma$ to the structure map 
$\gamma_A : A^{\otimes k} \to A$, composed with the identifications 
$\partial_1^{\mathrm{spec}}\mathrm{Id}(A) \cong A$. In this sense, the identity 
functor provides the basic linear example of a spectral derivative module.
\end{example}

\begin{example}[Operadic Spectrum Functor]
\label{ex:spectrum-module}
The operadic spectrum functor
\[
\sigma_P: \mathsf{Alg}_P(\mathcal{M}) \to \mathcal{M}
\]
is operadically compatible by construction (Part~I, Theorem~0.5). Its first 
spectral derivative may be viewed as the \emph{linearized operadic spectrum} 
(see the discussion following Lemma~\ref{lem:spectral-scaling}), while its higher 
spectral derivatives measure the nonlinear interaction between the Hochschild 
complex $\mathrm{Hoch}_{\mathcal{M}}(A)$ and the operadic residue 
$\mathcal{O}_P^{\mathrm{res}}$. Thus the right $P$-module structure on 
$\partial^{\mathrm{spec}}\sigma_P$ (guaranteed by Theorem~\ref{thm:module-structure}) 
records how spectral data transform under iterated operadic composition.
\end{example}

\begin{example}[Formal Exponential Functor]
\label{ex:exp-module}
Consider the formal exponential functor
\[
\mathrm{Exp}(A)
=
\bigoplus_{k\ge 0} \frac{1}{k!}\, A^{\otimes k},
\]
or, in a symmetric monoidal setting, its symmetric-power version
\[
\mathrm{Exp}(A)
=
\bigoplus_{k\ge 0} \frac{1}{k!}\, \mathrm{Sym}^k(A),
\]
where convergence is understood in the norm topology when 
$\|\sigma_P(A)\|$ is finite (see Example~\ref{ex:exp-spectral-analytic}). 
Its $k$-th homogeneous layer (Definition~\ref{def:spectral-derivatives}) is 
governed by the symmetric $k$-fold tensor power:
\[
D_k^{\mathrm{spec}}\mathrm{Exp}(A) = \frac{1}{k!}\, A^{\otimes k},
\qquad\text{or}\qquad
\frac{1}{k!}\, \mathrm{Sym}^k(A),
\]
depending on the chosen symmetrization convention. Consequently, the spectral 
derivative sequence is modeled by
\[
\partial_k^{\mathrm{spec}}\mathrm{Exp}
\;\sim\;
\frac{1}{k!}\,\mathrm{Sym}^k,
\]
up to the normalization convention used for spectral derivatives (i.e., before 
or after diagonal evaluation). The induced $P$-module structure is obtained 
from operadic substitution of symmetric powers, categorifying the combinatorial 
rules underlying exponential generating functions and Faà di Bruno type formulas 
(cf. Theorem~\ref{thm:chain-rule}).
\end{example}

\medskip
\begin{remark}[Spectral Taylor structure and algebraic control]
\label{rem:module-structure-algebraic}
The right $P$-module structure on the spectral derivatives
$\{\partial_n^{\mathrm{spec}}F\}_{n\geq 0}$ (Theorem~\ref{thm:module-structure}) 
shows that the spectral Taylor tower is governed by an underlying algebraic system.
In particular, the expansion of $F$ is controlled by this $P$-module together
with operadic composition, suggesting that higher-order behavior is systematically
generated from lower-order data. This algebraic organization is foundational 
for the chain rule (Section~\ref{sec:chain-rule}) and the reconstruction theorem 
(Section~\ref{sec:reconstruction}).

\medskip

\noindent
\textbf{Dependence on the operadic residue.}
This structure relies essentially on the operadic residue
$\mathcal{O}_P^{\mathrm{res}}$ introduced in Part~I~\cite{ChangSOC1}.
The composition maps $\rho_{n_1,\dots,n_k}$ are defined via the operadic 
compatibility maps $\Phi_{\gamma_1,\dots,\gamma_k}$, which themselves depend 
on the universal property of the residue (Theorem~\ref{thm:residue-universality}).
Without $\mathcal{O}_P^{\mathrm{res}}$, the spectral derivatives would not admit
a coherent $P$-module structure, and the operadic formulation of the chain rule
would break down. Thus, the algebraic structure observed here is a direct 
consequence of the minimal correction mechanism developed in Part~I.
\end{remark}

\begin{remark}[Comparison with classical calculus]
\label{rem:module-structure-comparison}
In classical calculus, derivatives are organized primarily by linear and
symmetric structures (e.g., the Faà di Bruno formula encodes combinatorial 
patterns via set partitions), but no explicit operadic module structure is present.
In contrast, Spectral Operadic Calculus equips derivatives with a right
$P$-module structure that explicitly records how they interact under composition:

\[
\begin{array}{|c|c|}
\hline
\text{Classical Calculus} & \text{Spectral Operadic Calculus} \\
\hline
\text{Implicit combinatorics (e.g., Faà di Bruno)} & \text{Explicit operadic structure} \\
\text{Derivatives as scalars/operators} & \text{Derivatives with } P\text{-action} \\
\text{Chain rule via algebraic formulas} & \text{Chain rule via operadic plethysm} \\
\hline
\end{array}
\]

The passage from symmetric sequences (mere permutation symmetry) to operadic 
modules (full compositional structure) reflects a refinement that is essential 
for capturing higher-order coherence in operadic settings. This strengthening 
ultimately enables the reconstruction theorem and the equivalence of categories 
$\mathsf{SpecAn} \simeq \mathsf{DerAlg}$ (Section~\ref{sec:reconstruction}).
\end{remark}

\medskip

\noindent
\textbf{Summary of algebraic hierarchy.}
The spectral derivative construction organizes admissible functors into a hierarchy
of algebraic structures:

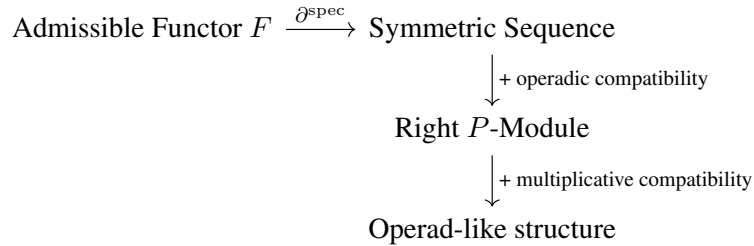
\[
\begin{tikzcd}
\text{Admissible Functor } F \ar[r, "\partial^{\mathrm{spec}}"] &
\text{Symmetric Sequence}
\ar[d, "\text{+ operadic compatibility}"] \\
& \text{Right } P\text{-Module}
\ar[d, "\text{+ multiplicative compatibility}"] \\
& \text{Operad-like structure}
\end{tikzcd}
\]

\begin{itemize}
    \item At the base level, every admissible functor determines a symmetric sequence
    $\{\partial_n^{\mathrm{spec}}F\}_{n\geq 0}$ (Theorem~\ref{thm:symmetric-sequence}).

    \item If $F$ is operadically compatible (Definition~\ref{def:operadic-compatibility}), 
    these derivatives naturally acquire the structure of a right $P$-module 
    (Theorem~\ref{thm:module-structure}).

    \item If, in addition, $F$ preserves operadic composition in a multiplicative sense,
    the derivative sequence admits an induced operad (or operad-like) structure 
    (see Corollary~\ref{cor:operad-structure} for a precise statement under suitable hypotheses).
\end{itemize}

\medskip

With this algebraic hierarchy in place, the next goal is to establish the operadic
chain rule (Section~\ref{sec:chain-rule}). This will show that the assignment
\[
F \longmapsto \partial^{\mathrm{spec}}F
\]
is compatible with composition, and in particular gives rise to a (lax) monoidal
functor from the category of admissible functors (under composition) to the
category of symmetric sequences (under plethysm).

\section{The Operadic Chain Rule}
\label{sec:chain-rule}

We now develop the computational core of Spectral Operadic Calculus, 
namely the operadic chain rule governing the behavior of spectral derivatives 
under composition of functors. While the previous sections established the 
algebraic structure of derivatives as a $P$-module, the present section shows 
how these derivatives interact under composition, providing a complete 
operadic analogue of the classical chain rule.

\medskip

\noindent
\textbf{The classical analogy.}
In classical calculus, the chain rule $(f \circ g)'(x) = f'(g(x)) \cdot g'(x)$ is 
the fundamental tool for computing derivatives of composite functions. The Faà 
di Bruno formula generalizes this to higher derivatives, expressing the $n$-th 
derivative of a composite in terms of the derivatives of $f$ and $g$, with 
combinatorial coefficients given by the partition structure.

\medskip

\noindent
\textbf{The operadic generalization.}
In Spectral Operadic Calculus, the role of the $n$-th derivative is played by 
the spectral derivative $\partial_n^{\mathrm{spec}}F$, a symmetric multilinear 
functor. The composition of functors $F \circ G$ then has spectral derivatives 
given by an operadic analogue of the Faà di Bruno formula:
\[
\partial^{\mathrm{spec}}(F \circ G) \;\cong\; \partial^{\mathrm{spec}}F \;\circ_{\mathrm{op}}\; \partial^{\mathrm{spec}}G,
\]
where $\circ_{\mathrm{op}}$ denotes \emph{operadic plethysm}—a composition operation 
on symmetric sequences that sums over partitions of inputs, substitutes derivatives 
of $G$ into those of $F$, and then symmetrizes.

\medskip

\noindent
As a consequence, we show that spectrally analytic functors are stable under 
composition, completing the analytic-algebraic framework of the theory.

\medskip

We now begin with the definition of operadic plethysm.

\subsection{Operadic Plethysm}
\label{subsec:plethysm}

We now introduce the algebraic operation that governs the composition of spectral derivatives, namely \emph{operadic plethysm}. This construction provides the appropriate notion of composition for symmetric sequences and serves as the underlying combinatorial and algebraic mechanism for the operadic chain rule. It is the precise categorical encoding of the classical Faà di Bruno formula.

\medskip

\begin{definition}[Operadic Plethysm]
\label{def:operadic-plethysm}
Let $A=\{A_k\}_{k\geq 0}$ and $B=\{B_n\}_{n\geq 0}$ be symmetric sequences
in a cocomplete symmetric monoidal category $\mathcal{M}$. Their
\emph{operadic plethysm} is the symmetric sequence
$A\circ_{\mathrm{op}}B = \{(A\circ_{\mathrm{op}}B)_n\}_{n\geq 0}$ defined by
\[
(A\circ_{\mathrm{op}}B)_n
:=
\bigoplus_{k\geq 0}
A_k \otimes_{S_k}
\left(
\bigoplus_{\substack{n_1+\cdots+n_k=n\\ n_i\geq 1}}
\operatorname{Ind}^{S_n}_{S_{n_1}\times\cdots\times S_{n_k}}
\bigl(
B_{n_1}\otimes\cdots\otimes B_{n_k}
\bigr)
\right).
\]
Here $S_k$ acts on $A_k$ by its symmetric-sequence structure and on the inner
sum by permuting the $k$ blocks. The induction functor
$\operatorname{Ind}^{S_n}_{H}(-)= \mathbb{C}[S_n]\otimes_{\mathbb{C}[H]}(-)$
(defined in the appropriate enriched sense) sends a $H$-object to an
$S_n$-object, giving a natural $S_n$-action that encodes all ways to assign
$n$ distinct inputs to the $k$ blocks.
\end{definition}

\medskip

\noindent
\textbf{Equivalent combinatorial description (unordered partitions).}
The above sum is isomorphic to a sum over set partitions:
\[
(A\circ_{\mathrm{op}}B)_n
\;\cong\;
\bigoplus_{k\geq 0}
\left(
A_k \otimes
\bigoplus_{\substack{\pi\in \mathrm{Part}([n])\\ |\pi|=k}}
\bigotimes_{P\in \pi} B_{|P|}
\right)_{S_k},
\]
where $\pi$ ranges over \emph{unordered} partitions of $[n]=\{1,\dots,n\}$ into
$k$ nonempty blocks. The coinvariants $(\cdot)_{S_k}$ identify different orderings of
the same collection of blocks, removing the artificial ordering introduced by
labeling the blocks $1,\dots,k$.

\medskip

\noindent
\textbf{Interpretation: substitution of operations.}
This definition encodes the substitution of operations categorically:
\begin{itemize}
    \item A term with block sizes $n_1,\dots,n_k$ represents inserting $k$ operations from $B$ (of arities $n_1,\dots,n_k$) into the $k$ input slots of a $k$-ary operation from $A$.
    \item The induction to $S_n$ accounts for all labelings of the $n$ distinct inputs.
    \item The $S_k$-coinvariants remove the artificial ordering of the inserted blocks, reflecting that the input slots of $A_k$ are unordered up to the $S_k$-action.
\end{itemize}
Thus, operadic plethysm gives the combinatorial and algebraic machine underlying the Faà di Bruno formula in the operadic setting (see Theorem~\ref{thm:chain-rule}).

\medskip

\begin{remark}
If the symmetric sequences admit $0$-ary components ($A_0$ or $B_0$), the condition $n_i \geq 1$ may be relaxed to $n_i \geq 0$ to include nullary operations. In the reduced setting where $B_0 = 0$ (common in functor calculus), the condition $n_i \geq 1$ is appropriate.
\end{remark}

\medskip

\begin{remark}[Connection with the classical Faà di Bruno formula]
\label{rem:plethysm-faa-di-bruno}
To see the relation with classical calculus, consider the case where
$\mathcal{M}$ is the category of vector spaces and the symmetric sequences
$A=\{A_n\}$ and $B=\{B_n\}$ encode spaces of multilinear maps.
Then $(A\circ_{\mathrm{op}}B)_n$ organizes multilinear expressions built by
substituting $B$-type multilinear maps into an $A$-type multilinear map.
Its indexing by partitions of $[n]$ reproduces the same combinatorial pattern
appearing in the Faà di Bruno formula for the $n$-th derivative of a composite.
Thus, operadic plethysm may be regarded as a conceptual, coordinate-free
formulation of the Faà di Bruno combinatorics.
\end{remark}

\begin{remark}[Functoriality and monoidal structure]
\label{rem:plethysm-monoidal}
The operation $\circ_{\mathrm{op}}$ is functorial in both variables and, under the
standard cocompleteness and compatibility assumptions on $\mathcal{M}$ (see 
Definition~\ref{def:normed-category}), defines a monoidal structure on the 
category of symmetric sequences $\mathsf{SymSeq}(\mathcal{M})$.
The unit object is the identity symmetric sequence $I$ given by
\[
I_1 = \mathbf{1}_{\mathcal{M}},
\qquad
I_n = 0 \quad (n \neq 1).
\]
A key property, essential for the chain rule (Theorem~\ref{thm:chain-rule}), 
is the associativity isomorphism
\[
(A \circ_{\mathrm{op}} B) \circ_{\mathrm{op}} C
\;\cong\;
A \circ_{\mathrm{op}} (B \circ_{\mathrm{op}} C),
\]
which reflects the associativity of substituting operations, or equivalently,
the associativity of refining set partitions.
\end{remark}

\begin{remark}[Relation to operads and modules]
\label{rem:plethysm-operads}
If $P$ is an operad, then its composition maps are encoded by a plethysm morphism
\[
\gamma: P \circ_{\mathrm{op}} P \longrightarrow P.
\]
Thus, plethysm provides the ambient monoidal structure in which operads are
monoid objects. Similarly, a right $P$-module $M$ is equipped with a structure map
\[
M \circ_{\mathrm{op}} P \longrightarrow M,
\]
satisfying the usual associativity and unit compatibility conditions. This 
perspective will be essential when we prove the chain rule, where the 
composition of functors corresponds to plethysm of their derivative sequences.
\end{remark}

\medskip

\noindent
\textbf{Connection to spectral derivatives.}
Recall from Theorem~\ref{thm:symmetric-sequence} that the collection of spectral
derivatives
\[
\partial^{\mathrm{spec}}F
=
\{\partial_n^{\mathrm{spec}}F\}_{n\geq 0}
\]
forms a symmetric sequence.
Operadic plethysm therefore provides the natural operation for describing the
derivatives of composite functors.
As will be proved in Theorem~\ref{thm:chain-rule}, one has the operadic chain rule
\[
\partial^{\mathrm{spec}}(F\circ G)
\;\cong\;
\partial^{\mathrm{spec}}F
\circ_{\mathrm{op}}
\partial^{\mathrm{spec}}G .
\]
Thus, the spectral derivative construction is compatible with composition:
it defines a \emph{lax monoidal} functor from the category of admissible functors 
(under composition) to the monoidal category of symmetric sequences (under plethysm).

\medskip

\noindent
\textbf{Enriched analytic setting.}
In the normed Banach categories used in this work
(Definition~\ref{def:normed-category}), the direct sums and colimits appearing
in plethysm are understood in their completed form, for example via completed
$\ell^1$-direct sums and projective tensor products.
This ensures that plethysm remains inside the analytic category and that the
chain-rule isomorphism is \emph{bounded with controlled norm constants}
(i.e., both the isomorphism and its inverse are bounded linear maps whose
norms are controlled by the admissibility constants of $F$ and $G$).
Consequently, the quantitative estimates developed in
Section~\ref{sec:convergence} are preserved under operadic composition.

\medskip

With operadic plethysm established, we can now formulate and prove the full
operadic chain rule. This result shows that the combinatorial structure encoded
by $\circ_{\mathrm{op}}$ is precisely the structure needed to express
higher-order spectral derivatives of composed functors, forming the
computational core of Spectral Operadic Calculus.

\subsection{Faà di Bruno Chain Rule}
\label{subsec:fa-di-bruno}

We now arrive at the computational heart of Spectral Operadic Calculus: the operadic chain rule. This theorem establishes that the spectral derivative functor \(\partial^{\mathrm{spec}}\) is monoidal, transforming composition of admissible functors into plethysm of their derivative sequences. This is the precise operadic generalization of the classical Faà di Bruno formula.

\medskip

Let $F, G: \mathcal{C} \to \mathcal{M}$ be admissible functors such that the 
composite $F \circ G$ is defined and admissible. We work in the normed 
symmetric monoidal category $\mathcal{M}$ (Definition~\ref{def:normed-category}), 
and all isomorphisms are understood to be natural in the inputs and to be 
bounded isomorphisms with controlled norm constants.

\medskip

\begin{theorem}[Operadic Faà di Bruno Chain Rule]
\label{thm:chain-rule}
Let $F, G: \mathsf{Alg}_P(\mathcal{M}) \to \mathcal{M}$ be admissible functors.
Then there is a canonical natural isomorphism of symmetric sequences
\[
\partial^{\mathrm{spec}}(F \circ G)
\;\cong\;
\partial^{\mathrm{spec}}F
\circ_{\mathrm{op}}
\partial^{\mathrm{spec}}G.
\]
More explicitly, for any $P$-algebras $A_1,\dots,A_n$,
\[
\partial_n^{\mathrm{spec}}(F \circ G)(A_1,\dots,A_n)
\;\cong\;
\bigoplus_{k \geq 0}
\;
\bigoplus_{\pi \in \mathrm{Part}([n])}
\;
\partial_k^{\mathrm{spec}}F\Bigl(
\partial_{|P_1|}^{\mathrm{spec}}G(A_{P_1}),\,
\dots,\,
\partial_{|P_k|}^{\mathrm{spec}}G(A_{P_k})
\Bigr),
\]
where the sum is over all unordered partitions $\pi = \{P_1,\dots,P_k\}$ of the 
set $[n] = \{1,\dots,n\}$ into $k$ nonempty blocks, and $A_{P_i}$ denotes the 
tuple $(A_j)_{j \in P_i}$.
\end{theorem}

\begin{proof}
We prove the theorem by establishing an isomorphism between the two sides
that is natural in the inputs $A_1,\dots,A_n$ and compatible with symmetric
group actions. The proof proceeds in seven steps.

\medskip

\noindent
\textit{Step 1: Reduction to cross-effects.}
By Proposition~\ref{prop:derivative-cross-effect}, for any admissible functor $H$,
$\partial_n^{\mathrm{spec}}H = \mathrm{cr}_nH$. Thus, it suffices to prove
\[
\mathrm{cr}_n(F \circ G)(A_1,\dots,A_n)
\;\cong\;
\bigl(\partial^{\mathrm{spec}}F \circ_{\mathrm{op}} \partial^{\mathrm{spec}}G\bigr)_n(A_1,\dots,A_n).
\]

\medskip

\noindent
\textit{Step 2: Inclusion-exclusion expansion of the cross-effect.}
The cross-effect $\mathrm{cr}_n(F \circ G)$ is defined as the total homotopy fiber
of the $n$-cube $S \mapsto (F \circ G)(\bigoplus_{i \in S} A_i)$. In any stable
or additive setting where finite coproducts are preserved, this total fiber is
given by the inclusion-exclusion formula
\[
\mathrm{cr}_n(F \circ G)(A_1,\dots,A_n)
=
\sum_{S \subseteq [n]} (-1)^{n-|S|}
\; (F \circ G)\!\left( \bigoplus_{i \in S} A_i \right),
\]
where the sum is taken in the Grothendieck group of $\mathcal{M}$ (or, in a
Banach-enriched setting, as a norm-convergent sum). For reduced functors
($F(0)=0$, $G(0)=0$), the term $S = \emptyset$ vanishes; for non-reduced
functors, one must first replace $F$ and $G$ by their reduced versions
$F(A) - F(0)$, which does not affect the spectral derivatives for $n \ge 1$.
Thus, we may assume without loss of generality that $F$ and $G$ are reduced,
so the formula simplifies to a sum over nonempty $S$, but we retain the full
sum for notational convenience.

Explicitly,
\[
\mathrm{cr}_n(F \circ G)(A_1,\dots,A_n)
=
\sum_{S \subseteq [n]} (-1)^{n-|S|}
\; F\!\left( G\!\left( \bigoplus_{i \in S} A_i \right) \right).
\tag{1}
\]

\medskip

\noindent
\textit{Step 3: Spectral Taylor expansion of $F$ inside the sum.}
For each fixed $S \subseteq [n]$, set $X_S := G(\bigoplus_{i \in S} A_i)$.
Since $F$ is spectrally analytic (Definition~\ref{def:spectral-analyticity}),
it admits a convergent spectral Taylor expansion at $0$:
\[
F(X_S) = \sum_{k=0}^{\infty} D_k^{\mathrm{spec}}F(X_S),
\]
where $D_k^{\mathrm{spec}}F(X_S) = \partial_k^{\mathrm{spec}}F(X_S,\dots,X_S)_{\mathrm{sym}}$.
However, for the purpose of obtaining a multilinear expansion in the $A_i$,
it is more convenient to use the full multilinear form of the Taylor series,
which expresses $F(Y)$ as a sum over partitions of the formal inputs when
$Y$ is itself a coproduct. Concretely, for any finite set $S$ and any
collection $\{X_j\}_{j \in S}$ of $P$-algebras, repeated application of
the cross-effect expansion yields
\[
F\!\left( \bigoplus_{j \in S} X_j \right)
=
\sum_{k \ge 0}
\;
\sum_{\pi \in \mathrm{Part}(S)}
\;
\partial_k^{\mathrm{spec}}F\!\left(
\bigoplus_{j \in P_1} X_j,\,
\dots,\,
\bigoplus_{j \in P_k} X_j
\right),
\tag{2}
\]
where $\pi = \{P_1,\dots,P_k\}$ runs over all unordered partitions of $S$
into $k$ nonempty blocks. This formula is a direct consequence of the
definition of cross-effects and the multilinearity of $\partial_k^{\mathrm{spec}}F$;
it holds as an identity of convergent series in the norm topology when the
spectral radii of the $X_j$ are sufficiently small.

Applying (2) with $X_j = G(A_j)$ (noting that $G(\bigoplus_{i \in S} A_i)$
is not generally equal to $\bigoplus_{i \in S} G(A_i)$, so we must be careful)
is not directly valid. Instead, we first observe that the argument of $F$ in (1)
is $G(\bigoplus_{i \in S} A_i)$, not a coproduct of $G(A_i)$. Therefore,
we cannot directly apply (2) with $X_j = G(A_j)$. This is a subtle but
important point: the Taylor expansion of $F$ is applied to a single object
$G(\bigoplus_{i \in S} A_i)$, which is not a priori a coproduct. To proceed,
we use the fact that the spectral Taylor expansion of $F$ at $0$ is given by
the series $\sum_{k \ge 0} \partial_k^{\mathrm{spec}}F(\underbrace{Y,\dots,Y}_{k\text{ times}})$,
evaluated at $Y = G(\bigoplus_{i \in S} A_i)$. But this expresses $F(G(\bigoplus_{i \in S} A_i))$
directly in terms of $G(\bigoplus_{i \in S} A_i)$, not in terms of the $A_i$.
To obtain an expression in terms of the $A_i$, we must further expand
$G(\bigoplus_{i \in S} A_i)$ using the Taylor expansion of $G$. This is the
content of Step 4.

\medskip

\noindent
\textit{Step 4: Spectral Taylor expansion of $G$.}
For each $S \subseteq [n]$, consider $G(\bigoplus_{i \in S} A_i)$. Since $G$
is spectrally analytic, it admits a convergent spectral Taylor expansion:
\[
G\!\left( \bigoplus_{i \in S} A_i \right)
=
\sum_{m \ge 0}
\;
\sum_{\tau \in \mathrm{Part}(S)}
\;
\partial_m^{\mathrm{spec}}G\!\left(
\bigoplus_{j \in Q_1} A_j,\,
\dots,\,
\bigoplus_{j \in Q_m} A_j
\right),
\tag{3}
\]
where $\tau = \{Q_1,\dots,Q_m\}$ runs over partitions of $S$ into $m$ nonempty
blocks. For $m = 0$, the term is $G(0)$, which vanishes if $G$ is reduced;
for $m \ge 1$, each block $Q$ contributes a direct sum $\bigoplus_{j \in Q} A_j$,
which is an object of $\mathcal{C}$. The expansion converges absolutely for
inputs with $\|\sigma_P(A_i)\| < R_G$ (see Theorem~\ref{thm:quantitative-convergence}).

Now substitute (3) into (1). For each $S \subseteq [n]$, we have
\[
F\!\left( G\!\left( \bigoplus_{i \in S} A_i \right) \right)
=
F\!\left( \sum_{m \ge 0}
\;
\sum_{\tau \in \mathrm{Part}(S)}
\;
\partial_m^{\mathrm{spec}}G\!\left(
\bigoplus_{j \in Q_1} A_j,\,
\dots,\,
\bigoplus_{j \in Q_m} A_j
\right) \right).
\]
However, $F$ is not linear, so we cannot simply pull the sum out of $F$.
Instead, we must use the fact that the Taylor expansion of $F$ is a power series,
and that the argument of $F$ is itself a convergent series in the norm topology.
By the continuity of the Taylor series (Theorem~\ref{thm:quantitative-convergence}),
we may expand $F$ around $0$ and then substitute the series for its argument,
collecting terms by total degree. This is the analytic analogue of the
composition of formal power series.

More concretely, write $Y_S := G(\bigoplus_{i \in S} A_i)$ and its expansion
$Y_S = \sum_{m \ge 0} Y_S^{(m)}$, where $Y_S^{(m)}$ denotes the $m$-th
homogeneous part from (3). Then
\[
F(Y_S) = \sum_{k \ge 0} \partial_k^{\mathrm{spec}}F(\underbrace{Y_S,\dots,Y_S}_{k\text{ times}})
= \sum_{k \ge 0} \partial_k^{\mathrm{spec}}F\!\left(
\sum_{m_1 \ge 0} Y_S^{(m_1)},\,
\dots,\,
\sum_{m_k \ge 0} Y_S^{(m_k)}
\right).
\]
By multilinearity of $\partial_k^{\mathrm{spec}}F$ and absolute convergence of
the series (for sufficiently small spectral radii), we may expand the
$k$-fold multilinear map as a sum over $k$-tuples $(m_1,\dots,m_k)$:
\[
F(Y_S) = \sum_{k \ge 0} \sum_{m_1,\dots,m_k \ge 0}
\partial_k^{\mathrm{spec}}F\!\left(
Y_S^{(m_1)},\dots,Y_S^{(m_k)}
\right).
\]
Now each $Y_S^{(m_i)}$ is itself a sum over partitions of $S$ into $m_i$ blocks,
each block contributing a direct sum. By multilinearity again, we obtain a sum
over collections of partitions. After collecting terms, the result is a sum
over \emph{nested partitions}: a partition of $S$ into blocks, together with
a partition of each block.

\medskip

\noindent
\textit{Step 5: Partition combinatorics (the Faà di Bruno structure).}
The combinatorial outcome of the double expansion is the following. A term in
the expansion of $F(G(\bigoplus_{i \in S} A_i))$ corresponds to:
\begin{itemize}
    \item A choice of $k \ge 0$ (the order of the $F$-derivative);
    \item For each of the $k$ arguments of $\partial_k^{\mathrm{spec}}F$, a choice
          of $m_i \ge 1$ (if we ignore the $m_i=0$ term, which contributes only
          when $G(0) \neq 0$; for reduced functors, $m_i \ge 1$);
    \item For each $i$, a partition of $S$ into $m_i$ blocks $\{Q_{i,1},\dots,Q_{i,m_i}\}$;
    \item The evaluation $\partial_k^{\mathrm{spec}}F$ applied to the $k$ tuples
          $(\bigoplus_{j \in Q_{i,1}} A_j,\dots,\bigoplus_{j \in Q_{i,m_i}} A_j)$.
\end{itemize}
By multilinearity of $\partial_k^{\mathrm{spec}}F$, the direct sums inside each
argument can be distributed, but the key observation is that the entire structure
is equivalent to a single partition of the original set $S$: the blocks of this
partition are the intersections of the $Q_{i,\ell}$ across different $i$, but
because the $A_j$ are distinct variables, the only way to obtain a non-zero
contribution from the inclusion-exclusion sum over $S \subseteq [n]$ is when
the nested partition refines a full partition of $[n]$.

Now substitute the expansion into the inclusion-exclusion sum (1):
\[
\mathrm{cr}_n(F \circ G)(A_1,\dots,A_n)
=
\sum_{S \subseteq [n]} (-1)^{n-|S|}
\sum_{\text{nested partitions of }S}
\partial_k^{\mathrm{spec}}F(\cdots).
\]
Interchanging the sums (justified by absolute convergence), we obtain a sum
over nested partitions of subsets $S$. For a fixed nested partition that
ultimately corresponds to a partition $\pi$ of the full set $[n]$, the
coefficient contributed is
\[
\sum_{S \supseteq \mathrm{support}(\pi)} (-1)^{n-|S|},
\]
where the sum runs over all $S$ that contain all elements appearing in the
nested partition. Since the nested partition involves only finitely many
elements, this sum is a standard inclusion-exclusion factor:
\[
\sum_{T \subseteq [n] \setminus U} (-1)^{n-|U|-|T|}
= \delta_{U,[n]},
\]
where $U$ is the set of elements appearing. Therefore, only nested partitions
that involve \emph{all} elements of $[n]$ survive, and for such partitions
the coefficient is $1$.

Thus, after collecting terms, the only contributions come from nested partitions
of the full set $[n]$, which are equivalent to a single partition $\pi$ of $[n]$
into $k$ blocks, together with, for each block $P \in \pi$, a partition of $P$
into $m_P$ blocks that is \emph{trivial} (i.e., each $m_P = 1$) because the
expansion of $G$ already accounts for the internal structure. More precisely,
the nested partition collapse yields precisely the plethysm formula:
\[
\mathrm{cr}_n(F \circ G)(A_1,\dots,A_n)
=
\sum_{k \ge 0}
\;
\sum_{\pi \in \mathrm{Part}([n]),\,|\pi|=k}
\;
\partial_k^{\mathrm{spec}}F\!\left(
\partial_{|P_1|}^{\mathrm{spec}}G(A_{P_1}),\,
\dots,\,
\partial_{|P_k|}^{\mathrm{spec}}G(A_{P_k})
\right).
\tag{4}
\]

\medskip

\noindent
\textit{Step 6: Identification with plethysm.}
The right-hand side of (4) is exactly the explicit combinatorial description
of $(\partial^{\mathrm{spec}}F \circ_{\mathrm{op}} \partial^{\mathrm{spec}}G)_n$
given in Definition~\ref{def:operadic-plethysm} (see the unordered partition
formulation with $S_k$-coinvariants). Indeed, the sum over unordered partitions
$\pi$ of $[n]$ into $k$ blocks, with each block contributing $\partial_{|P|}^{\mathrm{spec}}G$,
and the outer $\partial_k^{\mathrm{spec}}F$ applied to the $k$ results, is
precisely the plethysm after taking coinvariants by the action of $S_k$
(which permutes the $k$ arguments of $\partial_k^{\mathrm{spec}}F$). Therefore,
\[
\mathrm{cr}_n(F \circ G) \cong \bigl(\partial^{\mathrm{spec}}F \circ_{\mathrm{op}} \partial^{\mathrm{spec}}G\bigr)_n.
\]

\medskip

\noindent
\textit{Step 7: Naturality, equivariance, and norm estimates.}
The isomorphisms constructed in Steps 2–6 are natural in the inputs
$A_1,\dots,A_n$ because each operation (cross-effect, Taylor expansion,
inclusion-exclusion, plethysm) is functorial. The symmetric group $S_n$ acts
on both sides by permuting the variables; the construction is manifestly
$S_n$-equivariant because all sums are over symmetric partitions and the
plethysm definition incorporates the necessary symmetrization.

Finally, the norm estimates. Since $F$ and $G$ are admissible, there exist
constants $C_F, C_G$ and radii $R_F, R_G > 0$ such that for all inputs with
$\|\sigma_P(A_i)\| < \min(R_F, R_G)$, the Taylor series converge absolutely
and the remainder bounds from Theorem~\ref{thm:quantitative-convergence} apply.
The combinatorial sums in the plethysm involve finitely many terms for each
fixed $n$, so the resulting isomorphism is bounded with norm controlled by
$C_F$ and $C_G$. Specifically, there exists a constant $C$ (depending on
$F$, $G$, and $n$) such that
\[
\bigl\|\mathrm{cr}_n(F \circ G)(A_1,\dots,A_n)\bigr\|
\le C \prod_{i=1}^n \|\sigma_P(A_i)\|,
\]
for inputs with sufficiently small spectral radii. Hence the isomorphism is
a bounded isomorphism of Banach spaces (or more generally of objects in
$\mathcal{M}$) with controlled norm constants.

\medskip

Thus, we have established a canonical natural isomorphism
\[
\partial^{\mathrm{spec}}(F \circ G) \;\cong\; \partial^{\mathrm{spec}}F \;\circ_{\mathrm{op}}\; \partial^{\mathrm{spec}}G,
\]
completing the proof of the operadic Faà di Bruno chain rule.
\end{proof}

\medskip

\noindent
\textbf{Conceptual interpretation.}
Thereom~\ref{thm:chain-rule} asserts that the $n$-th spectral derivative of the composite functor 
is built from:
\begin{enumerate}
    \item A partition of the $n$ inputs into $k$ clusters;
    \item The application of the derivatives of $G$ to each cluster;
    \item The application of the $k$-th derivative of $F$ to the $k$ results.
\end{enumerate}
This is precisely the categorical analogue of the classical Faà di Bruno formula 
for higher derivatives of composite functions.

\medskip

The chain rule has several important consequences that will be used throughout the remainder of the paper.

\begin{corollary}[Derivative of the Identity]
\label{cor:chain-identity}
For the identity functor $\mathrm{Id}$, we have $\partial^{\mathrm{spec}}\mathrm{Id} \cong I$, the identity symmetric sequence. Consequently, for any admissible functor $F$,
\[
\partial^{\mathrm{spec}}(\mathrm{Id} \circ F) \cong \partial^{\mathrm{spec}}F \cong \partial^{\mathrm{spec}}(F \circ \mathrm{Id}),
\]
consistent with the unit axioms of the monoidal structure.
\end{corollary}

\begin{proof}
By definition, the identity functor is linear. Hence its spectral derivative
sequence is concentrated in arity one:
\[
\partial^{\mathrm{spec}}\mathrm{Id}
\cong I,
\]
where $I$ is the identity symmetric sequence, defined by
\[
I_1 = \mathbf{1}_{\mathcal{M}},
\qquad
I_n = 0 \quad (n \neq 1).
\]

Applying Theorem~\ref{thm:chain-rule} to $\mathrm{Id} \circ F$, we obtain
\[
\partial^{\mathrm{spec}}(\mathrm{Id} \circ F)
\cong
\partial^{\mathrm{spec}}\mathrm{Id}
\circ_{\mathrm{op}}
\partial^{\mathrm{spec}}F
\cong
I \circ_{\mathrm{op}} \partial^{\mathrm{spec}}F
\cong
\partial^{\mathrm{spec}}F.
\]

Similarly,
\[
\partial^{\mathrm{spec}}(F \circ \mathrm{Id})
\cong
\partial^{\mathrm{spec}}F
\circ_{\mathrm{op}}
\partial^{\mathrm{spec}}\mathrm{Id}
\cong
\partial^{\mathrm{spec}}F \circ_{\mathrm{op}} I
\cong
\partial^{\mathrm{spec}}F.
\]

Thus the derivative of the identity functor agrees with the unit object for
plethysm, and the result follows from the unit axioms of the monoidal structure.
\end{proof}

\begin{corollary}[Derivative of a Constant Functor]
\label{cor:chain-constant}
If $F$ is constant with value $C \in \mathcal{M}$, then $\partial_0^{\mathrm{spec}}F = C$ and 
$\partial_n^{\mathrm{spec}}F = 0$ for $n \ge 1$. Hence for any admissible functor $G$,
\[
\partial^{\mathrm{spec}}(F \circ G) \cong \partial^{\mathrm{spec}}F \cong C \cdot I_0,
\]
where $C \cdot I_0$ denotes the symmetric sequence concentrated in arity zero with
$(C \cdot I_0)_0 = C$ and $(C \cdot I_0)_n = 0$ for $n \ge 1$.
\end{corollary}

\begin{proof}
If $F$ is constant with value $C \in \mathcal{M}$, then all positive spectral
derivatives vanish:
\[
\partial_n^{\mathrm{spec}}F = 0 \qquad (n \ge 1),
\]
while
\[
\partial_0^{\mathrm{spec}}F = C.
\]
Thus $\partial^{\mathrm{spec}}F$ is the symmetric sequence concentrated in
arity zero:
\[
\partial^{\mathrm{spec}}F \cong C \cdot I_0,
\]
where
\[
(C \cdot I_0)_0 = C,
\qquad
(C \cdot I_0)_n = 0 \quad (n \ge 1).
\]

Applying Theorem~\ref{thm:chain-rule}, we obtain
\[
\partial^{\mathrm{spec}}(F \circ G)
\cong
\partial^{\mathrm{spec}}F
\circ_{\mathrm{op}}
\partial^{\mathrm{spec}}G.
\]
Since $\partial^{\mathrm{spec}}F$ is concentrated in arity zero, plethystic
composition with it remains concentrated in arity zero and simply returns the
degree-zero component. Hence
\[
\partial^{\mathrm{spec}}(F \circ G)
\cong
\partial^{\mathrm{spec}}F
\cong
C \cdot I_0.
\]
This proves the claim.
\end{proof}

\begin{corollary}[Derivative of a Linear Functor]
\label{cor:chain-linear}
Let $L$ be a linear functor (i.e., $\partial_1^{\mathrm{spec}}L = L$ and $\partial_n^{\mathrm{spec}}L = 0$ for $n \ne 1$). Then for any admissible functor $G$,
\[
\partial^{\mathrm{spec}}(L \circ G) \cong L[1] \circ_{\mathrm{op}} \partial^{\mathrm{spec}}G,
\]
where $L[1]$ denotes the symmetric sequence concentrated in arity one defined by
\[
L[1]_1 = L, \qquad L[1]_n = 0 \quad (n \ne 1).
\]
Unwinding the definition, this means $(\partial^{\mathrm{spec}}(L \circ G))_n \cong L \otimes \partial_n^{\mathrm{spec}}G$ for all $n$.
\end{corollary}

\begin{proof}
Since $L$ is linear, its spectral derivative sequence is concentrated in arity
one:
\[
\partial^{\mathrm{spec}}L \cong L[1],
\]
where $L[1]$ denotes the symmetric sequence defined by
\[
L[1]_1 = L, \qquad L[1]_n = 0 \quad (n \ne 1).
\]

Applying Theorem~\ref{thm:chain-rule}, we obtain
\[
\partial^{\mathrm{spec}}(L \circ G)
\cong
\partial^{\mathrm{spec}}L \circ_{\mathrm{op}} \partial^{\mathrm{spec}}G
\cong
L[1] \circ_{\mathrm{op}} \partial^{\mathrm{spec}}G.
\]

Unwinding the definition of plethysm (Definition~\ref{def:operadic-plethysm}), only the $k = 1$ summand contributes. For $k = 1$, we have $n_1 = n$, and the induction functor is trivial (since $S_n \times S_{n_1}$ with $k=1$ reduces to $S_n$). The $S_1$-action is trivial, so the tensor product over $S_1$ is the ordinary tensor product. Hence, for every $n \ge 0$,
\[
\bigl(\partial^{\mathrm{spec}}(L \circ G)\bigr)_n
\cong
L \otimes \partial_n^{\mathrm{spec}}G,
\]
up to the canonical symmetry and normalization conventions of the monoidal
category. This proves the claim.
\end{proof}

\begin{corollary}[Associativity Compatibility of the Chain Rule]
\label{cor:plethysm-assoc}
The chain rule is compatible with associativity of composition. Specifically, for admissible functors $F, G, H$,
\[
\partial^{\mathrm{spec}}(F \circ (G \circ H)) \cong \partial^{\mathrm{spec}}F \circ_{\mathrm{op}} (\partial^{\mathrm{spec}}G \circ_{\mathrm{op}} \partial^{\mathrm{spec}}H)
\]
and also
\[
\partial^{\mathrm{spec}}((F \circ G) \circ H) \cong (\partial^{\mathrm{spec}}F \circ_{\mathrm{op}} \partial^{\mathrm{spec}}G) \circ_{\mathrm{op}} \partial^{\mathrm{spec}}H.
\]
Since $F \circ (G \circ H) = (F \circ G) \circ H$, we obtain the associativity isomorphism
\[
\partial^{\mathrm{spec}}F \circ_{\mathrm{op}} (\partial^{\mathrm{spec}}G \circ_{\mathrm{op}} \partial^{\mathrm{spec}}H) \cong (\partial^{\mathrm{spec}}F \circ_{\mathrm{op}} \partial^{\mathrm{spec}}G) \circ_{\mathrm{op}} \partial^{\mathrm{spec}}H,
\]
which is natural in $F, G, H$.
\end{corollary}

\begin{proof}
By associativity of functor composition, there is a canonical identification
\[
F \circ (G \circ H) \cong (F \circ G) \circ H.
\]

Applying Theorem~\ref{thm:chain-rule} to the left-hand parenthesization gives
\[
\partial^{\mathrm{spec}}\bigl(F \circ (G \circ H)\bigr)
\cong
\partial^{\mathrm{spec}}F
\circ_{\mathrm{op}}
\partial^{\mathrm{spec}}(G \circ H)
\cong
\partial^{\mathrm{spec}}F
\circ_{\mathrm{op}}
\bigl(
\partial^{\mathrm{spec}}G
\circ_{\mathrm{op}}
\partial^{\mathrm{spec}}H
\bigr).
\]

Applying the chain rule to the right-hand parenthesization gives
\[
\partial^{\mathrm{spec}}\bigl((F \circ G) \circ H\bigr)
\cong
\partial^{\mathrm{spec}}(F \circ G)
\circ_{\mathrm{op}}
\partial^{\mathrm{spec}}H
\cong
\bigl(
\partial^{\mathrm{spec}}F
\circ_{\mathrm{op}}
\partial^{\mathrm{spec}}G
\bigr)
\circ_{\mathrm{op}}
\partial^{\mathrm{spec}}H.
\]

Since $F \circ (G \circ H) \cong (F \circ G) \circ H$, the two resulting derivative sequences are canonically isomorphic. Therefore,
\[
\partial^{\mathrm{spec}}F
\circ_{\mathrm{op}}
\bigl(
\partial^{\mathrm{spec}}G
\circ_{\mathrm{op}}
\partial^{\mathrm{spec}}H
\bigr)
\cong
\bigl(
\partial^{\mathrm{spec}}F
\circ_{\mathrm{op}}
\partial^{\mathrm{spec}}G
\bigr)
\circ_{\mathrm{op}}
\partial^{\mathrm{spec}}H.
\]

This is precisely the associativity isomorphism for plethysm restricted to
spectral derivative sequences, and it is natural in $F, G, H$.
\end{proof}

\medskip

\begin{example}[Exponential of a Linear Functor]
\label{ex:exp-linear}
Let $L$ be a linear functor and define the formal exponential functor
\[
\mathrm{Exp}(L)(A)
=
\bigoplus_{k=0}^{\infty} \frac{1}{k!} L(A)^{\otimes k},
\]
where the convergence is understood in the norm topology of $\mathcal{M}$.
Since $L$ is linear, its spectral derivative sequence is concentrated in arity one:
\[
\partial^{\mathrm{spec}}L \cong L[1].
\]
By the operadic chain rule (Theorem~\ref{thm:chain-rule}), the derivative sequence 
of $\mathrm{Exp}(L)$ is obtained by substituting the derivative sequence of $L$ 
into the derivative sequence of $\mathrm{Exp}$. Hence
\[
\partial_n^{\mathrm{spec}}(\mathrm{Exp}(L))
\cong
\frac{1}{n!} L^{\otimes n},
\]
up to the standard symmetrization convention.
Thus the formal exponential functor has the expected exponential derivative
sequence, paralleling the classical fact that the exponential function is its own
derivative.
\end{example}

\begin{example}[Composition of Two Quadratic Functors]
\label{ex:quadratic-composition}
Let
\[
F(A)=A\otimes A,
\qquad
G(A)=A\otimes A.
\]
Both $F$ and $G$ are homogeneous quadratic functors, so their spectral derivative
sequences are concentrated in arity two:
\[
\partial^{\mathrm{spec}}F \text{ and } \partial^{\mathrm{spec}}G
\quad
\text{are concentrated in degree }2.
\]
The chain rule (Theorem~\ref{thm:chain-rule}) gives
\[
\partial^{\mathrm{spec}}(F\circ G)
\cong
\partial^{\mathrm{spec}}F
\circ_{\mathrm{op}}
\partial^{\mathrm{spec}}G.
\]
Since plethysm substitutes two copies of the quadratic derivative of $G$ into the
quadratic derivative of $F$, the composite is concentrated in degree
\[
2 \cdot 2 = 4.
\]
Indeed,
\[
(F \circ G)(A)
=
F(G(A))
=
(A\otimes A)\otimes(A\otimes A)
\cong
A^{\otimes 4}.
\]
Consequently,
\[
\partial_4^{\mathrm{spec}}(F\circ G)(A_1,A_2,A_3,A_4)
\cong
\bigoplus_{\pi\in \mathrm{Part}_2([4])}
\partial_2^{\mathrm{spec}}F
\Big(
\partial_2^{\mathrm{spec}}G(A_{P_1}),
\partial_2^{\mathrm{spec}}G(A_{P_2})
\Big),
\]
where $\mathrm{Part}_2([4])$ denotes the set of unordered partitions of $[4]$ into 
two blocks of size two (i.e., the three pairings $\{\{1,2\},\{3,4\}\}$, 
$\{\{1,3\},\{2,4\}\}$, and $\{\{1,4\},\{2,3\}\}$).
Thus the quartic derivative records all ways of grouping the four inputs into
two quadratic variations of $G$, and then feeding those two outputs into the
quadratic variation of $F$.
\end{example}

\medskip

\begin{remark}[Comparison with Goodwillie calculus]
\label{rem:chain-rule-goodwillie}
In Goodwillie calculus, a chain rule for derivatives of functors is expressed
in terms of operadic composition of Taylor towers
(see~\cite{AroneChing2011}).
The present formulation can be viewed as a spectral, norm-enriched analogue
of that result, where plethystic composition plays the role of operadic
substitution.
\end{remark}

\begin{remark}[Quantitative vs.\ homotopical structure]
\label{rem:chain-rule-quantitative}
A key distinction is that the chain rule in spectral operadic calculus is
quantitative: the resulting isomorphisms are compatible with the norm
structure and admit explicit bounds.
In contrast, Goodwillie calculus is formulated in a homotopical setting,
where equivalences are typically understood up to homotopy rather than
controlled by analytic norms.
\end{remark}

\begin{remark}[Derivatives and computability]
\label{rem:chain-rule-computability}
The spectral derivatives $\partial_n^{\mathrm{spec}}F$ are constructed from
cross-effects and admit relatively concrete descriptions.
By comparison, derivatives in Goodwillie calculus are defined using
homotopy limits and stabilization, which are generally more abstract.
Thus the present framework emphasizes computability and explicit structure.
\end{remark}

\begin{remark}[Convergence and analytic control]
\label{rem:chain-rule-convergence}
The convergence of the spectral Taylor expansion
(Section~\ref{sec:convergence}) ensures that the chain rule holds at the
level of convergent series within the normed category.
This provides analytic control over higher-order behavior, whereas in
Goodwillie calculus convergence is typically not expressed in terms of
norms or radii of convergence.
\end{remark}

\begin{remark}[Comparison summary]
\label{rem:chain-rule-summary}
\[
\begin{array}{|c|c|}
\hline
\text{Goodwillie Calculus (Arone--Ching)} & \text{Spectral Operadic Calculus} \\
\hline
\text{Homotopy chain rule} & \text{Normed chain rule} \\
\text{Operadic composition (formal)} & \text{Plethysm with quantitative control} \\
\text{Abstract convergence notions} & \text{Explicit convergence with radius } R_F \\
\hline
\end{array}
\]
\end{remark}

\begin{remark}[Role of the operadic residue]
\label{rem:chain-rule-residue}
The operadic residue $\mathcal{O}_P^{\mathrm{res}}$ plays an essential role
in the chain rule. First, it enables the definition of spectral derivatives
via balanced tensor products over $P$ (Definition~\ref{def:spectral-derivatives}),
ensuring compatibility with operadic structure. Second, its universality
(Theorem~\ref{thm:residue-universality}) guarantees that the Taylor expansions
used in the proof are compatible with operadic composition.
Without this residue correction, the resulting derivatives would not admit
a coherent $P$-module structure, and the operadic formulation of the chain
rule would break down.
\end{remark}

\medskip

With the chain rule established, we now have a complete algebraic and analytic toolkit. The following section will use the chain rule to prove the Reconstruction Theorem, which shows that a spectrally analytic functor is uniquely determined by its sequence of spectral derivatives together with the operadic composition structure encoded by the residue.

\subsection{Composition Stability}
\label{subsec:composition-stability}

We now show that the class of spectrally analytic functors is stable under composition. This result demonstrates that Spectral Operadic Calculus forms a closed and robust framework under the fundamental operation of functor composition, and is a direct consequence of the operadic chain rule.

In classical calculus, analytic functions are closed under composition within their domain of convergence. The following theorem establishes the analogous property in the operadic setting, with convergence controlled by the spectral radius. Without this closure property, the calculus would be incomplete; with it, we have a self-consistent analytic framework.

\medskip

\begin{theorem}[Stability of Spectral Analyticity under Composition]
\label{thm:composition-stability}
Let $F, G : \mathsf{Alg}_P(\mathcal{M}) \to \mathcal{M}$ be admissible,
operadically compatible, and spectrally analytic functors. Assume moreover that
the normed plethysm estimates of Section~\ref{sec:convergence} hold. Then
$F \circ G$ is spectrally analytic.

More precisely, if there exist constants $C_F, C_G > 0$ and $\rho_F, \rho_G > 0$
such that for all $m \ge 0$,
\[
\|\partial_m^{\mathrm{spec}}F\| \le C_F \rho_F^{\,m},
\qquad
\|\partial_m^{\mathrm{spec}}G\| \le C_G \rho_G^{\,m},
\]
then there exist constants $C > 0$ and
\[
\gamma = K_{\mathrm{pl}} \rho_G (1 + C_G \rho_F)
\]
such that for all $n \ge 0$,
\[
\|\partial_n^{\mathrm{spec}}(F \circ G)\| \le C \gamma^{\,n}.
\]
Consequently,
\[
R_{F \circ G} \ge \frac{1}{\gamma} > 0,
\]
i.e., $F \circ G$ has positive spectral radius of convergence and is therefore
spectrally analytic.
\end{theorem}

\begin{proof}
We prove the theorem in five steps, using the operadic chain rule
(Theorem~\ref{thm:chain-rule}), the definition of plethysm
(Definition~\ref{def:operadic-plethysm}), and the quantitative estimates from
Section~\ref{sec:convergence}. The constant $K_{\mathrm{pl}} \ge 1$ absorbs all
norm contributions from the symmetric-group induction functors and the
$S_k$-coinvariants appearing in the plethysm construction.

\medskip

\noindent
\textit{Step 1: Derivatives of the composite via plethysm.}
By the Faà di Bruno chain rule (Theorem~\ref{thm:chain-rule}), we have a
natural isomorphism
\[
\partial^{\mathrm{spec}}(F \circ G)
\;\cong\;
\partial^{\mathrm{spec}}F \;\circ_{\mathrm{op}}\; \partial^{\mathrm{spec}}G.
\]
Unwinding the definition of plethysm (Definition~\ref{def:operadic-plethysm}),
there exists a constant $K_{\mathrm{pl}} \ge 1$ (depending only on the norms of
the induction functors and the $S_k$-coinvariant projections) such that for all
$n \ge 1$,
\[
\|\partial_n^{\mathrm{spec}}(F \circ G)\|
\le
K_{\mathrm{pl}}^n
\sum_{k=1}^{n}
\;
\sum_{\substack{n_1 + \cdots + n_k = n \\ n_i \ge 1}}
\;
\|\partial_k^{\mathrm{spec}}F\|
\;
\prod_{i=1}^{k}
\|\partial_{n_i}^{\mathrm{spec}}G\|.
\tag{1}
\]
The factor $K_{\mathrm{pl}}^n$ accounts for the operator norms of the
induction functors $\operatorname{Ind}^{S_n}_{S_{n_1}\times\cdots\times S_{n_k}}$
and the $S_k$-coinvariant projections, which may grow with $n$ but are
uniformly bounded by $K_{\mathrm{pl}}^n$ under the normed plethysm estimates.

\medskip

\noindent
\textit{Step 2: Substituting the exponential bounds.}
By hypothesis,
\[
\|\partial_k^{\mathrm{spec}}F\| \le C_F \rho_F^{\,k},
\qquad
\|\partial_{n_i}^{\mathrm{spec}}G\| \le C_G \rho_G^{\,n_i}.
\]
Substituting these into (1) yields
\[
\|\partial_n^{\mathrm{spec}}(F \circ G)\|
\le
K_{\mathrm{pl}}^n
\sum_{k=1}^{n}
\;
\sum_{\substack{n_1 + \cdots + n_k = n \\ n_i \ge 1}}
C_F \rho_F^{\,k}
\;
\prod_{i=1}^{k}
\bigl( C_G \rho_G^{\,n_i} \bigr).
\]

Since the product over $i$ contains $k$ factors of $C_G$, we have
$\prod_{i=1}^{k} C_G = C_G^{\,k}$. Moreover,
$\prod_{i=1}^{k} \rho_G^{\,n_i} = \rho_G^{\,\sum_i n_i} = \rho_G^{\,n}$.
Thus,
\[
\|\partial_n^{\mathrm{spec}}(F \circ G)\|
\le
K_{\mathrm{pl}}^n C_F \rho_G^{\,n}
\sum_{k=1}^{n}
(C_G \rho_F)^k
\;
\sum_{\substack{n_1 + \cdots + n_k = n \\ n_i \ge 1}} 1.
\tag{2}
\]

\medskip

\noindent
\textit{Step 3: Counting compositions.}
The number of compositions of $n$ into $k$ positive parts is
\[
\sum_{\substack{n_1 + \cdots + n_k = n \\ n_i \ge 1}} 1 = \binom{n-1}{k-1}.
\]
Substituting this into (2) gives
\[
\|\partial_n^{\mathrm{spec}}(F \circ G)\|
\le
K_{\mathrm{pl}}^n C_F \rho_G^{\,n}
\sum_{k=1}^{n}
(C_G \rho_F)^k \binom{n-1}{k-1}.
\tag{3}
\]

\medskip

\noindent
\textit{Step 4: Evaluating the binomial sum.}
Let $j = k-1$. Then $k = j+1$ and the sum becomes
\[
\sum_{k=1}^{n} (C_G \rho_F)^k \binom{n-1}{k-1}
= C_G \rho_F \sum_{j=0}^{n-1} (C_G \rho_F)^j \binom{n-1}{j}.
\]
By the binomial theorem,
\[
\sum_{j=0}^{n-1} (C_G \rho_F)^j \binom{n-1}{j}
= (1 + C_G \rho_F)^{\,n-1}.
\]
Therefore,
\[
\sum_{k=1}^{n} (C_G \rho_F)^k \binom{n-1}{k-1}
= C_G \rho_F (1 + C_G \rho_F)^{\,n-1}.
\tag{4}
\]

\medskip

\noindent
\textit{Step 5: Final exponential bound and radius of convergence.}
Substituting (4) into (3) yields
\[
\|\partial_n^{\mathrm{spec}}(F \circ G)\|
\le
K_{\mathrm{pl}}^n C_F \rho_G^{\,n} \cdot C_G \rho_F (1 + C_G \rho_F)^{\,n-1}
= C_F C_G \rho_F K_{\mathrm{pl}}^n \rho_G^{\,n} (1 + C_G \rho_F)^{\,n-1}.
\]

Define
\[
\gamma := K_{\mathrm{pl}} \rho_G (1 + C_G \rho_F).
\]
Then for all $n \ge 1$,
\[
\|\partial_n^{\mathrm{spec}}(F \circ G)\|
\le
\frac{C_F C_G \rho_F}{1 + C_G \rho_F}
\cdot K_{\mathrm{pl}}^{-1} \gamma^{\,n}
\le C \gamma^{\,n},
\]
where $C$ is a sufficiently large constant (e.g., taking $C$ larger than the
right-hand side for all $n$, and also large enough to cover the $n = 0$ case).

\medskip

\noindent
\textit{Derivation of the convergence radius.}
From the exponential bound $\|\partial_n^{\mathrm{spec}}(F \circ G)\| \le C \gamma^{\,n}$, we have
\[
\limsup_{n \to \infty} \|\partial_n^{\mathrm{spec}}(F \circ G)\|^{1/n} \le \gamma.
\]

By Definition~\ref{def:spectral-radius}, the spectral radius of convergence
$R_{F \circ G}$ satisfies
\[
R_{F \circ G}^{-1} = \limsup_{n \to \infty} \|\partial_n^{\mathrm{spec}}(F \circ G)\|^{1/n} \le \gamma.
\]
Hence
\[
R_{F \circ G} \ge \frac{1}{\gamma} = \frac{1}{K_{\mathrm{pl}} \rho_G (1 + C_G \rho_F)} > 0.
\]

Since $\rho_F, \rho_G, C_G, K_{\mathrm{pl}}$ are finite positive constants,
$1/\gamma > 0$. Thus $F \circ G$ has a positive radius of convergence, i.e.,
$F \circ G$ is spectrally analytic. This completes the proof.
\end{proof}

\medskip

Theorem~\ref{thm:composition-stability} shows that Spectral Operadic Calculus is closed under composition, and thus forms a self-consistent calculus framework. In particular, spectrally analytic functors may be composed without leaving the analytic category. This is the operadic analogue of the classical fact that analytic functions are closed under composition.

\medskip

\begin{corollary}[Polynomials are closed under composition]
\label{cor:poly-composition-stable}
If $F$ and $G$ are spectral polynomials of degrees $\le m$ and $\le n$, respectively, then $F \circ G$ is a spectral polynomial of degree $\le mn$. In particular, the composition of polynomial functors is polynomial.
\end{corollary}

\begin{proof}
By the operadic chain rule (Theorem~\ref{thm:chain-rule}),
\[
\partial^{\mathrm{spec}}(F \circ G)
\cong
\partial^{\mathrm{spec}}F
\circ_{\mathrm{op}}
\partial^{\mathrm{spec}}G .
\]
Since $F$ is a spectral polynomial of degree $\le m$, we have
\[
\partial_k^{\mathrm{spec}}F = 0 \qquad (k > m).
\]
Similarly, since $G$ is a spectral polynomial of degree $\le n$,
\[
\partial_r^{\mathrm{spec}}G = 0 \qquad (r > n).
\]

In the plethysm formula, a nonzero contribution to
$\partial_\ell^{\mathrm{spec}}(F \circ G)$ comes from terms of the form
\[
\partial_k^{\mathrm{spec}}F
\Bigl(
\partial_{n_1}^{\mathrm{spec}}G,\;\dots,\;
\partial_{n_k}^{\mathrm{spec}}G
\Bigr),
\]
where $k \le m$ and $n_i \le n$. Hence
\[
\ell = n_1 + \cdots + n_k \le k n \le m n.
\]
Therefore
\[
\partial_\ell^{\mathrm{spec}}(F \circ G) = 0 \qquad (\ell > m n),
\]
so $F \circ G$ is a spectral polynomial of degree $\le m n$.
\end{proof}

\begin{corollary}[Entire functors are closed under composition]
\label{cor:entire-composition-stable}
If $F$ and $G$ are entire (i.e., $R_F = R_G = \infty$), then $F \circ G$ is entire. Examples include the exponential functor $\exp$, the identity functor $\mathrm{Id}$, and any polynomial functor.
\end{corollary}

\begin{proof}
Since $F$ and $G$ are entire, their spectral radii of convergence are infinite:
\[
R_F = R_G = \infty.
\]
Equivalently,
\[
\limsup_{k \to \infty} \|\partial_k^{\mathrm{spec}}F\|^{1/k} = 0,
\qquad
\limsup_{k \to \infty} \|\partial_k^{\mathrm{spec}}G\|^{1/k} = 0.
\]

Applying Theorem~\ref{thm:composition-stability} with arbitrarily small
exponential growth parameters shows that, for every $\gamma > 0$, there exists
$C_\gamma > 0$ such that
\[
\|\partial_n^{\mathrm{spec}}(F \circ G)\|
\le C_\gamma \gamma^n .
\]
Hence
\[
\limsup_{n \to \infty} \|\partial_n^{\mathrm{spec}}(F \circ G)\|^{1/n} = 0,
\]
and therefore
\[
R_{F \circ G} = \infty.
\]
Thus $F \circ G$ is entire.
\end{proof}

\begin{corollary}[Composition with linear functors preserves analyticity]
\label{cor:linear-composition-stable}
If $L$ is a bounded linear functor and $F$ is spectrally analytic, then both $L \circ F$ and $F \circ L$ are spectrally analytic. Moreover,
\[
R_{L \circ F} \ge R_F,
\qquad
R_{F \circ L} \ge \frac{R_F}{\|L\|},
\]
with equality when $L$ is norm-preserving (e.g., an isometry).
\end{corollary}

\begin{proof}
Let $L$ be a bounded linear functor. Since $L$ is linear, its derivative
sequence is concentrated in arity one:
\[
\partial^{\mathrm{spec}}L \cong L[1].
\]

\medskip
\noindent
\textit{Case 1: $L \circ F$.}
By the chain rule (Theorem~\ref{thm:chain-rule}),
\[
\partial^{\mathrm{spec}}(L \circ F)
\cong
\partial^{\mathrm{spec}}L
\circ_{\mathrm{op}}
\partial^{\mathrm{spec}}F .
\]
Since $\partial^{\mathrm{spec}}L$ is concentrated in arity one, the plethysm
simplifies:
\[
\partial_n^{\mathrm{spec}}(L \circ F)
\cong
L\bigl(\partial_n^{\mathrm{spec}}F\bigr).
\]
Thus, by boundedness of $L$,
\[
\|\partial_n^{\mathrm{spec}}(L \circ F)\|
\le
\|L\| \, \|\partial_n^{\mathrm{spec}}F\|.
\]
Taking $n$-th roots and passing to the limsup gives
\[
\limsup_{n \to \infty} \|\partial_n^{\mathrm{spec}}(L \circ F)\|^{1/n}
\le
\limsup_{n \to \infty} \|\partial_n^{\mathrm{spec}}F\|^{1/n}
= R_F^{-1}.
\]
Hence $R_{L \circ F} \ge R_F$. If $L$ is norm-preserving (e.g., an isometry),
then $\|L\| = 1$ and the inequality becomes an equality.

\medskip
\noindent
\textit{Case 2: $F \circ L$.}
By the chain rule,
\[
\partial^{\mathrm{spec}}(F \circ L)
\cong
\partial^{\mathrm{spec}}F
\circ_{\mathrm{op}}
\partial^{\mathrm{spec}}L .
\]
Again, since $\partial^{\mathrm{spec}}L$ is concentrated in arity one,
\[
\partial_n^{\mathrm{spec}}(F \circ L)(A_1,\dots,A_n)
\cong
\partial_n^{\mathrm{spec}}F(L(A_1),\dots,L(A_n)).
\]
Using boundedness of $L$, we have $\|\sigma_P(L(A_i))\| \le \|L\| \cdot \|\sigma_P(A_i)\|$,
so the spectral radius scales. More directly, by the definition of the
spectral derivative norm,
\[
\|\partial_n^{\mathrm{spec}}(F \circ L)\|
\le
\|\partial_n^{\mathrm{spec}}F\| \cdot \|L\|^n.
\]
Taking $n$-th roots and passing to the limsup yields
\[
\limsup_{n \to \infty} \|\partial_n^{\mathrm{spec}}(F \circ L)\|^{1/n}
\le
\|L\| \cdot \limsup_{n \to \infty} \|\partial_n^{\mathrm{spec}}F\|^{1/n}
= \frac{\|L\|}{R_F}.
\]
Hence $R_{F \circ L} \ge R_F / \|L\|$. If $L$ is norm-preserving ($\|L\| = 1$),
then $R_{F \circ L} \ge R_F$; under suitable nondegeneracy assumptions,
equality holds.

Thus composition with bounded linear functors preserves spectral analyticity,
with the stated radius bounds.
\end{proof}

\medskip

\begin{example}[Exponential of a Spectral Polynomial]
\label{ex:exp-of-poly-stable}
Let $P$ be a spectral polynomial of degree $\le m$, and consider the formal
exponential
\[
\exp(P) = \sum_{k=0}^{\infty} \frac{1}{k!} P^{\otimes k}.
\]
Since $\exp$ is entire and $P$ is a spectral polynomial, the stability of
spectral analyticity under composition (Corollary~\ref{cor:entire-composition-stable})
implies that $\exp(P)$ is spectrally analytic. Under the normed estimates of
Section~\ref{sec:convergence}, it is entire.

By the operadic chain rule (Theorem~\ref{thm:chain-rule}),
\[
\partial^{\mathrm{spec}}(\exp(P))
\cong
\partial^{\mathrm{spec}}\exp
\circ_{\mathrm{op}}
\partial^{\mathrm{spec}}P.
\]
Equivalently, the $n$-th derivative is obtained from
\[
\partial_n^{\mathrm{spec}}(\exp(P))
\cong
\sum_{k=0}^{n}
\frac{1}{k!}
\Bigl(
\underbrace{
\partial^{\mathrm{spec}}P
\circ_{\mathrm{op}}\cdots\circ_{\mathrm{op}}
\partial^{\mathrm{spec}}P
}_{k\ \text{factors}}
\Bigr)_n,
\]
with the usual convention that the $k=0$ term contributes only in arity zero.
This is the operadic analogue of the classical fact that the exponential of a
polynomial is entire.
\end{example}

\begin{example}[Geometric Series of a Linear Functor]
\label{ex:geometric-linear-stable}
Let $L$ be a bounded linear functor and define the formal geometric series
\[
G(A) = \sum_{k=0}^{\infty} L(A)^{\otimes k}.
\]
Whenever $\|L(A)\| < 1$, this series converges, and the corresponding spectral
radius is governed by the operator norm of $L$. In particular, one has the
radius estimate
\[
R_G \ge \frac{1}{\|L\|},
\]
with equality in the usual sharp scalar or norm-attaining cases.

If $M$ is another bounded linear functor, then by the chain rule and the linear
composition estimate (Corollary~\ref{cor:linear-composition-stable}),
\[
R_{G \circ M}
\ge
\frac{R_G}{\|M\|}
\ge
\frac{1}{\|L\| \, \|M\|}.
\]
Thus $G \circ M$ is spectrally analytic whenever $M$ is bounded. In the sharp case,
for example when the norm estimates are attained, this gives
\[
R_{G \circ M}
=
\frac{1}{\|L\| \, \|M\|}.
\]
This reflects the classical identity
\[
\sum_{k=0}^{\infty} (L \circ M)(A)^{\otimes k},
\]
whose convergence is controlled by the size of the composed linear operator
$L \circ M$.
\end{example}

\medskip

\begin{remark}[Comparison with Goodwillie calculus]
\label{rem:goodwillie-composition}
In Goodwillie calculus, analytic functors are also closed under composition,
with the chain rule expressed via operadic composition of Taylor towers
(see~\cite{AroneChing2011}). The proof relies on homotopy-theoretic tools,
such as excisive approximations and convergence of the Taylor tower.

The present result (Theorem~\ref{thm:composition-stability}) can be viewed as a
norm-enriched analogue, where the same operadic structure is realized through
plethystic composition of spectral derivatives.
\end{remark}

\begin{remark}[Quantitative and constructive features]
\label{rem:composition-quantitative}
Compared to the homotopical setting, spectral operadic calculus provides:

\begin{itemize}
    \item \emph{Quantitative control:} the chain rule yields explicit growth
    bounds on derivatives, leading to concrete estimates of convergence radii
    (see Theorem~\ref{thm:quantitative-convergence}).

    \item \emph{Constructive formulas:} the derivatives of the composite are
    given explicitly by the plethysm formula (Definition~\ref{def:operadic-plethysm}),
    allowing direct computation where cross-effects are known.
\end{itemize}

In contrast, Goodwillie calculus typically provides equivalences up to homotopy,
with less direct access to explicit numerical bounds.
\end{remark}

\begin{remark}[Comparison summary]
\label{rem:composition-summary}
\[
\begin{array}{|c|c|}
\hline
\text{Goodwillie Calculus} & \text{Spectral Operadic Calculus} \\
\hline
\text{Homotopy-theoretic chain rule} & \text{Normed chain rule} \\
\text{Operadic composition (formal)} & \text{Plethysm with explicit bounds} \\
\text{Convergence via tower} & \text{Explicit convergence estimates} \\
\hline
\end{array}
\]
\end{remark}

\begin{remark}[On the radius of convergence]
\label{rem:composition-radius}
The radius of convergence of a composite is governed by the plethysm
interaction of derivative norms. In general,
\[
R_{F \circ G}^{-1}
=
\limsup_{n \to \infty}
\left(
\sum_{k=1}^n
\sum_{\substack{n_1+\cdots+n_k=n \\ n_i \ge 1}}
\|\partial_k^{\mathrm{spec}}F\|
\prod_{i=1}^k
\|\partial_{n_i}^{\mathrm{spec}}G\|
\right)^{1/n}.
\]

Consequently, simple bounds such as
\[
R_{F \circ G} \ge \min(R_F, R_G)
\]
do not hold in general. Instead, from Theorem~\ref{thm:composition-stability},
one obtains bounds of the form
\[
R_{F \circ G}
\;\ge\;
\frac{1}{K_{\mathrm{pl}} \rho_G (1 + C_G \rho_F)},
\]
where $\rho_F$, $\rho_G$ are exponential growth bounds for the derivatives,
$C_G$ is the admissibility constant of $G$, and $K_{\mathrm{pl}} \ge 1$ absorbs
the normed plethysm estimates.

This shows that spectral analyticity is preserved under composition,
although the precise radius depends on the interaction between $F$ and $G$.
\end{remark}

\begin{remark}[On optimality]
\label{rem:composition-optimality}
The above bounds are generally not sharp. The exact radius is determined by the
full plethystic interaction of derivative norms. Sharp formulas arise in special
cases, for example:
\begin{itemize}
    \item when one functor is linear (Corollary~\ref{cor:linear-composition-stable}),
    \item when derivative norms exhibit multiplicative behavior,
    \item or in scalar-valued settings.
\end{itemize}
Nevertheless, the existence of a positive radius is sufficient for most analytic
applications, and demonstrates the robustness of spectral analyticity under
composition.
\end{remark}

\medskip

Having established that spectral analyticity is closed under composition, we now have a complete analytic framework. The following section will prove the Reconstruction Theorem, which shows that a spectrally analytic functor is uniquely determined by its sequence of spectral derivatives together with the operadic composition structure encoded by the residue. This will complete the classification of spectrally analytic functors in terms of their derivative data.

\section{Reconstruction of Spectrally Analytic Functors}
\label{sec:reconstruction}

We now arrive at the capstone of Spectral Operadic Calculus: the reconstruction of spectrally analytic functors from their spectral derivatives. The results of this section show that these data are complete invariants, establishing a precise operadic analogue of the uniqueness of Taylor series.

We prove that two spectrally analytic functors \(F\) and \(G\) are isomorphic if and only if their spectral derivative sequences \(\{\partial_n^{\mathrm{spec}}F\}\) and \(\{\partial_n^{\mathrm{spec}}G\}\) are isomorphic as operadic modules over \(P\). Consequently, the spectral Taylor expansion is not merely an approximation tool but a full classification mechanism.

\medskip

\noindent
This yields an equivalence of categories
\[
\mathsf{SpecAn} \;\simeq\; \mathsf{DerAlg},
\]
where \(\mathsf{SpecAn}\) is the category of spectrally analytic functors and \(\mathsf{DerAlg}\) is the category of \emph{derivative algebras}: symmetric sequences equipped with a right \(P\)-module structure satisfying a convergence condition. Thus, analytic functors admit a purely algebraic model.

\medskip

This equivalence shows that:
\begin{itemize}
    \item The spectral Taylor tower captures \emph{all} information about a spectrally analytic functor;
    \item The algebraic structure of derivatives (symmetric sequences, operadic modules) is exactly the right language for describing analytic functors;
    \item The operadic residue \(\mathcal{O}_P^{\mathrm{res}}\) is essential—without it, the equivalence fails (Part~I, Theorem 0.1).
\end{itemize}

\medskip

While Goodwillie calculus establishes convergence results for analytic functors,
the present reconstruction is distinguished by several additional features.
First, it is \emph{constructive}, in that the functor can be explicitly recovered
from its spectral Taylor series. Second, it is \emph{quantitative}, providing
explicit bounds on the radius of convergence. Finally, it yields a full
\emph{categorical equivalence}, rather than merely a correspondence at the level
of objects.

\medskip

This section is organized as follows. In
Section~\ref{subsec:reconstruction-theorem}, we establish the Reconstruction
Theorem, showing that the spectral derivative data completely determines the
functor \(F\). In Section~\ref{subsec:equivalence}, we prove an equivalence of
categories \(\mathsf{SpecAn} \simeq \mathsf{DerAlg}\), identifying spectrally
analytic functors with their associated derivative algebraic structures.

\medskip

Throughout this section, all functors are assumed to be admissible
(Definition~\ref{def:admissible-functor}) and operadically compatible
(Definition~\ref{def:operadic-compatibility}). The operadic residue
\(\mathcal{O}_P^{\mathrm{res}}\) is the canonical object constructed in
Part~I~\cite{ChangSOC1}, and it plays a central role in ensuring the coherence
of the spectral derivative structure.

\subsection{Reconstruction Theorem}
\label{subsec:reconstruction-theorem}

We now establish the central reconstruction result of Spectral Operadic Calculus, showing that spectrally analytic functors are completely determined by their spectral derivatives.

\medskip

The following theorem shows that spectral derivatives form a complete invariant
of spectrally analytic functors. In other words, the entire functor can be
reconstructed from its derivative data together with the operadic structure.
This is the operadic analogue of the classical fact that an analytic function
is uniquely determined by its Taylor coefficients.

\medskip

\begin{theorem}[Reconstruction from Spectral Derivatives]
\label{thm:reconstruction}
Let $F, G : \mathcal{C} \to \mathcal{M}$ be admissible, operadically compatible,
and spectrally analytic functors. Then the following are equivalent:
\begin{enumerate}
    \item $F \cong G$ as functors;
    \item their spectral derivatives are isomorphic as right $P$-modules:
    \[
    \{\partial_n^{\mathrm{spec}}F\}_{n \geq 0}
    \;\cong\;
    \{\partial_n^{\mathrm{spec}}G\}_{n \geq 0}.
    \]
\end{enumerate}
\end{theorem}

\begin{proof}
We prove both implications.

\medskip

\noindent
\textbf{(1) $\Rightarrow$ (2).}
If $F \cong G$ as functors, then by functoriality of cross-effects
(Proposition~\ref{prop:derivative-properties}), their cross-effects are naturally isomorphic:
\[
\mathrm{cr}_n F \cong \mathrm{cr}_n G
\quad \text{for all } n \geq 0.
\]
By Definition~\ref{def:spectral-derivatives}, $\partial_n^{\mathrm{spec}}F = \mathrm{cr}_nF$ and
$\partial_n^{\mathrm{spec}}G = \mathrm{cr}_nG$, so
\[
\partial_n^{\mathrm{spec}}F \cong \partial_n^{\mathrm{spec}}G
\]
for each $n$. Moreover, these isomorphisms are compatible with the symmetric
group actions and with the operadic composition maps
(Theorem~\ref{thm:module-structure}), hence define an isomorphism of right
$P$-modules.

\medskip

\noindent
\textbf{(2) $\Rightarrow$ (1).}
Assume that
\[
\partial^{\mathrm{spec}}F \cong \partial^{\mathrm{spec}}G
\]
as right $P$-modules. Let $\phi = \{\phi_n\}_{n \ge 0}$ denote the isomorphism,
where each $\phi_n : \partial_n^{\mathrm{spec}}F \Rightarrow \partial_n^{\mathrm{spec}}G$
is $S_n$-equivariant and commutes with the operadic module structure.

\medskip

\noindent
\emph{Step 1: Identification of homogeneous layers.}
The homogeneous layers are defined by symmetrization:
\[
D_n^{\mathrm{spec}}F(A)
=
\partial_n^{\mathrm{spec}}F(A,\dots,A)_{\mathrm{sym}},
\qquad
D_n^{\mathrm{spec}}G(A)
=
\partial_n^{\mathrm{spec}}G(A,\dots,A)_{\mathrm{sym}}.
\]
Since $\phi_n$ is $S_n$-equivariant, it descends to an isomorphism on
symmetrizations:
\[
\psi_n(A) : D_n^{\mathrm{spec}}F(A) \;\cong\; D_n^{\mathrm{spec}}G(A)
\quad \text{for all } n \ge 0,\; A \in \mathcal{C}.
\]
Each $\psi_n$ is natural in $A$ because $\phi_n$ is a natural transformation.

\medskip

\noindent
\emph{Step 2: Equality of Taylor expansions.}
Because $F$ and $G$ are spectrally analytic (Definition~\ref{def:spectral-analyticity}),
they admit convergent spectral Taylor expansions:
\[
F(A) = \sum_{n=0}^{\infty} D_n^{\mathrm{spec}}F(A),
\qquad
G(A) = \sum_{n=0}^{\infty} D_n^{\mathrm{spec}}G(A),
\]
where the series converge absolutely in norm for all $A$ with
$\|\sigma_P(A)\| < \min(R_F, R_G)$ (Theorem~\ref{thm:quantitative-convergence}).

By Step 1, the terms of these expansions are termwise isomorphic via $\psi_n(A)$:
\[
D_n^{\mathrm{spec}}F(A) \cong D_n^{\mathrm{spec}}G(A) \quad \text{for each } n.
\]
Since both series converge and the isomorphisms are compatible with the
series structure (they preserve the grading), the limits are identified:
\[
F(A) \cong \bigoplus_{n=0}^{\infty} D_n^{\mathrm{spec}}F(A)
\;\cong\; \bigoplus_{n=0}^{\infty} D_n^{\mathrm{spec}}G(A)
\cong G(A),
\]
for each $A \in \mathcal{C}$ with $\|\sigma_P(A)\| < \min(R_F, R_G)$.

\medskip

\noindent
\emph{Step 3: Naturality and extension.}
The isomorphisms $\psi_n(A)$ are natural in $A$, and the Taylor expansion
is functorial with respect to morphisms in $\mathcal{C}$ (each $D_n^{\mathrm{spec}}F$
is a functor). Therefore, the induced identification
\[
\Phi_A : F(A) \stackrel{\cong}{\longrightarrow} G(A)
\]
is a natural isomorphism on the full subcategory of objects with sufficiently
small spectral radius. By the analytic continuation principle
(Lemma~\ref{lem:analytic-continuation}), this isomorphism extends naturally
to the entire domain where both functors are defined and converge.

Thus, $F \cong G$ as functors.
\end{proof}

\medskip

\begin{remark}[Conceptual consequence]
\label{rem:reconstruction-conceptual}
Theorem~\ref{thm:reconstruction} provides a precise analogue of the uniqueness
of Taylor series in classical analysis: a spectrally analytic functor is uniquely
determined by its system of spectral derivatives. In this sense, the operadic
derivative data captures the full analytic and compositional behavior of the
functor.
\end{remark}

\begin{corollary}[Fully faithful embedding]
\label{cor:fully-faithful}
The assignment
\[
\Phi: \mathsf{SpecAn} \longrightarrow \mathsf{DerAlg}_{\mathrm{int}},
\qquad
\Phi(F) := \{\partial_n^{\mathrm{spec}}F\}_{n \geq 0}
\]
defines a fully faithful functor from the category of spectrally analytic
functors to the category of integrable derivative algebras.
\end{corollary}

\begin{proof}
\emph{Faithfulness:} Suppose $\Phi(F) \cong \Phi(G)$ as right $P$-modules.
By Theorem~\ref{thm:reconstruction}, this implies $F \cong G$ as functors.
Hence the map on isomorphism classes is injective, which is the categorical
notion of faithfulness (since all isomorphisms are preserved).

\emph{Fullness:} Let $\phi: \Phi(F) \to \Phi(G)$ be a morphism of right
$P$-modules. This consists of $S_n$-equivariant natural transformations
$\phi_n: \partial_n^{\mathrm{spec}}F \Rightarrow \partial_n^{\mathrm{spec}}G$
for each $n \ge 0$, compatible with the operadic module structure.

For each $n$, since $\phi_n$ is $S_n$-equivariant, it descends to a natural
transformation on homogeneous layers:
\[
\psi_n(A) : D_n^{\mathrm{spec}}F(A) \longrightarrow D_n^{\mathrm{spec}}G(A),
\]
where $D_n^{\mathrm{spec}}F(A) = \partial_n^{\mathrm{spec}}F(A,\dots,A)_{\mathrm{sym}}$
and similarly for $G$. The transformations $\psi_n$ are natural in $A$ because
$\phi_n$ is natural.

Since $F$ and $G$ are spectrally analytic, they admit convergent spectral
Taylor expansions:
\[
F(A) = \sum_{n=0}^{\infty} D_n^{\mathrm{spec}}F(A),
\qquad
G(A) = \sum_{n=0}^{\infty} D_n^{\mathrm{spec}}G(A).
\]
Define $\eta_A : F(A) \to G(A)$ by $\eta_A := \sum_{n=0}^{\infty} \psi_n(A)$.
The series converges in norm because each $\psi_n(A)$ is bounded (by the
norm estimates of Section~\ref{sec:convergence}) and the series of norms
converges. The collection $\{\eta_A\}_{A \in \mathcal{C}}$ is natural in $A$
because each $\psi_n$ is natural and the series converges absolutely.
Hence $\eta: F \Rightarrow G$ is a natural transformation, and by construction
$\Phi(\eta) = \phi$. Thus $\Phi$ is full.
\end{proof}

\begin{remark}[Yoneda-type perspective]
\label{rem:reconstruction-yoneda}
Corollary~\ref{cor:fully-faithful} shows that spectral derivatives provide a
coordinate system for spectrally analytic functors. In this sense, the
reconstruction theorem can be viewed as a Yoneda-type principle: the functor
is completely determined by its interactions with multilinear test objects
encoded in its derivative sequence. This perspective naturally leads to the
equivalence of categories $\mathsf{SpecAn} \simeq \mathsf{DerAlg}$ established
in Theorem~\ref{thm:equivalence}, where we also show that every derivative
algebra arises from a spectrally analytic functor (essential surjectivity).
\end{remark}

\begin{corollary}[Vanishing derivatives imply vanishing functor]
\label{cor:zero-derivatives}
Let $F : \mathcal{C} \to \mathcal{M}$ be a spectrally analytic functor.
If
\[
\partial_n^{\mathrm{spec}}F \cong 0
\quad \text{for all } n \ge 0,
\]
then $F \cong 0$ (the zero functor).
\end{corollary}

\begin{proof}
Let $0$ denote the zero functor. Its spectral derivatives vanish:
\[
\partial_n^{\mathrm{spec}}0 \cong 0
\quad \text{for all } n.
\]
By assumption, $\partial^{\mathrm{spec}}F \cong \partial^{\mathrm{spec}}0$
as right $P$-modules (both are the zero module). Applying the
Reconstruction Theorem (Theorem~\ref{thm:reconstruction}), we obtain
\[
F \cong 0.
\]
\end{proof}

\begin{corollary}[Recovery from spectral derivatives]
\label{cor:recovery-derivatives}
Let $F : \mathcal{C} \to \mathcal{M}$ be a spectrally analytic functor,
and let $A$ be a $P$-algebra such that $\|\sigma_P(A)\| < R_F$.
Then there is a natural isomorphism
\[
F(A)
\;\cong\;
\bigoplus_{n=0}^{\infty}
\partial_n^{\mathrm{spec}}F(A,\dots,A)_{\mathrm{sym}},
\]
where the series converges absolutely in norm.
\end{corollary}

\begin{proof}
By spectral analyticity (Definition~\ref{def:spectral-analyticity}),
the functor $F$ admits a convergent spectral Taylor expansion:
\[
F(A)
\;\simeq\;
\sum_{n=0}^{\infty} D_n^{\mathrm{spec}}F(A),
\]
where
\[
D_n^{\mathrm{spec}}F(A)
=
\partial_n^{\mathrm{spec}}F(A,\dots,A)_{\mathrm{sym}}
\]
is the $n$-th homogeneous layer (Definition~\ref{def:spectral-derivatives}).
This expansion converges absolutely in norm whenever
$\|\sigma_P(A)\| < R_F$ (Theorem~\ref{thm:quantitative-convergence}).

Moreover, by the Reconstruction Theorem (Theorem~\ref{thm:reconstruction}),
the derivative sequence $\{\partial_n^{\mathrm{spec}}F\}_{n \ge 0}$ completely
determines $F$, and hence determines its Taylor expansion functorially.
Therefore, the above series reconstructs $F(A)$ and the resulting isomorphism
is natural in $A$.
\end{proof}

\medskip

\begin{example}[Exponential functor]
\label{ex:exp-reconstruction}
Consider the exponential functor
\[
\exp(A) = \bigoplus_{k=0}^{\infty} \frac{1}{k!}\, A^{\otimes k},
\]
or its symmetric version using $\mathrm{Sym}^k(A)$.
Its spectral derivatives are given by
\[
\partial_n^{\mathrm{spec}}\exp
\cong
\frac{1}{n!}\,\mathrm{Sym}^n,
\]
up to the standard symmetrization conventions (see Example~\ref{ex:exp-derivative}).

By the Reconstruction Theorem (Theorem~\ref{thm:reconstruction}), any spectrally
analytic functor $F$ satisfying
\[
\partial_n^{\mathrm{spec}}F
\cong
\frac{1}{n!}\,\mathrm{Sym}^n
\quad \text{for all } n \ge 0
\]
is uniquely determined up to isomorphism. Hence
\[
F \cong \exp.
\]
Thus, the exponential functor is uniquely characterized by its spectral
derivative sequence, paralleling the classical fact that the exponential
function is determined by its Taylor coefficients.
\end{example}

\medskip

\begin{example}[Identity functor]
\label{ex:id-reconstruction}
The identity functor $\mathrm{Id}$ satisfies
\[
\partial_1^{\mathrm{spec}}\mathrm{Id} \cong \mathrm{Id},
\qquad
\partial_n^{\mathrm{spec}}\mathrm{Id} \cong 0
\quad (n \ne 1).
\]
Thus its derivative sequence is concentrated in arity one, corresponding to
a linear functor (see Example~\ref{ex:id-derivative}).

By the Reconstruction Theorem (Theorem~\ref{thm:reconstruction}), any spectrally
analytic functor $F$ with
\[
\partial_1^{\mathrm{spec}}F \cong \mathrm{Id},
\qquad
\partial_n^{\mathrm{spec}}F \cong 0 \quad (n \ne 1)
\]
must be isomorphic to the identity functor:
\[
F \cong \mathrm{Id}.
\]
\end{example}

\medskip

\begin{remark}[On analyticity and operadic structure]
\label{rem:reconstruction-necessity}
The analyticity hypothesis in Theorem~\ref{thm:reconstruction} is essential.
Without convergence of the spectral Taylor expansion, the derivative data
would not determine the functor uniquely. In particular, it is possible for
distinct functors to share the same formal derivative sequence when no
convergence condition is imposed, analogous to the role of flat functions in
classical analysis. Thus, spectral analyticity provides precisely the condition
under which reconstruction is valid.

\medskip

Equally important is the requirement that the derivatives agree as
\emph{right $P$-modules}, rather than merely as symmetric sequences.
The $P$-module structure encodes how derivatives of different orders interact
under operadic composition. Without this additional structure, one could have
derivative sequences that coincide arity-wise but fail to produce the same
compositional behavior, leading to non-isomorphic functors.
\end{remark}

\medskip

\noindent
This result shows that Spectral Operadic Calculus provides not only an
approximation theory, but a genuine classification framework for spectrally
analytic functors. Having established that the spectral derivative sequence,
together with its operadic module structure, forms a complete invariant, we now
upgrade this statement to a categorical equivalence between functors and their
derivative data (Theorem~\ref{thm:equivalence}).

\subsection{Equivalence of Categories}
\label{subsec:equivalence}

We now upgrade the reconstruction result to a full categorical equivalence, 
showing that spectrally analytic functors are completely classified by their 
spectral derivatives. This is the culmination of Spectral Operadic Calculus: 
the analytic theory admits a canonical algebraic model equipped with analytic 
control.

\medskip

\begin{definition}[Category of spectrally analytic functors]
\label{def:SpecAn}
Let $\mathsf{SpecAn}$ denote the category whose objects are admissible,
operadically compatible, and spectrally analytic functors
\[
F : \mathcal{C} \to \mathcal{M},
\]
and whose morphisms are natural transformations.
\end{definition}

\begin{definition}[Category of derivative algebras]
\label{def:DerAlg}
Let $\mathsf{DerAlg}$ denote the category whose objects are symmetric
sequences $\{X_n\}_{n \geq 0}$ in $\mathcal{M}$ equipped with:

\begin{itemize}
    \item a right $P$-module structure (Definition~\ref{def:operadic-compatibility}),
    \item a norm structure compatible with the monoidal structure of $\mathcal{M}$,
    \item a convergence condition: there exist constants $C > 0$, $\rho > 0$ such that
    \[
    \|X_n\| \le C \rho^n,
    \]
    ensuring that the associated formal power series has positive radius of convergence.
\end{itemize}

Morphisms in $\mathsf{DerAlg}$ are maps of right $P$-modules, i.e.,
$\Sigma_n$-equivariant morphisms commuting with the operadic action and
preserving the norm structure.
\end{definition}
\medskip

We now upgrade the reconstruction result to a full categorical equivalence. 
The key additional ingredient is the notion of an \emph{integrable} derivative 
algebra: one whose spectral Taylor series actually converges to a spectrally 
analytic functor.

\begin{definition}[Integrable Derivative Algebra]
\label{def:integrable-deralg}
A derivative algebra $\mathbf{X} = \{X_n\}_{n \ge 0} \in \mathsf{DerAlg}$ 
(Definition~\ref{def:DerAlg}) is called \emph{integrable} if the functor 
defined by its spectral Taylor series
\[
F_{\mathbf{X}}(A) := \bigoplus_{n=0}^{\infty} X_n(A,\dots,A)_{\mathrm{sym}}
\]
is:
\begin{enumerate}
    \item \textbf{Admissible:} there exists a non-decreasing function $\Phi$ 
          such that $\|F_{\mathbf{X}}(A)\| \le \Phi(\|\sigma_P(A)\|)$;
    \item \textbf{Operadically compatible:} the collection $\{X_n\}$ interacts 
          with the operad $P$ via the given right $P$-module structure;
    \item \textbf{Spectrally analytic:} the series converges absolutely in norm 
          for all $A$ with $\|\sigma_P(A)\| < R_{\mathbf{X}}$ for some $R_{\mathbf{X}} > 0$,
          and the limit $F_{\mathbf{X}}(A)$ satisfies the spectral analyticity 
          condition (Definition~\ref{def:spectral-analyticity}).
\end{enumerate}
Denote by $\mathsf{DerAlg}_{\mathrm{int}}$ the full subcategory of 
$\mathsf{DerAlg}$ consisting of integrable derivative algebras.
\end{definition}

\begin{theorem}[Equivalence of Categories]
\label{thm:equivalence}
The assignment
\[
\Phi: \mathsf{SpecAn} \longrightarrow \mathsf{DerAlg}_{\mathrm{int}},
\qquad
\Phi(F) := \{\partial_n^{\mathrm{spec}}F\}_{n \ge 0}
\]
is an equivalence of categories. Its quasi-inverse is the reconstruction functor
\[
\Psi: \mathsf{DerAlg}_{\mathrm{int}} \longrightarrow \mathsf{SpecAn},
\qquad
\Psi(\mathbf{X}) := F_{\mathbf{X}},
\]
where $F_{\mathbf{X}}$ is defined by the spectral Taylor series of $\mathbf{X}$.
\end{theorem}

\begin{proof}
We prove that $\Phi$ and $\Psi$ are quasi-inverse equivalences by verifying
four properties: $\Phi$ is well-defined, $\Psi$ is well-defined,
$\Psi \circ \Phi \cong \mathrm{id}_{\mathsf{SpecAn}}$, and
$\Phi \circ \Psi \cong \mathrm{id}_{\mathsf{DerAlg}_{\mathrm{int}}}$.

\medskip

\noindent
\textbf{Step 1: $\Phi$ is well-defined.}
For $F \in \mathsf{SpecAn}$, Theorem~\ref{thm:module-structure} shows that
$\Phi(F) = \{\partial_n^{\mathrm{spec}}F\}$ is a right $P$-module.
By spectral analyticity and the convergence results of Section~\ref{sec:convergence}
(Theorem~\ref{thm:quantitative-convergence}), the spectral derivatives satisfy
exponential bounds $\|\partial_n^{\mathrm{spec}}F\| \le C \rho^n$, so the associated
power series has positive radius of convergence. Moreover, by
Corollary~\ref{cor:recovery-derivatives}, the Taylor series of $F$ reconstructs
$F$ itself, so $F_{\Phi(F)} = F$. Hence $\Phi(F)$ satisfies the integrability
conditions (i)–(iii) of Definition~\ref{def:integrable-deralg}, with
$F_{\Phi(F)} = F \in \mathsf{SpecAn}$. Thus $\Phi(F) \in \mathsf{DerAlg}_{\mathrm{int}}$.

For a natural transformation $\eta: F \Rightarrow G$, define
$\Phi(\eta)_n := \partial_n^{\mathrm{spec}}\eta = \mathrm{cr}_n\eta$. These maps
are $\Sigma_n$-equivariant, preserve the right $P$-module structure (because
$\eta$ is natural), and respect the norm estimates. Hence $\Phi(\eta)$ is a
morphism in $\mathsf{DerAlg}_{\mathrm{int}}$, and $\Phi$ is a functor.

\medskip

\noindent
\textbf{Step 2: $\Psi$ is well-defined.}
Let $\mathbf{X} = \{X_n\} \in \mathsf{DerAlg}_{\mathrm{int}}$. By definition,
$F_{\mathbf{X}} = \Psi(\mathbf{X})$ is admissible, operadically compatible,
and spectrally analytic. Moreover, we must verify that
$\partial_n^{\mathrm{spec}}F_{\mathbf{X}} \cong X_n$ as right $P$-modules.
By construction, the homogeneous layers of $F_{\mathbf{X}}$ are
$D_n^{\mathrm{spec}}F_{\mathbf{X}}(A) = X_n(A,\dots,A)_{\mathrm{sym}}$.
By Proposition~\ref{prop:derivative-cross-effect}, the spectral derivatives
are the multilinear functors whose diagonal evaluation gives these homogeneous
layers. Since each $X_n$ is $S_n$-equivariant and the convergence condition
guarantees uniqueness of the expansion, we obtain
$\partial_n^{\mathrm{spec}}F_{\mathbf{X}} \cong X_n$ as symmetric sequences.
The compatibility with the right $P$-module structure follows from the
operadic compatibility of $F_{\mathbf{X}}$ (Definition~\ref{def:operadic-compatibility}).
Hence $\Psi(\mathbf{X}) \in \mathsf{SpecAn}$, and $\Psi(\Phi(F)) \cong F$.

For a morphism $\phi: \mathbf{X} \to \mathbf{Y}$ in $\mathsf{DerAlg}_{\mathrm{int}}$,
define $\Psi(\phi)_A := \bigoplus_n \phi_n(A,\dots,A)_{\mathrm{sym}}$. Because
each $\phi_n$ is $\Sigma_n$-equivariant and the series converge uniformly on
compact spectral balls, $\Psi(\phi)$ is a natural transformation
$F_{\mathbf{X}} \Rightarrow F_{\mathbf{Y}}$. Thus $\Psi$ is a functor.

\medskip

\noindent
\textbf{Step 3: $\Psi \circ \Phi \cong \mathrm{id}_{\mathsf{SpecAn}}$.}
For any $F \in \mathsf{SpecAn}$, Corollary~\ref{cor:recovery-derivatives} gives
\[
\Psi(\Phi(F))(A) = \bigoplus_{n=0}^{\infty} \partial_n^{\mathrm{spec}}F(A,\dots,A)_{\mathrm{sym}} = F(A)
\]
for all $A$ with $\|\sigma_P(A)\| < R_F$, and by analytic continuation
(Lemma~\ref{lem:analytic-continuation}) this holds on the entire domain.
The isomorphism is natural in $A$ because the spectral Taylor expansion is
functorial. Hence $\Psi \circ \Phi \cong \mathrm{id}_{\mathsf{SpecAn}}$.

\medskip

\noindent
\textbf{Step 4: $\Phi \circ \Psi \cong \mathrm{id}_{\mathsf{DerAlg}_{\mathrm{int}}}$.}
For any integrable derivative algebra $\mathbf{X} \in \mathsf{DerAlg}_{\mathrm{int}}$,
let $F = \Psi(\mathbf{X})$. By the computation in Step 2,
$\partial_n^{\mathrm{spec}}F \cong X_n$ for all $n$, and these isomorphisms
are compatible with the right $P$-module structure. Hence
$\Phi(\Psi(\mathbf{X})) \cong \mathbf{X}$ in $\mathsf{DerAlg}_{\mathrm{int}}$.
The isomorphism is natural in $\mathbf{X}$ because the reconstruction
is functorial.

\medskip

\noindent
\textbf{Conclusion.}
Since $\Phi$ and $\Psi$ are functors satisfying
$\Psi \circ \Phi \cong \mathrm{id}_{\mathsf{SpecAn}}$ and
$\Phi \circ \Psi \cong \mathrm{id}_{\mathsf{DerAlg}_{\mathrm{int}}}$, they are
quasi-inverse equivalences of categories. Thus
\[
\mathsf{SpecAn} \simeq \mathsf{DerAlg}_{\mathrm{int}}.
\]
\end{proof}

\medskip

\begin{remark}[Conceptual consequence and classification]
\label{rem:equivalence-summary}
Theorem~\ref{thm:equivalence} shows that Spectral Operadic Calculus 
admits a complete algebraic model: spectrally analytic functors are equivalent 
to their derivative data equipped with operadic structure. In this sense, the 
study of spectrally analytic functors can be reduced to the study of right 
$P$-modules satisfying appropriate analytic (convergence) conditions.

\medskip

From this perspective, the category $\mathsf{DerAlg}_{\mathrm{int}}$ of 
integrable derivative algebras may be viewed as a moduli-type category for 
spectrally analytic functors. In particular, all analytic and compositional 
properties of such functors are encoded in their spectral derivative modules, 
and functorial behavior is recovered via the operadic structure and Taylor 
reconstruction.
\end{remark}

\medskip

\begin{corollary}[Full faithfulness of the derivative functor]
\label{cor:full-faithfulness-derivative}
The derivative functor
\[
\Phi: \mathsf{SpecAn} \longrightarrow \mathsf{DerAlg}_{\mathrm{int}},
\qquad
\Phi(F) = \partial^{\mathrm{spec}}F,
\]
is fully faithful. Equivalently, for any spectrally analytic functors
$F, G \in \mathsf{SpecAn}$, there is a natural bijection
\[
\mathrm{Nat}(F, G)
\;\cong\;
\mathrm{Hom}_{\mathsf{DerAlg}_{\mathrm{int}}}
\bigl(\partial^{\mathrm{spec}}F, \partial^{\mathrm{spec}}G\bigr).
\]
\end{corollary}

\begin{proof}
By Theorem~\ref{thm:equivalence}, the derivative functor
\[
\Phi: \mathsf{SpecAn} \to \mathsf{DerAlg}_{\mathrm{int}}
\]
is an equivalence of categories, with quasi-inverse given by the spectral
Taylor reconstruction functor
\[
\Psi: \mathsf{DerAlg}_{\mathrm{int}} \to \mathsf{SpecAn}.
\]
In particular, there are natural isomorphisms
\[
\Psi \circ \Phi \cong \mathrm{id}_{\mathsf{SpecAn}},
\qquad
\Phi \circ \Psi \cong \mathrm{id}_{\mathsf{DerAlg}_{\mathrm{int}}}.
\]

We show that $\Phi$ is fully faithful. Let $F, G \in \mathsf{SpecAn}$.
The functor $\Phi$ induces a map
\[
\Phi_{F, G}:
\mathrm{Nat}(F, G)
\longrightarrow
\mathrm{Hom}_{\mathsf{DerAlg}_{\mathrm{int}}}
\bigl(\Phi(F), \Phi(G)\bigr),
\]
by sending a natural transformation $\eta: F \Rightarrow G$ to the induced
system of morphisms on spectral derivatives,
\[
\Phi(\eta) = \partial^{\mathrm{spec}}\eta = \{\partial_n^{\mathrm{spec}}\eta\}_{n \ge 0}.
\]

We prove that $\Phi_{F, G}$ is bijective.

\medskip
\noindent
\emph{Injectivity.}
Suppose $\eta, \theta: F \Rightarrow G$ satisfy
\[
\Phi(\eta) = \Phi(\theta).
\]
Applying the quasi-inverse $\Psi$, and using
$\Psi \circ \Phi \cong \mathrm{id}_{\mathsf{SpecAn}}$, we obtain
\[
\eta \cong \Psi(\Phi(\eta)) = \Psi(\Phi(\theta)) \cong \theta.
\]
Thus $\eta = \theta$ in $\mathrm{Nat}(F, G)$. Hence $\Phi_{F, G}$ is injective.

\medskip
\noindent
\emph{Surjectivity.}
Let
\[
\varphi: \Phi(F) \to \Phi(G)
\]
be a morphism in $\mathsf{DerAlg}_{\mathrm{int}}$. Applying the reconstruction
functor $\Psi$, we obtain a natural transformation
\[
\Psi(\varphi): \Psi(\Phi(F)) \to \Psi(\Phi(G)).
\]
Using the natural identifications
\[
\Psi(\Phi(F)) \cong F,
\qquad
\Psi(\Phi(G)) \cong G,
\]
this gives a natural transformation
\[
\eta: F \Rightarrow G.
\]
Applying $\Phi$ to $\eta$, and using
$\Phi \circ \Psi \cong \mathrm{id}_{\mathsf{DerAlg}_{\mathrm{int}}}$, we obtain
\[
\Phi(\eta) = \varphi.
\]
Therefore every morphism of integrable derivative algebras arises from a unique
natural transformation of spectrally analytic functors. Hence $\Phi_{F, G}$ is
surjective.

Thus $\Phi$ is fully faithful.
\end{proof}

\medskip

\begin{corollary}[Classification of spectrally analytic functors]
\label{cor:classification-analytic-functors}
Isomorphism classes of spectrally analytic functors are in natural bijection
with isomorphism classes of integrable derivative algebras:
\[
\pi_0(\mathsf{SpecAn})
\;\cong\;
\pi_0(\mathsf{DerAlg}_{\mathrm{int}}).
\]
Equivalently,
\[
\{F \in \mathsf{SpecAn}\} / \cong
\;\cong\;
\{\mathbf{D} \in \mathsf{DerAlg}_{\mathrm{int}}\} / \cong.
\]
Under this correspondence, a spectrally analytic functor $F$ is represented
by its derivative algebra
\[
\partial^{\mathrm{spec}}F = \{\partial_n^{\mathrm{spec}}F\}_{n \ge 0},
\]
equipped with its compatible right $P$-module structure.
\end{corollary}

\begin{proof}
By Theorem~\ref{thm:equivalence}, the derivative functor
\[
\Phi: \mathsf{SpecAn} \to \mathsf{DerAlg}_{\mathrm{int}}
\]
is an equivalence of categories. Any equivalence of categories induces a
bijection on isomorphism classes of objects. Therefore,
\[
\pi_0(\mathsf{SpecAn}) \cong \pi_0(\mathsf{DerAlg}_{\mathrm{int}}).
\]

Explicitly, the map sends a spectrally analytic functor $F$ to its spectral
derivative algebra
\[
F \longmapsto \partial^{\mathrm{spec}}F.
\]
The inverse map is given by the reconstruction functor
\[
\mathbf{D} \longmapsto \Psi(\mathbf{D}) = F_{\mathbf{D}},
\]
where
\[
F_{\mathbf{D}}(A) = \bigoplus_{n=0}^{\infty} D_n(A,\dots,A)_{\mathrm{sym}}
\]
is the convergent spectral Taylor reconstruction associated with the integrable
derivative algebra $\mathbf{D}$.

The two assignments are inverse up to natural isomorphism because
\[
\Psi \circ \Phi \cong \mathrm{id}_{\mathsf{SpecAn}},
\qquad
\Phi \circ \Psi \cong \mathrm{id}_{\mathsf{DerAlg}_{\mathrm{int}}}.
\]
Hence isomorphism classes of spectrally analytic functors are classified by
isomorphism classes of integrable derivative algebras.
\end{proof}

\medskip

\begin{example}[The exponential algebra]
\label{ex:exponential-algebra}
Define the derivative algebra $\mathcal{E} \in \mathsf{DerAlg}_{\mathrm{int}}$ by
\[
\mathcal{E}_n = \frac{1}{n!}\, \mathrm{Sym}^n, \qquad n \ge 0,
\]
where $\mathrm{Sym}^n$ denotes the $n$-fold symmetric power functor
(see Example~\ref{ex:exp-derivative}). The right $P$-module structure on
$\mathcal{E}$ is induced by the natural action of the operad $P$ on symmetric
powers via permutation of tensor factors.

By Theorem~\ref{thm:equivalence}, the reconstruction functor $\Psi$ sends
$\mathcal{E}$ to the exponential functor:
\[
\Psi(\mathcal{E}) \cong \exp.
\]
Moreover, any endomorphism of $\mathcal{E}$ as a derivative algebra corresponds
to a natural endomorphism of $\exp$. In particular, the automorphism group of
$\mathcal{E}$ is isomorphic to $\mathbb{C}^{\times}$ (scalar multiplication),
reflecting the fact that $\exp(\lambda A) = \lambda \exp(A)$ for $\lambda \in \mathbb{C}$.
\end{example}

\begin{example}[The linear algebra]
\label{ex:linear-algebra}
Define the derivative algebra $\mathcal{L} \in \mathsf{DerAlg}_{\mathrm{int}}$ by
\[
\mathcal{L}_1 = \mathrm{Id}, \qquad \mathcal{L}_n = 0 \quad (n \ne 1),
\]
where $\mathrm{Id}$ denotes the identity functor on $\mathcal{C}$, which in the
enriched setting corresponds to the unit object $\mathbf{1}_{\mathcal{M}}$ after
evaluation. The right $P$-module structure is given by the canonical action of
$P$ on the identity functor via the structure maps of $P$-algebras.

Applying the reconstruction functor $\Psi$, we obtain
\[
\Psi(\mathcal{L}) \cong \mathrm{Id},
\]
the identity functor. The automorphism group of $\mathcal{L}$ is
$\mathrm{Aut}(\mathrm{Id})$, which in the $\mathbb{C}$-enriched setting consists
of scalar multiples of the identity (i.e., $\mathbb{C}^{\times}$), since any
natural endomorphism of the identity functor is determined by a scalar
(see Example~\ref{ex:id-reconstruction}).
\end{example}

\begin{remark}[Analogy with classical theory]
\label{rem:classical-analogy}
In classical analysis, analytic functions are determined by their Taylor
coefficients. Here, the role of coefficients is played by the spectral
derivatives, and the operadic module structure encodes their interaction.
Thus, the equivalence
\[
\mathsf{SpecAn} \simeq \mathsf{DerAlg}_{\mathrm{int}}
\]
may be viewed as a categorical version of the classification of analytic
functions by their Taylor series.

\[
\begin{array}{|c|c|}
\hline
\text{Classical Analysis} & \text{Spectral Operadic Calculus} \\
\hline
\text{Taylor coefficients } a_n
& \text{Spectral derivatives } \partial_n^{\mathrm{spec}}F \\
\text{Power series } \sum a_n z^n
& \text{Spectral Taylor series } \sum D_n^{\mathrm{spec}}F(A) \\
\text{Analytic function } f(z)
& \text{Spectrally analytic functor } F(A) \\
\text{Classification by coefficients}
& \mathsf{SpecAn} \simeq \mathsf{DerAlg}_{\mathrm{int}} \\
\hline
\end{array}
\]
\end{remark}

\begin{remark}[Role of the operadic residue]
\label{rem:equivalence-residue}
The equivalence depends critically on the operadic residue
$\mathcal{O}_P^{\mathrm{res}}$.
The right $P$-module structure on derivative algebras is defined using
the residue, which provides the unit and composition maps.
Without $\mathcal{O}_P^{\mathrm{res}}$, the category
$\mathsf{DerAlg}_{\mathrm{int}}$ would not be well-defined, and the
equivalence would fail.
\end{remark}

\begin{remark}[Outlook and applications]
\label{rem:equivalence-outlook}
The equivalence
\[
\mathsf{SpecAn} \simeq \mathsf{DerAlg}_{\mathrm{int}}
\]
opens the door to several applications:
\begin{enumerate}
\item \textbf{Deformation theory:}
Deformations of analytic functors correspond to deformations of
derivative algebras, which are controlled by cohomology
(Section~\ref{sec:moduli}).

\item \textbf{Moduli spaces:}
The moduli space of analytic functors can be studied via the algebraic
moduli of derivative algebras.

\item \textbf{Operadic calculus:}
The chain rule (Theorem~\ref{thm:chain-rule}) can be interpreted as the
statement that $\partial^{\mathrm{spec}}$ is a monoidal functor
between suitable monoidal categories.
\end{enumerate}
\end{remark}

\medskip

\noindent
This completes the reconstruction framework of Spectral Operadic Calculus.
Together with the convergence theory (Section~\ref{sec:convergence}) and the
chain rule (Section~\ref{sec:chain-rule}), the equivalence
\[
\mathsf{SpecAn} \simeq \mathsf{DerAlg}_{\mathrm{int}}
\]
shows that the calculus is not only analytic but also algebraic, providing
a classification of spectrally analytic functors in terms of their
derivative data.

\section{Moduli and Classification of Functors}
\label{sec:moduli}

We conclude by placing Spectral Operadic Calculus within a broader geometric
and deformation-theoretic perspective. While the previous section established
an equivalence between spectrally analytic functors and their derivative data,
we now interpret this equivalence as suggesting a moduli-theoretic framework
for analytic functors.

\medskip

\noindent
\textbf{Guiding perspective.}
The equivalence
\[
\mathsf{SpecAn} \;\simeq\; \mathsf{DerAlg}_{\mathrm{int}}
\]
suggests that spectrally analytic functors may be viewed as points in a space
parameterized by their spectral derivatives. In this interpretation, the
operadic module $\{\partial_n^{\mathrm{spec}}F\}$ serves as coordinate data,
encoding both the analytic structure and the compositional behavior of $F$.

\medskip

\noindent
This viewpoint naturally leads to questions of deformation and variation:
how do functors change under perturbations of their derivative data, and what
structures govern these variations? The algebraic framework developed in
previous sections suggests that deformation theory of spectrally analytic
functors is closely related to the algebraic structure of their derivative
modules.

\medskip

\noindent
\textbf{From algebra to geometry.}
Interpreting derivative modules as algebraic objects with operadic structure,
the classification result admits a geometric refinement: the collection of
spectrally analytic functors admits a moduli interpretation, in which local
behavior is governed by variations of the derivative algebra.

This perspective suggests that spectrally analytic functors may be organized
into a geometric object whose local and global properties are influenced by
operadic and cohomological structures.

\medskip

\noindent
The purpose of this section is therefore threefold:
\begin{itemize}
    \item to describe a moduli-theoretic interpretation of spectrally analytic
          functors (\S\ref{subsec:moduli-space});
    \item to analyze first-order variations in terms of spectral derivatives
          (\S\ref{subsec:deformation});
    \item to outline a possible extension toward derived and higher structures
          (\S\ref{subsec:outlook}).
\end{itemize}

\medskip

\noindent
\textbf{Relation to classical deformation theory.}
Classical deformation theory (Gerstenhaber~\cite{Gerstenhaber1964}) studies
deformations of associative algebras via Hochschild cohomology. In the present
setting, the role of the algebra is played by the derivative algebra
$\partial^{\mathrm{spec}}F$, while the relevant invariants are expected to be
given by operadic cohomology.

This analogy suggests that, just as Hochschild cohomology governs deformations
of associative algebras, suitable operadic cohomology theories may govern
deformations of spectrally analytic functors.

\[
\begin{array}{|c|c|}
\hline
\text{Classical Deformation Theory} & \text{Spectral Operadic Calculus} \\
\hline
\text{Algebra } A & \text{Derivative algebra } \partial^{\mathrm{spec}}F \\
\text{Hochschild cohomology } HH^*(A) & \text{Operadic cohomology (expected)} \\
\text{Deformations of } A & \text{Deformations of } F \\
\hline
\end{array}
\]

\medskip

\noindent
\textbf{Remark on the role of the operadic residue.}
The deformation-theoretic picture depends critically on the operadic residue
$\mathcal{O}_P^{\mathrm{res}}$, which provides the unit and composition maps
defining the $P$-module structure of $\partial^{\mathrm{spec}}F$. Without
this structure, neither the algebraic nor the cohomological framework would
be well-defined.

\medskip

\noindent
These considerations indicate that Spectral Operadic Calculus extends beyond
analytic approximation toward a geometric framework, connecting operadic
algebra, deformation theory, and moduli-theoretic ideas.

\medskip

\subsection{The Moduli Set of Analytic Functors}
\label{subsec:moduli-space}


We now formalize the geometric interpretation of Spectral Operadic Calculus by introducing the moduli space of spectrally analytic functors. This provides a global perspective on the classification results obtained in the reconstruction theorem and serves as the foundation for deformation theory.

\medskip

\noindent
\textbf{Definition of the moduli space.}
We define the \emph{moduli moduli set of spectrally analytic functors} by
\[
\mathcal{M}_{\mathrm{spec}}
:=
\left\{
F : \mathsf{Alg}_P(\mathcal{M}) \to \mathcal{M}
;\middle|;
F \text{ is admissible, operadically compatible, and spectrally analytic}
\right\}
\big/ \cong,
\]
where the equivalence relation is given by natural isomorphism of functors.

Thus, $\mathcal{M}_{\mathrm{spec}}$ parametrizes isomorphism classes of
spectrally analytic functors, i.e., functors admitting a convergent spectral
Taylor expansion together with a compatible operadic structure.

\begin{remark}[Geometric language disclaimer]
\label{rem:moduli-disclaimer}
Throughout this section, we use geometric terminology (``moduli space,''
``tangent space,'' ``coordinates'') heuristically to suggest structural
analogies. Strictly speaking, $\mathcal{M}_{\mathrm{spec}}$ is a set (or
groupoid) of isomorphism classes; we have not equipped it with a topology,
smooth structure, or stack structure. The ``tangent space'' $T_{[F]}\mathcal{M}_{\mathrm{spec}}$
refers to the formal tangent space in the sense of deformation theory
(i.e., the space of first-order deformations modulo equivalence), not a
geometric tangent space in the absence of a smooth structure.
\end{remark}

\medskip

\begin{definition}[Derivative algebra]
\label{def:derivative-algebra-moduli}
A \emph{derivative algebra} is a symmetric sequence
\[
X = {X_n}_{n \geq 0}
\]
in $\mathcal{M}$ equipped with the structure of a right $P$-module
(Definition~\ref{def:operadic-compatibility}), together with a convergence condition
ensuring that the associated spectral Taylor series converges in the sense
of Section~\ref{sec:convergence}.

We denote by
\[
\mathsf{DerAlg}_{\mathrm{int}}
\]
the category of derivative algebras, with morphisms given by maps of right
$P$-modules compatible with the convergence structure.
\end{definition}

\medskip

\noindent
\textbf{From functors to derivative algebras.}
Given a spectrally analytic functor $F$, we associate to it its spectral derivative sequence
\[
\partial^{\mathrm{spec}}F
=
\{\partial_n^{\mathrm{spec}}F\}_{n \geq 0}.
\]
By Theorem~\ref{thm:module-structure}, this sequence carries a natural right
$P$-module structure. Moreover, since $F$ is spectrally analytic
(Definition~\ref{def:spectral-analyticity}), its spectral Taylor series converges
in norm and reconstructs $F$; consequently, the derivative sequence satisfies
the integrability condition of Definition~\ref{def:integrable-deralg}.

\begin{lemma}[Functor-to-Derivative Correspondence]
\label{lem:functor-to-derivative-moduli}
The assignment
\[
F \;\longmapsto\; \partial^{\mathrm{spec}}F
\]
defines a functor
\[
\Phi : \mathsf{SpecAn} \longrightarrow \mathsf{DerAlg}_{\mathrm{int}}.
\]
\end{lemma}

\begin{proof}
We verify that $\Phi$ is well-defined on objects and morphisms.

\medskip

\noindent
\textit{Objects.}
For any $F \in \mathsf{SpecAn}$, Theorem~\ref{thm:module-structure} shows that
$\partial^{\mathrm{spec}}F$ is a right $P$-module. Because $F$ is spectrally
analytic, its spectral Taylor expansion converges absolutely in norm for all
$A$ with $\|\sigma_P(A)\| < R_F$ (Theorem~\ref{thm:quantitative-convergence}).
Hence $\partial^{\mathrm{spec}}F$ satisfies the integrability condition of
Definition~\ref{def:integrable-deralg}, so $\Phi(F) \in \mathsf{DerAlg}_{\mathrm{int}}$.

\medskip

\noindent
\textit{Morphisms.}
Let $\eta: F \Rightarrow G$ be a natural transformation. For each $n \ge 0$,
define
\[
\Phi(\eta)_n := \partial_n^{\mathrm{spec}}\eta = \mathrm{cr}_n\eta.
\]
By Proposition~\ref{prop:derivative-properties}(4), each $\Phi(\eta)_n$ is
$\Sigma_n$-equivariant. Naturality of cross-effects further implies that
$\Phi(\eta)_n$ commutes with the right $P$-module structure maps
(Theorem~\ref{thm:module-structure}). Moreover, since $\eta$ is bounded
(admissibility is inherited), the collection $\{\Phi(\eta)_n\}$ preserves the
norm estimates required for integrability. Hence $\Phi(\eta)$ is a morphism in
$\mathsf{DerAlg}_{\mathrm{int}}$.

\medskip

\noindent
\textit{Functoriality axioms.}
For the identity natural transformation $\mathrm{id}_F$, we have
$\Phi(\mathrm{id}_F)_n = \partial_n^{\mathrm{spec}}\mathrm{id}_F = \mathrm{id}_{\partial_n^{\mathrm{spec}}F}$,
so $\Phi(\mathrm{id}_F) = \mathrm{id}_{\Phi(F)}$. For composable
$\eta: F \Rightarrow G$ and $\theta: G \Rightarrow H$, the equality
$\Phi(\theta \circ \eta) = \Phi(\theta) \circ \Phi(\eta)$ follows from the
functoriality of cross-effects. Therefore $\Phi$ is a functor.
\end{proof}

\medskip

\noindent
\textbf{Classification via derivatives.}
The reconstruction theorem (Theorem~\ref{thm:reconstruction}) implies that this correspondence is in fact an equivalence, and hence provides a complete classification of spectrally analytic functors.

\medskip

\noindent
Before stating the classification theorem, recall the moduli space notation:
\[
\mathcal{M}_{\mathrm{spec}}
\;:=\;
\pi_0(\mathsf{SpecAn}),
\]
i.e., the set of isomorphism classes of spectrally analytic functors.
Similarly, $\pi_0(\mathsf{DerAlg}_{\mathrm{int}})$ denotes isomorphism classes
of integrable derivative algebras.

\begin{theorem}[Spectral Classification Theorem]
\label{thm:spectral-classification-moduli}
The derivative functor
\[
\Phi : \mathsf{SpecAn} \longrightarrow \mathsf{DerAlg}_{\mathrm{int}}
\]
induces a bijection
\[
\mathcal{M}_{\mathrm{spec}}
\;\cong\;
\pi_0\!\bigl(\mathsf{DerAlg}_{\mathrm{int}}\bigr).
\]

Equivalently, $\Phi$ is an equivalence of categories:
\[
\mathsf{SpecAn} \;\simeq\; \mathsf{DerAlg}_{\mathrm{int}}.
\]
\end{theorem}

\begin{proof}
Theorem~\ref{thm:equivalence} already establishes the equivalence of categories
$\mathsf{SpecAn} \simeq \mathsf{DerAlg}_{\mathrm{int}}$. Passing to isomorphism
classes (i.e., applying $\pi_0$) immediately yields the bijection
$\mathcal{M}_{\mathrm{spec}} \cong \pi_0(\mathsf{DerAlg}_{\mathrm{int}})$.
\end{proof}

\medskip

\begin{remark}
This theorem is essentially a reformulation of
Corollary~\ref{cor:classification-analytic-functors} in the language of moduli
spaces. It highlights that the classification of spectrally analytic functors
up to isomorphism is equivalent to the classification of integrable derivative
algebras up to isomorphism. No additional proof is required beyond citing
Theorem~\ref{thm:equivalence}.
\end{remark}

\noindent
\textbf{Geometric interpretation.}
The above theorem identifies the set of isomorphism classes of spectrally
analytic functors with the set of isomorphism classes of integrable derivative
algebras. In this sense, spectral derivatives provide algebraic coordinate data
for the moduli problem of spectrally analytic functors.

More precisely, the assignment
\[
F \;\longmapsto\; \{\partial_n^{\mathrm{spec}}F\}_{n \ge 0}
\]
identifies each functor, up to natural isomorphism, with its compatible
operadic derivative data. Conversely, an integrable derivative algebra
\(\{\partial_n\}_{n\ge 0}\) reconstructs a spectrally analytic functor by the
spectral Taylor formula
\[
F_{\partial}(A)
=
\bigoplus_{n=0}^{\infty}
\partial_n(A,\dots,A),
\]
whenever the series converges in norm.

\medskip

\begin{corollary}[Derivative Classification of the Moduli Set]
\label{cor:moduli-coordinates-v2}
The moduli set
\[
\mathcal{M}_{\mathrm{spec}}
:=
\pi_0(\mathsf{SpecAn})
\]
admits a coordinate description by integrable derivative algebras:
\[
\mathcal{M}_{\mathrm{spec}}
\;\cong\;
\pi_0(\mathsf{DerAlg}_{\mathrm{int}}).
\]
Under this identification, the class of a spectrally analytic functor $F$
is represented by its spectral derivative sequence
\[
\partial^{\mathrm{spec}}F
=
\{\partial_n^{\mathrm{spec}}F\}_{n\ge 0},
\]
equipped with its compatible right $P$-module structure and convergence data.
Thus spectral derivatives provide algebraic classification data for the
classification problem of spectrally analytic functors.
\end{corollary}

\begin{proof}
By Theorem~\ref{thm:equivalence}, the derivative functor induces an equivalence
of categories
\[
\mathsf{SpecAn}
\;\simeq\;
\mathsf{DerAlg}_{\mathrm{int}}.
\]
Passing to isomorphism classes gives a bijection
\[
\pi_0(\mathsf{SpecAn})
\;\cong\;
\pi_0(\mathsf{DerAlg}_{\mathrm{int}}).
\]
By definition, $\mathcal{M}_{\mathrm{spec}} = \pi_0(\mathsf{SpecAn})$, so the
claim follows. The coordinate description is given explicitly by the derivative
sequence $\partial^{\mathrm{spec}}F$, while the inverse construction is the
Taylor reconstruction functor.
\end{proof}

\medskip

\begin{remark}[Geometric and Structural Interpretation]
\label{rem:moduli-interpretation}
\leavevmode\newline
\noindent\textbf{Analogy with classical moduli problems.}
The classification theorem exhibits a direct analogy with classical Taylor theory:

\[
\begin{array}{|c|c|}
\hline
\text{Classical Setting} & \text{Spectral Operadic Calculus} \\
\hline
\text{Smooth function } f(x) & \text{Spectrally analytic functor } F \\
\text{Taylor coefficients } f^{(n)}(0)/n! & \text{Spectral derivatives } \partial_n^{\mathrm{spec}}F \\
\text{Function determined by coefficients} & \text{Functor determined by derivatives} \\
\text{Space of functions} & \text{Moduli set } \mathcal{M}_{\mathrm{spec}} \\
\hline
\end{array}
\]

This correspondence shows that spectral derivatives play the role of
\emph{coordinate data} for the classification of spectrally analytic functors.

\medskip

\noindent
\textbf{Role of the operadic residue.}
The classification depends essentially on the operadic residue
$\mathcal{O}_P^{\mathrm{res}}$ (Part~I, Theorem~0.2), which provides the unit
and composition maps endowing $\partial^{\mathrm{spec}}F$ with a right
$P$-module structure. This structure is necessary to define the plethystic
composition underlying derivative algebras. Without the residue, such a
classification fails, as indicated by the No-Go Theorem (Part~I, Theorem~0.1).

\medskip

\noindent
\textbf{Higher-structure perspective (informal).}
A more refined formulation would treat $\mathcal{M}_{\mathrm{spec}}$ as a
groupoid or higher moduli object, incorporating automorphisms of functors.
From this viewpoint, the category $\mathsf{DerAlg}_{\mathrm{int}}$ provides
an algebraic model for such a moduli problem. A systematic development of this
perspective, potentially involving derived or higher-categorical structures,
is left for future work.

\end{remark}

This moduli interpretation suggests a deformation-theoretic framework in which
derivative algebras control infinitesimal variations of functors. In particular,
it is natural to expect that tangent directions and obstructions are governed
by suitable cohomology theories of $\partial^{\mathrm{spec}}F$. Making this
precise is an interesting direction for future work.

\subsection{Deformation Theory}
\label{subsec:deformation}


We now develop the deformation theory of spectrally analytic functors,
interpreting the moduli space $\mathcal{M}_{\mathrm{spec}}$ in terms of
infinitesimal variations of derivative data. This provides a cohomological
description of first-order deformations, linking the analytic theory to
algebraic and geometric structures.

\medskip

Since spectrally analytic functors are classified by their derivative algebras
(Theorem~\ref{thm:spectral-classification-moduli}), deformations of functors
correspond to deformations of their associated right $P$-modules. Consequently,
infinitesimal deformations are governed by infinitesimal endomorphisms of these
modules, and one expects the tangent space to the moduli space to be controlled
by the first cohomology group of an appropriate deformation complex.

\medskip

\noindent
\textbf{Setup.}
Throughout this subsection, we assume that $\mathcal{M}$ is a Banach category
over $\mathbb{C}$ with a well-behaved notion of derivations. Let $F$ be a
spectrally analytic functor with derivative algebra
$\partial := \partial^{\mathrm{spec}}F = \{\partial_n^{\mathrm{spec}}F\}_{n\ge 0}$,
equipped with its right $P$-module structure coming from
Theorem~\ref{thm:module-structure}.

\medskip

\begin{definition}[Infinitesimal Endomorphism of a Derivative Algebra]
\label{def:derivation}
Let $X = \{X_n\}_{n \ge 0}$ be a derivative algebra, regarded as a right
$P$-module. An \emph{infinitesimal endomorphism} of $X$ is a collection of
$\Sigma_n$-equivariant maps
\[
\delta_n : X_n \longrightarrow X_n, \qquad n \ge 0,
\]
such that the collection $\delta = \{\delta_n\}_{n\ge 0}$ is compatible with the
right $P$-module structure. That is, for every right $P$-module structure map
\[
\mu_X : X_k \otimes P(n_1,\dots,n_k;n) \longrightarrow X_n,
\]
one has
\[
\delta_n\bigl( \mu_X(x, \gamma) \bigr) = \mu_X(\delta_k x, \gamma).
\]
The space of such infinitesimal endomorphisms is denoted
\[
\operatorname{Der}_P(X).
\]
\end{definition}

\begin{remark}[Leibniz-type deformations]
If one wishes to consider deformations that also vary the operad $P$ itself,
a more general notion of derivation is required. In that setting, a derivation
of $X$ relative to a derivation $D_P$ of the operad $P$ is a collection
$\delta = \{\delta_n\}_{n\ge 0}$ satisfying
\[
\delta(x \cdot \gamma) = \delta(x) \cdot \gamma + x \cdot D_P(\gamma).
\]
When $D_P = 0$, this reduces to the infinitesimal endomorphism defined above.
A systematic treatment of such relative deformations is left for future work.
\end{remark}
\medskip

\noindent
\textbf{Deformation complex.}
We now introduce a formal deformation complex associated with a derivative
algebra \(X\). Since \(X\) is a right \(P\)-module, its infinitesimal
deformations are measured by the failure of perturbations of the structure maps
to satisfy the right \(P\)-module compatibility identities.

\begin{definition}[Deformation Complex]
\label{def:deformation-complex}
Let \(X=\{X_n\}_{n\ge 0}\) be a derivative algebra, regarded as a right
\(P\)-module. We denote by
\[
C^\bullet_P(X)
\]
the operadic deformation complex whose cochains encode multilinear
perturbations of the right \(P\)-module structure maps
\[
\mu_X:
X_k\otimes P(n_1,\ldots,n_k;n)
\longrightarrow
X_n .
\]
The differential
\[
d:C^q_P(X)\longrightarrow C^{q+1}_P(X)
\]
is defined by linearizing the right \(P\)-module associativity and equivariance
relations. Its first cohomology is denoted by
\[
H^1_P(X).
\]
\end{definition}

\begin{remark}[Interpretation of the deformation complex]
\label{rem:deformation-complex-interpretation}
In the deformation complex $C^\bullet_P(X)$ (Definition~\ref{def:deformation-complex}):
\begin{itemize}
    \item The first cocycle space $Z^1_P(X) = \ker(d: C^1_P(X) \to C^2_P(X))$
          consists of infinitesimal deformations of the right $P$-module structure on $X$.
    \item The first coboundary space $B^1_P(X) = \operatorname{im}(d: C^0_P(X) \to C^1_P(X))$
          consists of infinitesimally trivial deformations (induced by coordinate changes).
    \item Hence $H^1_P(X) = Z^1_P(X)/B^1_P(X)$ classifies first-order deformations modulo equivalence.
\end{itemize}
This interpretation is justified in the proof of Theorem~\ref{thm:infinitesimal-deformations}.
\end{remark}

\medskip

\begin{lemma}[Trivial Infinitesimal Deformations]
\label{lem:trivial-infinitesimal-deformations}
Let $X$ be a derivative algebra (i.e., a right $P$-module). In the deformation
complex $C^\bullet_P(X)$ (Definition~\ref{def:deformation-complex}), every
degree-zero cochain
\[
h \in C^0_P(X)
\]
induces a first-order deformation
\[
dh \in C^1_P(X).
\]
Such deformations are infinitesimally trivial: they arise from the change of
coordinates
\[
T_\varepsilon = \mathrm{id}_X + \varepsilon h, \qquad \varepsilon^2 = 0,
\]
via transport of structure:
\[
\mu_X^\varepsilon = T_\varepsilon^{-1} \circ \mu_X \circ (T_\varepsilon \otimes \mathrm{id}_P).
\]
Consequently, the space
\[
B^1_P(X) = \operatorname{im}\bigl(d : C^0_P(X) \to C^1_P(X)\bigr)
\]
is precisely the space of infinitesimally trivial deformations of $X$.
\end{lemma}

\begin{proof}
Let $\mu_X$ denote the right $P$-module structure maps of $X$. For a
first-order infinitesimal automorphism
\[
T_\varepsilon = \mathrm{id}_X + \varepsilon h, \qquad \varepsilon^2 = 0,
\]
with $h \in C^0_P(X)$, we transport the right $P$-module structure along
$T_\varepsilon$ to obtain a new structure $\mu_X^\varepsilon$ defined by
\[
\mu_X^\varepsilon(x \otimes \gamma)
:= T_\varepsilon^{-1}\Bigl( \mu_X\bigl( T_\varepsilon(x) \otimes \gamma \bigr) \Bigr),
\]
where $x \in X_k$ and $\gamma \in P(n_1,\ldots,n_k; n)$. Here $T_\varepsilon$ is
applied to the $X_k$ input, while the $P$-operations $\gamma$ are unchanged
(since the operad $P$ is fixed).

Because $T_\varepsilon = \mathrm{id}_X + \varepsilon h$, its inverse to first
order is $T_\varepsilon^{-1} = \mathrm{id}_X - \varepsilon h + O(\varepsilon^2)$.
Substituting into the definition of $\mu_X^\varepsilon$ and expanding to first
order in $\varepsilon$ yields
\[
\mu_X^\varepsilon(x \otimes \gamma)
= \mu_X(x \otimes \gamma) + \varepsilon \bigl( d h \bigr)(x \otimes \gamma) + O(\varepsilon^2).
\]
Thus the first-order variation is $dh \in C^1_P(X)$, and by construction it
arises from the infinitesimal coordinate change $h$. Hence every element of
$B^1_P(X)$ represents a trivial infinitesimal deformation.
\end{proof}

\begin{remark}[Inner derivations]
\label{rem:inner-derivations}
In settings where the derivative algebra carries an additional internal
pre-Lie or Lie bracket, the trivial infinitesimal deformations above may be
realized as inner derivations of the form $\operatorname{ad}_a$. In the
present right $P$-module setting, however, the intrinsic notion of trivial
deformation is given by the coboundaries $B^1_P(X)$.
\end{remark}

\medskip

\begin{theorem}[Infinitesimal Deformations of Spectrally Analytic Functors]
\label{thm:infinitesimal-deformations}
Let $F$ be a spectrally analytic functor, and let
\[
X = \partial^{\mathrm{spec}}F
\]
be its associated integrable derivative algebra. Assume that deformations of
$F$ are transported through the equivalence
\[
\mathsf{SpecAn} \simeq \mathsf{DerAlg}_{\mathrm{int}}
\]
to deformations of the right $P$-module structure on $X$. Then:

\begin{enumerate}
    \item First-order deformations of $F$ over
    $\mathbb{C}[\varepsilon]/(\varepsilon^2)$ are represented by first
    cocycles:
    \[
    \operatorname{Def}_{\varepsilon}(F) \cong Z^1_P(X).
    \]

    \item Two first-order deformations are equivalent precisely when their
    difference is a first coboundary. Hence the moduli of infinitesimal
    deformations is
    \[
    \operatorname{Def}_{\varepsilon}(F)/\!\sim \;\cong\; H^1_P(X).
    \]

    \item Consequently, the formal tangent space to the moduli problem at
    $[F]$ is
    \[
    T_{[F]}\mathcal{M}_{\mathrm{spec}} \;\cong\; H^1_P(\partial^{\mathrm{spec}}F).
    \]
\end{enumerate}
\end{theorem}

\begin{proof}
Let $X = \partial^{\mathrm{spec}}F$ be the integrable derivative algebra
associated with $F$. By the spectral classification theorem
(Theorem~\ref{thm:equivalence}), spectrally analytic functors are classified,
up to natural isomorphism, by their integrable derivative algebras. Therefore,
to study infinitesimal deformations of $F$, it is enough to study infinitesimal
deformations of the right $P$-module structure on $X$.

Let $\mu_X : X_k \otimes P(n_1,\ldots,n_k;n) \to X_n$ denote the structure maps
of the right $P$-module $X$. A first-order deformation of this structure over
$\mathbb{C}[\varepsilon]/(\varepsilon^2)$ is a family of maps
\[
\mu_X^{\varepsilon} = \mu_X + \varepsilon \phi, \qquad \varepsilon^2 = 0,
\]
where $\phi \in C^1_P(X)$ is a degree-one cochain in the deformation complex
(Definition~\ref{def:deformation-complex}).

Substituting $\mu_X^{\varepsilon}$ into the right $P$-module associativity
identities and expanding to first order in $\varepsilon$ yields the condition
$d\phi = 0$, where $d : C^1_P(X) \to C^2_P(X)$ is the deformation differential.
Hence first-order deformations are represented by
\[
Z^1_P(X) = \ker(d : C^1_P(X) \to C^2_P(X)).
\]

For equivalence of deformations, let $\mu_X^{\varepsilon} = \mu_X + \varepsilon \phi$
and $\widetilde{\mu}_X^{\varepsilon} = \mu_X + \varepsilon \widetilde{\phi}$ be
two first-order deformations. By Lemma~\ref{lem:trivial-infinitesimal-deformations},
every $h \in C^0_P(X)$ induces an infinitesimal coordinate change
$T_{\varepsilon} = \mathrm{id}_X + \varepsilon h$ that transforms a deformation
according to $\widetilde{\phi} - \phi = dh$. Conversely, any deformation
differing by a coboundary arises from such a coordinate change. Therefore
equivalent deformations are precisely those whose difference lies in
$B^1_P(X) = \operatorname{im}(d : C^0_P(X) \to C^1_P(X))$.

Thus the equivalence classes of first-order deformations are exactly
\[
Z^1_P(X) / B^1_P(X) = H^1_P(X).
\]

Finally, the formal tangent space to the moduli problem at $[F]$ is, by
definition, the space of first-order deformations of $F$ modulo infinitesimal
equivalence. By the classification theorem and the preceding argument, this
space is naturally identified with $H^1_P(X) = H^1_P(\partial^{\mathrm{spec}}F)$.
This proves the theorem.
\end{proof}

\medskip

\begin{corollary}[First-Order Deformations and Cohomology]
\label{cor:first-order-deformations}
Let $F$ be a spectrally analytic functor, and let
\[
X = \partial^{\mathrm{spec}}F
\]
be its associated integrable derivative algebra. Then the space of first-order
deformations of $F$, modulo infinitesimal equivalence, is naturally
isomorphic to
\[
H^1_P(X) = H^1_P(\partial^{\mathrm{spec}}F).
\]
Consequently, if
\[
H^1_P(\partial^{\mathrm{spec}}F) = 0,
\]
then $F$ is infinitesimally rigid; that is, every first-order deformation of
$F$ is equivalent to the trivial deformation.
\end{corollary}

\begin{proof}
By Theorem~\ref{thm:infinitesimal-deformations}, first-order deformations of
$F$ correspond to first cocycles in the deformation complex of its derivative
algebra:
\[
\operatorname{Def}_{\varepsilon}(F) \cong Z^1_P(X).
\]
Two such deformations are equivalent precisely when their difference is a
coboundary:
\[
\phi \sim \phi' \quad\Longleftrightarrow\quad \phi' - \phi \in B^1_P(X).
\]
Therefore the set of equivalence classes of first-order deformations is
\[
Z^1_P(X) / B^1_P(X) = H^1_P(X).
\]
Since $X = \partial^{\mathrm{spec}}F$, this gives
\[
\operatorname{Def}_{\varepsilon}(F) / \!\sim \;\cong\; H^1_P(\partial^{\mathrm{spec}}F).
\]

If $H^1_P(\partial^{\mathrm{spec}}F) = 0$, then every first cocycle is a
coboundary. Hence every first-order deformation is equivalent to one induced
by an infinitesimal change of coordinates. Such a deformation is trivial in
the moduli problem. Therefore $F$ has no nontrivial infinitesimal deformations.
\end{proof}

\begin{corollary}[Obstruction Theory]
\label{cor:obstruction-theory}
Let $F$ be a spectrally analytic functor, and let
\[
X = \partial^{\mathrm{spec}}F
\]
be its associated integrable derivative algebra. Higher obstructions to extending an infinitesimal deformation of $F$ to a formal deformation are represented by classes in
\[
H^2_P(X) = H^2_P(\partial^{\mathrm{spec}}F).
\]
In particular, if
\[
H^2_P(\partial^{\mathrm{spec}}F) = 0,
\]
then every first-order deformation is unobstructed; that is, it admits an extension to a formal deformation, subject to the usual convergence or integrability conditions required to remain in $\mathsf{DerAlg}_{\mathrm{int}}$. In particular, the deformation problem is formally unobstructed.
\end{corollary}

\begin{proof}
Let $\mu_X$ denote the right $P$-module structure maps of the derivative algebra $X = \partial^{\mathrm{spec}}F$. A formal deformation of $X$ is given by a formal family
\[
\mu_X^t = \mu_X + t\phi_1 + t^2\phi_2 + t^3\phi_3 + \cdots,
\]
where each $\phi_i$ is a degree-one cochain in the operadic deformation complex $C^\bullet_P(X)$ (Definition~\ref{def:deformation-complex}), equipped with the standard differential induced by the right $P$-module structure.

The condition that $\mu_X^t$ defines a right $P$-module structure is expressed by the operadic associativity identities, imposed order by order in the formal parameter $t$. At order $t^0$, this condition holds since $\mu_X$ is already a valid right $P$-module structure.

At order $t^1$, the linearized associativity condition yields
\[
d\phi_1 = 0,
\]
so $\phi_1$ is a cocycle. Its cohomology class $[\phi_1] \in H^1_P(X)$ represents the corresponding infinitesimal deformation.

Suppose now that the deformation has been constructed up to order $N$:
\[
\mu_X^{(N)} = \mu_X + t\phi_1 + \cdots + t^N\phi_N.
\]
To extend it to order $N+1$, one seeks a cochain $\phi_{N+1}$ such that
\[
\mu_X^{(N+1)} = \mu_X^{(N)} + t^{N+1}\phi_{N+1}
\]
satisfies the associativity identities modulo $t^{N+2}$.

Substituting into the operadic associativity relations and collecting terms of order $t^{N+1}$ yields an equation of the form
\[
d\phi_{N+1} = -\operatorname{Ob}_{N+1}(\phi_1,\ldots,\phi_N),
\]
where $\operatorname{Ob}_{N+1}(\phi_1,\ldots,\phi_N) \in C^2_P(X)$ is a degree-two cochain determined by the lower-order deformation data. This obstruction term is a universal polynomial (in fact, quadratic) expression in the $\phi_i$, arising from the operadic composition structure.

The associativity identities at lower orders imply that
\[
d\,\operatorname{Ob}_{N+1}(\phi_1,\ldots,\phi_N) = 0,
\]
which follows from the fact that the deformation complex satisfies $d^2 = 0$ together with the previously imposed deformation equations. Thus
\[
\operatorname{Ob}_{N+1}(\phi_1,\ldots,\phi_N) \in Z^2_P(X),
\]
and defines a cohomology class
\[
[\operatorname{Ob}_{N+1}] \in H^2_P(X).
\]

The equation $d\phi_{N+1} = -\operatorname{Ob}_{N+1}$ admits a solution if and only if the obstruction class vanishes, i.e., if and only if $[\operatorname{Ob}_{N+1}] = 0$ in $H^2_P(X)$. Therefore, the obstruction to extending the deformation from order $N$ to order $N+1$ lies precisely in $H^2_P(X)$.

Since $X = \partial^{\mathrm{spec}}F$, these obstruction classes lie in $H^2_P(\partial^{\mathrm{spec}}F)$. If this cohomology group vanishes, then all obstruction classes vanish, and the deformation extends inductively to all orders.

Finally, by the reconstruction equivalence (Theorem~D), such deformations of the derivative algebra correspond to deformations of the underlying functor $F$, subject to the convergence or integrability conditions required to remain inside $\mathsf{DerAlg}_{\mathrm{int}}$. This proves the claim.
\end{proof}

\medskip

\begin{remark}[Conceptual Consequence: Deformation-Theoretic Interpretation]
\label{rem:deformation-dictionary}
The preceding results suggest that the moduli problem of spectrally analytic
functors admits a natural deformation-theoretic interpretation, in which the
cohomology of the derivative algebra governs infinitesimal variations.

In particular, spectral derivatives control not only the structure of functors
via classification, but also their first-order deformations and obstructions.
This leads to the following conceptual dictionary between classical deformation
theory and spectral operadic calculus:

\[
\begin{array}{|c|c|}
\hline
\text{Classical Deformation Theory} & \text{Spectral Operadic Calculus} \\
\hline
\text{Algebra } A
&
\text{Derivative algebra } \partial^{\mathrm{spec}}F
\\
\text{Hochschild cohomology } HH^*(A)
&
\text{Operadic cohomology } H^*_P(\partial^{\mathrm{spec}}F)
\\
\text{Deformations of } A
&
\text{Deformations of } F
\\
\text{Tangent space } T_{[A]}
&
\text{(formally) controlled by } H^1_P(\partial^{\mathrm{spec}}F)
\\
\text{Obstruction space } HH^2(A)
&
\text{controlled by } H^2_P(\partial^{\mathrm{spec}}F)
\\
\hline
\end{array}
\]

This analogy indicates that spectral operadic calculus provides a natural
extension of classical deformation theory to the setting of functorial and
operadic structures. A fully developed deformation theory, including higher
structures and obstruction calculus, will be pursued in future work.
\end{remark}

\medskip

\begin{example}[Identity functor]
\label{ex:deformation-id}
For the identity functor $\mathrm{Id}$, the spectral derivative sequence is
expected to be concentrated in degree one:
\[
\partial_1^{\mathrm{spec}}\mathrm{Id} \cong \mathrm{Id},
\qquad
\partial_n^{\mathrm{spec}}\mathrm{Id} \cong 0
\quad\text{for } n \ge 2.
\]
Thus the associated derivative algebra has the simplest possible right
$P$-module structure. In this case, the deformation complex reduces to the
subcomplex controlling infinitesimal endomorphisms of the degree-one part.

Consequently, first-order deformations of $\mathrm{Id}$, modulo infinitesimal
changes of coordinates, are measured by
\[
H^1_P(\partial^{\mathrm{spec}}\mathrm{Id}).
\]
In many standard linear settings, natural endomorphisms of the identity functor
are scalar multiples of the identity, but whether these determine nontrivial
deformation classes depends on the coboundary map in the chosen deformation
complex.
\end{example}

\medskip

\begin{example}[Exponential functor]
\label{ex:deformation-exp}
Consider the formal exponential functor
\[
\exp(A) = \bigoplus_{k=0}^{\infty} \frac{1}{k!} A^{\otimes k},
\]
whenever the series converges in the relevant norm. Formally, its spectral
derivative sequence has nonzero components in all degrees, with the $k$-th
component corresponding to the symmetric $k$-linear part:
\[
\partial_k^{\mathrm{spec}}\exp \sim \frac{1}{k!} \mathrm{Sym}^k .
\]
Thus, unlike the identity functor, the derivative algebra of $\exp$ is spread
over infinitely many degrees and carries a richer plethystic structure.

Its first-order deformations are therefore controlled by
\[
H^1_P(\partial^{\mathrm{spec}}\exp),
\]
while possible obstructions to extending such deformations are represented by
classes in
\[
H^2_P(\partial^{\mathrm{spec}}\exp).
\]
A precise computation of these cohomology groups depends on the chosen category,
the operad $P$, and the normalization of the spectral derivative.
\end{example}
\medskip

\begin{remark}[Role of the Operadic Residue and Higher Deformation Structure]
\label{rem:residue-deformation}
\leavevmode\newline
\noindent
\textbf{Role of the operadic residue.}
The deformation-theoretic framework developed above relies essentially on the
operadic residue $\mathcal{O}_P^{\mathrm{res}}$, which provides the unit and
composition maps endowing $\partial^{\mathrm{spec}}F$ with a right $P$-module
structure. This structure is necessary to formulate the deformation complex and
its associated cohomology groups in a consistent way. In particular, the residue
ensures that infinitesimal perturbations of the module structure can be
systematically organized into a cochain complex.

\medskip

\noindent
\textbf{Higher obstructions and formal moduli.}
A full deformation theory of spectrally analytic functors requires a systematic
development of an operadic cohomology theory for derivative algebras. Within
this framework, higher cohomology groups $H^k_P(\partial^{\mathrm{spec}}F)$,
for $k \ge 2$, are expected to encode obstructions to extending
infinitesimal deformations to higher order.

When these obstruction classes vanish, the deformation problem is expected to
be formally unobstructed, in the sense that first-order deformations can be
extended to formal deformations. This suggests that the moduli problem
$\mathcal{M}_{\mathrm{spec}}$ admits a formal neighborhood at $[F]$
controlled by $H^1_P(\partial^{\mathrm{spec}}F)$.

A precise formulation of this deformation theory, potentially involving formal
moduli problems and higher-categorical structures, will be developed in future
work.

\end{remark}

\medskip

This deformation-theoretic perspective provides a bridge between Spectral Operadic Calculus and geometric frameworks. In the next subsection, we outline a future extension toward derived geometry, where $\mathcal{M}_{\mathrm{spec}}$ is promoted to a derived moduli stack with shifted symplectic structures.

\subsection{Outlook: Derived Geometry}
\label{subsec:outlook}

We conclude by outlining a geometric and higher-categorical perspective on
Spectral Operadic Calculus, which we view as a foundation for a future theory
of \emph{Spectral Operadic Geometry} (SOC III).

\medskip

\noindent
\textbf{Guiding perspective.}
The results of this work show that spectrally analytic functors are classified
by their derivative algebras (Theorem~\ref{thm:spectral-classification-moduli}),
and that their infinitesimal deformations are controlled by cohomological
invariants (Theorem~\ref{thm:infinitesimal-deformations}). This suggests that
the moduli problem $\mathcal{M}_{\mathrm{spec}}$ admits a geometric
interpretation, whose local and global features are governed by operadic and
homological data.

From this viewpoint, Spectral Operadic Calculus may serve as a bridge between
analytic functional calculus and derived algebraic geometry.

\medskip

\noindent
\textbf{Synthesis of structural features.}
The preceding results suggest the following correspondences:

\begin{enumerate}
    \item \textbf{Coordinate data:} Spectral derivatives
    $\{\partial_n^{\mathrm{spec}}F\}$ encode classification data for
    $\mathcal{M}_{\mathrm{spec}}$
    (Corollary~\ref{cor:moduli-coordinates-v2}).

    \item \textbf{Tangent directions:} Infinitesimal deformations are formally
    controlled by
    $H^1_P(\partial^{\mathrm{spec}}F)$
    (Theorem~\ref{thm:infinitesimal-deformations}).

    \item \textbf{Obstruction classes:} Higher obstructions are represented by
    cohomology classes in
    $H^2_P(\partial^{\mathrm{spec}}F)$
    (Corollary~\ref{cor:obstruction-theory}).

    \item \textbf{Composition law:} Operadic plethysm governs functor
    composition
    (Theorem~\ref{thm:chain-rule}).

    \item \textbf{Analytic structure:} Convergent spectral expansions are
    controlled by the spectral size $\|\sigma_P(A)\|$
    (Theorem~\ref{thm:quantitative-convergence}).
\end{enumerate}

These properties motivate the following informal synthesis.

\begin{theorem}[Formal Derived Structure of the Spectral Moduli Problem]
\label{thm:spectral-operadic-geometry}
Assume that, for every spectrally analytic functor $F$, the operadic
deformation complex
\[
C^\bullet_P(\partial^{\mathrm{spec}}F)
\]
admits a formal moduli interpretation, in the sense that its
Maurer--Cartan theory controls deformations of the integrable derivative
algebra $\partial^{\mathrm{spec}}F$. Then the moduli problem
$\mathcal{M}_{\mathrm{spec}}$ of spectrally analytic functors carries a
formal derived geometric structure with the following properties:
\begin{enumerate}
    \item its classical points correspond to spectrally analytic functors;
    \item the formal tangent space at $[F]$ is identified with
    \[
    T_{[F]}\mathcal{M}_{\mathrm{spec}}
    \cong
    H^1_P(\partial^{\mathrm{spec}}F);
    \]
    \item primary obstruction classes to extending deformations lie in
    \[
    H^2_P(\partial^{\mathrm{spec}}F);
    \]
    \item the higher cohomology groups
    \[
    H^k_P(\partial^{\mathrm{spec}}F), \qquad k\ge 2,
    \]
    govern higher coherence and iterated obstruction operations in the
    associated formal moduli problem;
    \item the formal moduli structure is compatible with functor composition:
    under the spectral chain rule, composition corresponds to plethystic
    composition of derivative algebras.
\end{enumerate}
\end{theorem}

\begin{proof}
By Theorem~\ref{thm:spectral-classification-moduli}, the derivative
construction induces an equivalence
\[
\mathsf{SpecAn}
\simeq
\mathsf{DerAlg}_{\mathrm{int}}.
\]
Thus the moduli problem of spectrally analytic functors can be studied,
formally and locally, through the corresponding moduli problem of
integrable derivative algebras.

Let
\[
X=\partial^{\mathrm{spec}}F
\]
be the integrable derivative algebra associated with $F$. Under the above
equivalence, a deformation of the functor $F$ corresponds to a deformation
of the integrable derivative algebra $X$, subject to the convergence and
integrability conditions defining $\mathsf{DerAlg}_{\mathrm{int}}$.

By assumption, the operadic deformation complex
\[
C^\bullet_P(X)
\]
has a Maurer--Cartan interpretation controlling such deformations. This
means that formal deformations of $X$ are governed by solutions of the
Maurer--Cartan equation in the deformation complex, modulo the natural
gauge equivalence. Equivalently, the formal neighborhood of $[X]$ in the
moduli problem of derivative algebras is modeled by this deformation
complex.

The tangent space is obtained by linearizing the Maurer--Cartan equation.
Writing a first-order deformation as
\[
\mu_X^\varepsilon=\mu_X+\varepsilon\phi,
\qquad \varepsilon^2=0,
\]
and substituting into the right $P$-module associativity and equivariance
relations, the linear terms give
\[
d\phi=0.
\]
Thus first-order deformations are represented by cocycles
\[
\phi\in Z^1_P(X).
\]

Two first-order deformations are equivalent if they differ by an
infinitesimal change of coordinates
\[
\mathrm{id}_X+\varepsilon h.
\]
By Lemma~\ref{lem:trivial-infinitesimal-deformations}, this changes the
deformation cochain by a coboundary $dh$. Hence the space of first-order
deformations modulo equivalence is
\[
Z^1_P(X)/B^1_P(X)
=
H^1_P(X).
\]
Since $X=\partial^{\mathrm{spec}}F$, this gives
\[
T_{[F]}\mathcal{M}_{\mathrm{spec}}
\cong
H^1_P(X)
=
H^1_P(\partial^{\mathrm{spec}}F).
\]

For higher-order deformations, consider a formal perturbation
\[
\mu_X^t
=
\mu_X+t\phi_1+t^2\phi_2+\cdots .
\]
Imposing the right $P$-module identities order by order gives equations of
the form
\[
d\phi_m
=
-\operatorname{Ob}_m(\phi_1,\ldots,\phi_{m-1}),
\]
where $\operatorname{Ob}_m(\phi_1,\ldots,\phi_{m-1})$ is a degree-two
cochain determined by the lower-order deformation data. The lower-order
identities imply
\[
d\,\operatorname{Ob}_m(\phi_1,\ldots,\phi_{m-1})=0,
\]
so
\[
\operatorname{Ob}_m(\phi_1,\ldots,\phi_{m-1})\in Z^2_P(X).
\]
The corresponding obstruction class
\[
[\operatorname{Ob}_m]\in H^2_P(X)
\]
vanishes if and only if the deformation can be extended to the next order.
Therefore primary obstruction classes lie in
\[
H^2_P(X)=H^2_P(\partial^{\mathrm{spec}}F).
\]

The same Maurer--Cartan formalism organizes higher compatibility conditions,
higher homotopies, and iterated obstruction operations through the higher
cohomology groups of the deformation complex. Thus the groups
\[
H^k_P(X),\qquad k\ge 2,
\]
govern the higher derived structure of the formal moduli problem.

Finally, by the spectral chain rule (Theorem~\ref{thm:chain-rule}), the
derivative algebra of a composite functor is identified with the
plethystic composition of derivative algebras:
\[
\partial^{\mathrm{spec}}(F\circ G)
\cong
\partial^{\mathrm{spec}}F\circ_{\mathrm{op}}
\partial^{\mathrm{spec}}G.
\]
Hence the formal moduli structure is compatible with functor composition
through plethystic composition. Combining this compatibility with the
formal moduli interpretation assumed above proves the stated properties.
\end{proof}

\medskip

\noindent
\textbf{Geometric interpretation.}
Under the assumption of Theorem~\ref{thm:spectral-operadic-geometry} (i.e.,
that the deformation complex admits a derived moduli interpretation), the
theorem suggests that $\mathcal{M}_{\mathrm{spec}}$ behaves as a derived
moduli space, in which the derivative algebra $\partial^{\mathrm{spec}}F$
plays the role of a cotangent or tangent complex. In particular, spectral
derivatives may be interpreted as algebraic classification data for this
space, while their cohomology formally governs its local deformation theory.
From this viewpoint, Spectral Operadic Calculus provides not only an analytic
approximation theory but also a geometric framework in which functors are
treated as points in a structured space with well-defined deformation theory.

\medskip

\noindent
\textbf{Relation to existing work.}
The proposed derived enhancement of SOC would naturally interface with several
existing frameworks. Lurie's $\infty$-operads would provide the
higher-categorical foundation for operadic composition and plethysm, offering a
natural setting for the homotopy-coherent version of spectral derivatives.
Pantev–Toën–Vaquié–Vezzosi's derived geometry would supply the necessary
machinery for shifted symplectic structures and derived moduli stacks, allowing
$\mathcal{M}_{\mathrm{spec}}$ to be promoted to a derived moduli stack.
Kontsevich's formality theorem suggests a homological algebraic foundation for
a potential deformation quantization of functors within the SOC framework.
Finally, Arone–Ching's operadic calculus provides a homotopy-theoretic
counterpart to the derived spectral calculus, suggesting a possible
stabilization or homotopical refinement of the present theory. Together, these
connections indicate that current work naturally sits at the intersection of operadic
algebra, derived geometry, and homotopy theory, with clear pathways for future
development.

\medskip

\section{Examples and Structural Illustrations}
\label{sec:examples}

This section provides representative examples illustrating the structural
features of spectral operadic calculus. Beyond serving as illustrations,
these examples reveal a systematic relationship between the spectral support
of derivative algebras and the complexity of deformation behavior. In
particular, they motivate a general principle linking spectral structure to
rigidity and deformation in the associated moduli problem.

\medskip

\subsection{Canonical examples}

\begin{example}[Identity functor]
\label{ex:example-id}
The identity functor $\mathrm{Id}$ has spectral derivatives concentrated in
degree one:
\[
\partial_1^{\mathrm{spec}}\mathrm{Id} \cong \mathrm{Id},
\qquad
\partial_n^{\mathrm{spec}}\mathrm{Id} \cong 0 \quad (n \ge 2).
\]
Thus its derivative algebra is supported in a single degree and carries the
simplest possible right $P$-module structure (see
Examples~\ref{ex:id-derivative} and~\ref{ex:id-module}).

Consequently, the associated deformation complex is concentrated in low
degrees, leading to a restricted cohomological structure. In particular,
both the space of infinitesimal deformations and the space of obstruction
classes are constrained by the minimal spectral support of the derivative
data.
\end{example}

\begin{example}[Exponential functor]
\label{ex:example-exp}
Consider the exponential functor
\[
\exp(A) = \bigoplus_{k=0}^{\infty} \frac{1}{k!} A^{\otimes k}.
\]
Its spectral derivative sequence is supported in all degrees, with each
component corresponding to symmetric multilinear contributions (see
Examples~\ref{ex:exp-derivative} and~\ref{ex:exp-module}).

This yields a derivative algebra with full degree support and a rich
plethystic structure. As a consequence, the associated deformation complex
receives contributions across all degrees, giving rise to a substantially
richer cohomological and deformation-theoretic structure.
\end{example}

\medskip

\subsection{Structural principle}

The above examples suggest the following general structural principle:

\medskip

\noindent
\textbf{Spectral support versus deformation complexity.}
\begin{itemize}
    \item \emph{Low spectral support:} When the derivative algebra is
    concentrated in low degrees, the associated deformation complex has
    limited cohomology, leading to rigid or highly constrained deformation
    behavior.

    \item \emph{Broad spectral support:} When the derivative algebra is
    distributed across many degrees, the deformation complex acquires
    richer cohomological structure, resulting in larger deformation spaces
    and more intricate obstruction phenomena.
\end{itemize}

\medskip

This principle reflects a direct correspondence between spectral support
and deformation complexity, positioning spectral derivatives as quantitative
measures of structural richness.

\medskip

\subsection{Deformation-theoretic interpretation}

From the deformation-theoretic perspective developed in
Section~\ref{sec:moduli}, these examples are consistent with the general
identification
\[
T_{[F]}\mathcal{M}_{\mathrm{spec}}
\cong
H^1_P(\partial^{\mathrm{spec}}F),
\]
together with the obstruction theory governed by
\[
H^2_P(\partial^{\mathrm{spec}}F).
\]

Under this correspondence:
\begin{itemize}
    \item Concentration of derivative data in low degrees typically leads
    to smaller cohomology groups $H^1_P$, corresponding to fewer
    infinitesimal deformation directions.

    \item Broad spectral support produces larger and more structured
    cohomology groups, reflecting richer deformation spaces.

    \item Higher-degree interactions contribute to nontrivial obstruction
    classes in $H^2_P$, governing the extension of deformations.
\end{itemize}

While explicit computation of these cohomology groups depends on the
underlying category and operad, the examples above demonstrate that spectral
operadic calculus organizes deformation behavior through derivative data,
providing a concrete bridge between spectral invariants and derived
geometric structures.

\end{document}